\documentclass{article}\usepackage[]{graphicx}\usepackage[]{color}
\usepackage{amsmath}
\usepackage{amsthm}
\usepackage{amsfonts}
\usepackage{amscd}
\usepackage{amssymb}
\usepackage{graphicx}
\usepackage{mathtools}
\usepackage{natbib}
\usepackage{url}
\usepackage{textcomp}

\usepackage{alltt}
\usepackage[utf8]{inputenc}
\usepackage{xr}

\usepackage{algorithm}
\usepackage{algorithmicx}
\usepackage{algpseudocode}

\usepackage{geometry}
\DeclarePairedDelimiter\abs{\lvert}{\rvert}

\newcommand{\nats}{\mathbb{N}}
\newcommand{\R}{\mathbb{R}}
\newcommand{\exreal}{\overline{\R}}

\newcommand{\fatdot}{\,\cdot\,}
\newcommand{\Estar}{E^{\textstyle{*}}}

\newcommand{\tstar}{\theta^{\textstyle{*}}}

\newcommand{\etastar}{\eta^{\textstyle{*}}}
\newcommand{\deltastar}{\delta^{\textstyle{*}}}

\newcommand{\lkn}{\lambda_{k,n}}
\newcommand{\ekn}{e_{k,n}}
\newcommand{\weakto}{\stackrel{d}{\longrightarrow}}

\newcommand{\logonep}[1]{\log(1 + }

\newcommand{\inner}[1]{\langle #1 \rangle}
\newcommand{\set}[1]{\{\, #1 \,\}}

\DeclareMathOperator{\logit}{logit}
\DeclareMathOperator{\pr}{pr}
\DeclareMathOperator{\dom}{dom}
\DeclareMathOperator{\interior}{int}

\DeclareMathOperator{\cl}{cl}
\DeclareMathOperator{\Hyp}{H}
\DeclareMathOperator{\E}{E}
\DeclareMathOperator{\sign}{sign}

\newcommand{\opand}{\mathrel{\rm and}}

\newtheorem{cor}{Corollary}
\newtheorem{lem}{Lemma}
\newtheorem{thm}{Theorem}
\newtheorem{defn}{Definition}

\makeatletter
\def\maxwidth{ %
  \ifdim\Gin@nat@width>\linewidth
    \linewidth
  \else
    \Gin@nat@width
  \fi
}
\makeatother

\definecolor{fgcolor}{rgb}{0.345, 0.345, 0.345}
\newcommand{\hlnum}[1]{\textcolor[rgb]{0.686,0.059,0.569}{#1}}%
\newcommand{\hlstr}[1]{\textcolor[rgb]{0.192,0.494,0.8}{#1}}%
\newcommand{\hlcom}[1]{\textcolor[rgb]{0.678,0.584,0.686}{\textit{#1}}}%
\newcommand{\hlopt}[1]{\textcolor[rgb]{0,0,0}{#1}}%
\newcommand{\hlstd}[1]{\textcolor[rgb]{0.345,0.345,0.345}{#1}}%
\newcommand{\hlkwa}[1]{\textcolor[rgb]{0.161,0.373,0.58}{\textbf{#1}}}%
\newcommand{\hlkwb}[1]{\textcolor[rgb]{0.69,0.353,0.396}{#1}}%
\newcommand{\hlkwc}[1]{\textcolor[rgb]{0.333,0.667,0.333}{#1}}%
\newcommand{\hlkwd}[1]{\textcolor[rgb]{0.737,0.353,0.396}{\textbf{#1}}}%

\usepackage{framed}
\makeatletter
\newenvironment{kframe}{%
 \def\at@end@of@kframe{}%
 \ifinner\ifhmode%
  \def\at@end@of@kframe{\end{minipage}}%
  \begin{minipage}{\columnwidth}%
 \fi\fi%
 \def\FrameCommand##1{\hskip\@totalleftmargin \hskip-\fboxsep
 \colorbox{shadecolor}{##1}\hskip-\fboxsep
     \hskip-\linewidth \hskip-\@totalleftmargin \hskip\columnwidth}%
 \MakeFramed {\advance\hsize-\width
   \@totalleftmargin\z@ \linewidth\hsize
   \@setminipage}}%
 {\par\unskip\endMakeFramed%
 \at@end@of@kframe}
\makeatother

\definecolor{shadecolor}{rgb}{.97, .97, .97}
\definecolor{messagecolor}{rgb}{0, 0, 0}
\definecolor{warningcolor}{rgb}{1, 0, 1}
\definecolor{errorcolor}{rgb}{1, 0, 0}
\newenvironment{knitrout}{}{} 

\allowdisplaybreaks

\title{Computationally efficient likelihood inference in exponential families 
  when the maximum likelihood estimator does not exist}

\author{Daniel J. Eck$^1$ and Charles J. Geyer$^2$ \\[1em]
\normalsize 1. Department of Statistics, University of Illinois Urbana-Champaign \\
\normalsize 2. Department of Statistics, University of Minnesota
}

\begin{document}

\maketitle

\begin{abstract}
In a regular full exponential family, the maximum likelihood estimator (MLE) need not exist in the traditional sense.  However, the MLE may exist in the completion of the exponential family.  Existing algorithms for finding the MLE in the completion solve many linear programs; they are slow in small problems and too slow for large problems. We provide new, fast, and scalable methodology for finding the MLE in the completion of the exponential family.  This methodology is based on conventional maximum likelihood computations which come close, in a sense, to finding the MLE in the completion of the exponential family. These conventional computations construct a likelihood maximizing sequence of canonical parameter values which goes uphill on the likelihood function until they meet a convergence criteria. Nonexistence of the MLE in this context results from a degeneracy of the canonical statistic of the exponential family, the canonical statistic is on the boundary of its support.  There is a correspondance between this boundary and the null eigenvectors of the Fisher information matrix. Convergence of Fisher information along a likelihood maximizing sequence follows from cumulant generating function (CGF) convergence along a likelihood maximizing sequence, conditions for which are given. This allows for the construction of necessarily one-sided confidence intervals for mean value parameters when the MLE exists in the completion. We demonstrate our methodology on three examples in the main text and three additional examples in the Appendix.  We show that when the MLE exists in the completion of the exponential family, our methodology provides statistical inference that is much faster than existing techniques.
\end{abstract}




\label{firstpage}


\section{Introduction}

In a regular full discrete exponential family, the MLE for the canonical parameter does not exist when the observed value of the canonical statistic lies on the boundary of its convex support \citep[Theorem~9.13]{barndorff-nielsen}, but the MLE does exist in a completion of the exponential family. Completions for exponential families have been described by 
\citet[pp.~154--156]{barndorff-nielsen}, 
\citet[pp.~191--201]{brown},
\citet{csiszar,csiszar2008}, and 
\citet[unpublished PhD thesis, Chapter~4]{geyer}. 
The completion that we discuss here will consist of the limit of densities under the the topology of pointwise convergence. The properties of this closure are similar to those in \citet[Chapter~4]{geyer} with conditions similar to those in \cite{brown}.  The issue of when the MLE exists in the conventional sense and what to do when it does not is very important because of the wide use of generalized linear models (GLMs) for discrete data and log-linear models for categorical data.   

Nonexistence of the MLE in these contexts is a widely studied problem.  Advances have been made in establishing necessary and sufficient conditions for existence of the MLE 
\citep{haberman1974analysis, aickin1979existence, albert-anderson, santner1986note, silvapulle1986existence, eriksson2006polyhedral, fienberg2012maximum}, 
the development of an extended or generalized MLE when the traditional MLE does not exist through convex cores of measures \citep{csiszar2001convex, csiszar2003information, csiszar, csiszar2008} and through geometric properties of exponential families and log-linear models
\citep{barndorff-nielsen, brown, geyer, verbeek1992compactification, 
  geyer-gdor, fienberg2012maximum, matuvs2015limiting, 
  wang2019approximating}.  
The issue of nonexistence also arises in exponential families for spatial lattice processes \citep{geyer-interface, geyer-thompson}, spatial point processes \citep{geyer-moller,points}, aster models \citep{geyer2007aster}, aster models with dependency groups \citep{eck2015integrated}, and random graphs 
\citep{ergm-package, ergm-paper, rinaldo-fienberg-zhou, schweinberger}. In every application of these (with the exception of aster models), existing statistical software gives completely invalid results when the MLE does not exist in the traditional sense, and such software either does not check for this problem or does weak checks that can emit both false positives and false negatives. Moreover, even if these checks correctly detect the nonexistence of the MLE, conventional software implements no valid inferential method in this setting. Authoritative textbooks \citep[Section~6.5]{agresti} discuss the issue but provide no solutions. 

\citet{geyer-gdor} developed methodology for constructing hypothesis tests and confidence intervals when the MLE in an exponential family does not exist in the traditional sense. The algorithm in \citet{geyer-gdor}, implemented in the \texttt{rcdd} R package \citep{geyer-rcdd}, are based on doing many linear programs. This algorithm does at most $n$ linear programs, where $n$ is the number of cases of a GLM or the number of cells in a contingency table, in order to determine the existence of the MLE in the traditional sense. Each of these linear programs has $p$ variables, where $p$ is the number of parameters of the model, and up to $n$ inequality constraints. Since linear programming can take time exponential in $n$ when pivoting algorithms are used, and since such algorithms are necessary in computational geometry to get correct answers despite inaccuracy of computer arithmetic (see the warnings about the need to use rational arithmetic in the documentation for R package \texttt{rcdd}), these algorithms can be very slow. Typically, they take several minutes of computer time for toy problems and can take longer than users are willing to wait for real applications.  Previous theoretical discussions \citep{barndorff-nielsen, brown, csiszar, csiszar2008, fienberg2012maximum, matuvs2015limiting, wang2019approximating} of these issues do not provide algorithms, use the notions of faces of convex sets or convex core of measure, are specific to particular discrete exponential families, or are all much harder to compute than the algorithm of \citet{geyer-gdor}. Therefore they provide no explicit direction toward efficient computing. Thus a valid appropriate solution to this issue that is efficiently computable would be very important.  

The MLE in the completion is not only a limit of distributions in the original family but also a distribution in the original family conditioned on the affine hull of a face of the effective domain of the log likelihood supremum function \citep[Theorem~4.3]{geyer}.  Valid statistical inference when the MLE does not exist in the conventional sense requires knowledge of this affine hull.  This affine hull is a support of the canonical statistic under the MLE distribution in the completion.  Hence it is a translate of the null space of the Fisher information matrix, which is the variance-covariance matrix of the canonical statistic for an exponential family.  This affine hull must contain the mean vector of the canonical statistic under the MLE distribution.  Hence knowing the mean vector and variance-covariance matrix of the canonical statistic under the MLE distribution allows us to conduct valid statistical inference, and the MLE will give us good approximations of these quantities.  We will estimate the correct affine hull from the null space of the estimated Fisher information matrix.

In this paper, we develop methodology for constructing hypothesis tests and confidence intervals when the MLE is in the completion.  The MLE in the completion is not only a limit of distributions in the original family but also a distribution in the original family conditioned on the affine hull of a face of the effective domain of the log likelihood supremum function \citep[Theorem~4.3]{geyer}.  Valid statistical inference when the MLE does not exist in the conventional sense requires knowledge of this affine hull.  This affine hull is a support of the canonical statistic under the MLE distribution in the completion.  Hence it is a translate of the null space of the Fisher information matrix, which is the variance-covariance matrix of the canonical statistic for an exponential family.  This affine hull must contain the mean vector of the canonical statistic under the MLE distribution.  Hence knowing the mean vector and variance-covariance matrix of the canonical statistic under the MLE distribution allows us to conduct valid statistical inference, and the MLE will give us good approximations of these quantities.  We will estimate the correct affine hull from the null space of the estimated Fisher information matrix.  In this paper, we make the following contributions: 
\begin{itemize}
\item We provide a computationally efficient solution that has its origins with conventional maximum likelihood computations and avoids the computationally slow linear programming algorithms in \cite{geyer-gdor}. Our computations come close, in a sense, to finding the MLE in the completion of the exponential family.  Informally our approach is to first consider a likelihood maximizing sequence of canonical parameter estimates that goes uphill on the likelihood function until a convergence criteria is satisfied.  At this point, canonical parameter estimates are still infinitely far away from the MLE in the completion, but mean value parameter estimates are close to the MLE in the completion, and the corresponding probability distributions are close in total variation norm to the MLE probability distribution in the completion.  
\item We show that probability distributions evaluated along a likelihood maximizing sequence of canonical parameter vectors are close in the sense of moment generating function convergence (Theorems~\ref{main} and \ref{convex-poly-thm} below) and consequently moments of all orders are also close.  Specifically, under the conditions needed for the closure in \cite{brown}, Theorem~\ref{convex-poly-thm} restores the convergence of moments that were a consequence of the original \cite{barndorff-nielsen} theory which was appropriate for logistic and multinomial regression. The conditions of \cite{brown} hold for infinite state space models such as Poisson regression and other interesting exponential family models. Our convergence of moments results follow from a dominated convergence argument for generalized affine functions (limits of affine functions), a convex geometry argument for generalized affine functions, and a Painlev{\'e}-Kuratowski set convergence argument which implies that null spaces of the Fisher matrix evaluated along likelihood maximizing sequence of canonical parameter vectors converge. 
\item We develop the theoretical foundations of generalized affine functions which are the pointwise limits of sequences affine functions.  Densities of exponential families are affine functions in the data. Thus, generalized affine functions represent limiting densities along sequences of canonical parameter vectors. This theory is relevant for the closure of exponential family under study and it is essential for the convergence of moments along likelihood maximizing sequences results mentioned in the preceding bullet point.
\end{itemize}

In a recent paper, \citet{candes2019phase} studied phase transitions for logistic regression models with Gaussian covariates. They showed that one may be able to determine whether or not the MLE is likely to exist before an analysis is conducted. The configuration of $n$ and $p$ in their setting is such that $n/p \to \kappa$ where $\kappa < 1$. Our methodology has the potential to provide useful and computationally inexpensive statistical inferences in this specific setting, even when phase transition arguments say that the MLE is unlikely to exist apriori. This alleviates the concern made in Section 1.2 of \cite{candes2019phase} that the geometric characterization of exponential families does not tell us when we can expect an MLE to exist and when we cannot. 

Our methodology is implemented in the R package \texttt{glmdr} \citep{glmdr}.  We demonstrate the performance of our methodology on several extensive didactic examples.  These include complete separation in logistic regression and Poisson regression.  Computational efficiency of our methodology is illustrated in Section~\ref{example-bigcategorical}.  Quasi-complete separation examples in logistic regression and Bradley-Terry models are investigated in the Appendix. Detailed R code corresponding to these examples is also provided throughout the Appendix.

\section{Motivating example}
\label{sec:motivating-example}

Consider the case of complete separation in the logistic regression model as a motivating example.  When perfect separation occurs, the canonical statistic is observed to be on the boundary of its convex support.  Suppose that we have one predictor vector $x$ having values 10, 20, 30, 40, 60, 70, 80, 90, and suppose the components of the response vector $y$ are 0, 0, 0, 0, 1, 1, 1, 1.  Then the simple logistic regression model that has linear predictor $\eta = \beta_0 + \beta_1 x$ exhibits failure of the MLE to exist in the traditional sense.  This example is the same as that of \citet[Section~6.5.1]{agresti}. 

\begin{figure}
\begin{center}
\includegraphics[width=3in]{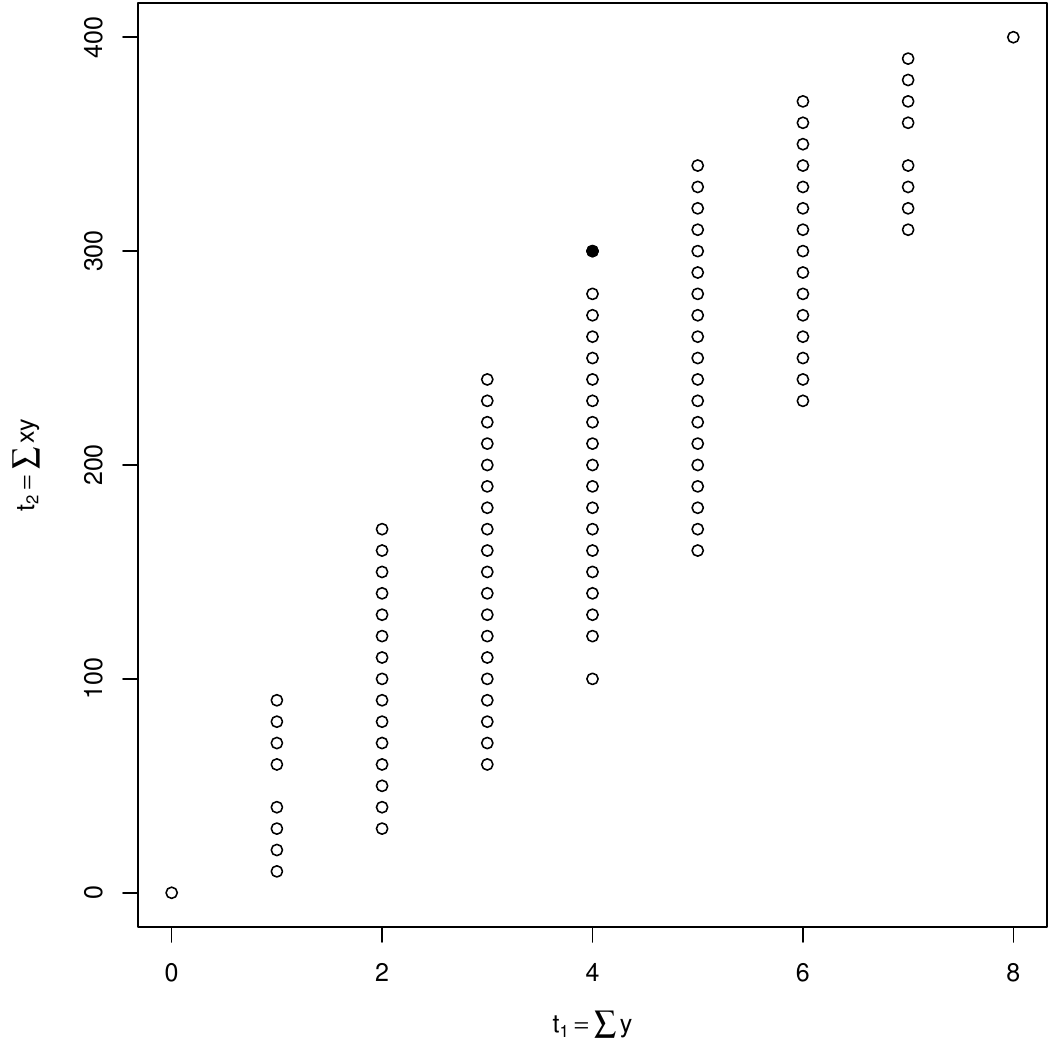}
\includegraphics[width=3in]{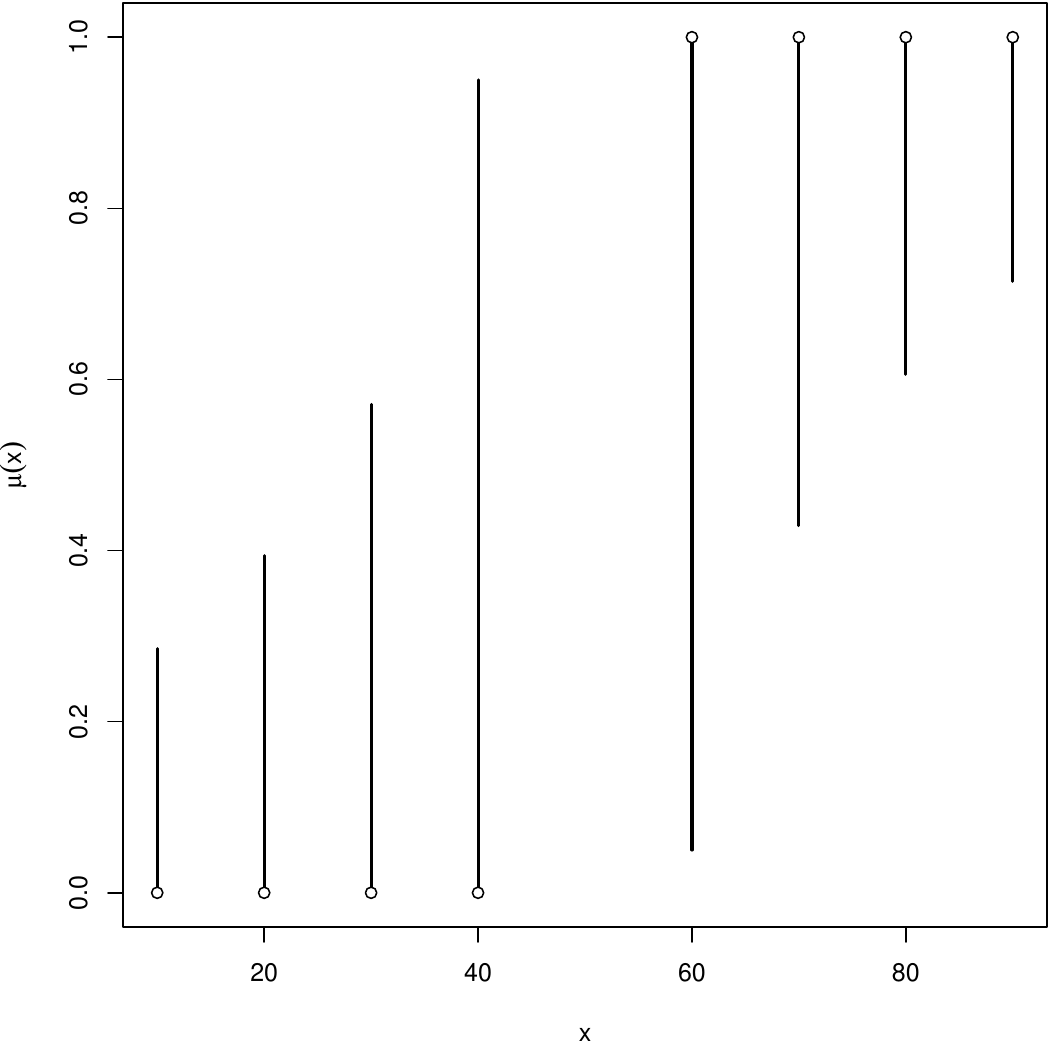}
\end{center}
\caption{{\bf Left panel}: Observed value and support of the submodel canonical statistic vector $M^T y$ for the example of Section~\ref{sec:motivating-example}. Solid dot is the observed value of this statistic. {\bf Right panel}: One-sided 95\% confidence intervals for saturated model mean value parameters.  Bars are the intervals; $\mu(x)$ is the probability of observing response value one when the predictor value is $x$. Solid dots are the observed data.
  }
\label{fig:boundary}
\end{figure}

For an exponential family, the submodel canonical statistic is $M^T y$, where $M$ is the model matrix. The left panel of Figure~\ref{fig:boundary} shows the observed value of the canonical statistic vector and the support (all possible values) of this vector. As is obvious from the figure, the observed value of the canonical statistic is on the boundary of the convex support, in which case the MLE does not exist in the traditional sense.  In this example, the MLE in the completion corresponds to a completely degenerate distribution.  This MLE distribution says no data other than what was observed could have been observed.  But the sample is not the population and estimates are not parameters.  Therefore, this degeneracy is not a problem.  To illustrate the uncertainty of estimation, we show confidence intervals (necessarily one-sided) for the saturated model mean value parameters.  These one-sided confidence intervals are obtained from functionality in the accompanying \texttt{glmdr} package.

The right panel of Figure~\ref{fig:boundary} shows that, as would be expected from so little data, the confidence intervals are very wide. The MLE in the completion says the probability of observing a response equal to one jumps from zero to one somewhere between 40 and 60. The confidence intervals show that we are fairly sure that this probability goes from near zero at $x = 10$ to near one at $x = 90$ but we are very unsure where jumps are if there are any. We discuss how these intervals are constructed in Section~\ref{sec:one-sided}.  The idea is to first find all canonical parameter values such that the probability of observing the realized degenerate data is greater than some testing level $\alpha$. We then map those canonical parameter values to the mean value parameterization.  The degeneracy follows from the estimated Fisher information matrix (for the saturated model canonical parameter vector, also called the linear predictor) at the MLE being singular which it is within the accuracy of computer arithmetic. In this motivating example, the Fisher information matrix is the zero matrix. In this case the MLE of all the saturated model mean value parameters agree with the observed data; they are on the boundary of the set of possible values, either zero or one.

In other examples, such as examples~\ref{example-poisson} and \ref{example-bigcategorical} below, the MLE distribution is only partially but not completely degenerate.  This follows from the estimated Fisher information matrix being singular (to within the accuracy of computer arithmetic) but not the zero matrix.  The MLE distribution constrains some components of the response vector to be equal to their observed values, but not all of them. The remaining unconstrained components can be estimated using traditional methods. This is explained in Sections~\ref{sec:first-characterization}.  

The methodology that we develop is applicable for any discrete regular full exponential family where the MLE does not exist in the traditional sense.  We redo Example 2.3 of \cite{geyer-gdor} in Section~\ref{example-poisson} using the methodology developed here, and we find that our methodology produces the inferences in that paper in a fraction of the time. We also provide an analysis on a big data set (too large for the methods of \citet{geyer-gdor} to run in an acceptable amount of time) to show the (relative) quickness of our implementation.

\section{Standard exponential families}
\label{sec:laptr}

Let $\lambda$ be a positive Borel measure on a finite-dimensional
vector space $E$.
The \emph{log Laplace transform} of $\lambda$ is the
function $c : \Estar \to \exreal$ defined by
\begin{equation} \label{eq:cumfun}
   c(\theta) = \log \int e^{\inner{x, \theta}} \, \lambda(d x),
   \qquad \theta \in \Estar,
\end{equation}
where $\Estar$ is the dual space of $E$,
where $\inner{\fatdot, \fatdot}$ is the canonical
bilinear form placing $E$ and $\Estar$ in duality,
and where $\exreal$ is the extended real number system,
which adds the values $- \infty$ and $+ \infty$ to the real numbers
with the obvious extensions to the arithmetic and topology
\citep[Section~1.E]{rock-wets}.

If one prefers, one can take $E = \Estar = \R^p$ for some $p$, and
define
$$
   \inner{x, \theta} = \sum_{i = 1}^p x_i \theta_i,
   \qquad \text{$x \in \R^p$ and $\theta \in \R^p$},
$$
but the coordinate-free view of vector spaces offers more generality
and more elegance.
Also, as we are about to see, if
$E$ is the sample space of a standard exponential family, then a subset of
$\Estar$ is the canonical parameter space, and the distinction between
$E$ and $\Estar$ helps remind us that we should not consider these two
spaces to be the same space.  

A log Laplace transform is a lower semicontinuous convex function that
nowhere takes the value $- \infty$ (the value $+ \infty$ is allowed and
occurs where the integral in \eqref{eq:cumfun} does not exist)
\citep[Theorem~2.1]{geyer}.   The \emph{effective domain} of an
extended-real-valued convex function $c$ on $\Estar$ is
$$
   \dom c = \set{ \theta \in \Estar : c(\theta) < + \infty }.
$$

For every $\theta \in \dom c$, the function $f_\theta : E \to \R$ defined by 
\begin{equation} \label{eq:densities}
   f_\theta(x) = e^{\inner{x, \theta} - c(\theta)}, \qquad x \in E,
\end{equation}
is a probability density with respect to $\lambda$.
The set
$
   \mathcal{F} = \set{ f_\theta : \theta \in \Theta },
$
where $\Theta$ is any nonempty subset of $\dom c$, is called a \emph{standard exponential family of densities with respect to $\lambda$}. This family is \emph{full} if $\Theta = \dom c$.  We also say $\mathcal{F}$ is the standard exponential family \emph{generated by} $\lambda$ having canonical parameter space $\Theta$, and $\lambda$ is the \emph{generating measure} of $\mathcal{F}$. The log likelihood of this family having densities \eqref{eq:densities} is
\begin{equation} \label{log-like}
  l_x(\theta) = \inner{x, \theta} - c(\theta).
\end{equation}

A general exponential family \citep[Chapter~1]{geyer} is a family
of probability distributions having a sufficient statistic $X$ taking values
in a finite-dimensional vector space $E$ that induces a family of distributions
on $E$ that have a standard exponential family of densities with respect to
some generating measure.  Reduction by sufficiency loses no statistical 
information, so the theory of standard exponential families tells us 
everything about general exponential families \citep[Section~1.2]{geyer}.  

In the context of general exponential families $X$ is called the
\emph{canonical statistic} and $\theta$ the \emph{canonical parameter}
(the terms \emph{natural statistic} and \emph{natural parameter}
are also used).
The set $\Theta$ is the canonical parameter space of the family,
the set $\dom c$ is the canonical parameter space of the full family
having the same generating measure.  A full exponential family is said to be 
\emph{regular} if its canonical parameter space $\dom c$ is an open subset of 
$\Estar$.

\section{Calculating the MLE in the completion}

We first define the completion of the exponential family.

\begin{defn} \label{closure}
Let $\theta_n$, $n = 1,\ldots$, be a sequence of canonical parameter vectors for a standard exponential family having log likelihood \eqref{log-like}. Let $h_n(x) = l_x(\theta_n)$, and suppose that $h_n(x) \to h(x)$ pointwise as $n \to \infty$ where limits $-\infty$ and $+\infty$ are allowed. The limiting functions $h$ form the closure of the exponential family.	
\end{defn}
In the above definition $h_n$ is a sequence of affine functions and the limiting function $h$ is a \emph{generalized affine function}. Generalized affine functions and their properties are defined and discussed in Section~\ref{sec:generalizedaffine}.

\subsection{Assumptions}
\label{sec:Assumptions}

So far everything has been for general exponential families. 
Our implementation requires that the conditions of \citet{brown} hold, and 
those conditions hold for logistic and log-linear models for categorical 
data analysis.  Now, we restrict our attention 
to discrete GLMs.  This, in effect, includes log-linear models
for contingency tables because we can always assume Poisson sampling,
which makes them equivalent to multinomial sampling 
[\citealp[Section~8.6.7]{agresti}; \citealp[Section~3.17]{geyer-gdor}].

The conditions of Brown that are required for our theory to hold are 
from \citet[pp.~193--197]{brown}.  These conditions are:

\begin{itemize}
\item[(i)] The support of the exponential family is a countable set $X$.
\item[(ii)] The exponential family is regular.
\item[(iii)] Every $x \in X$ is contained in the relative interior of an 
  exposed face $F$ of the convex support $K$.
\item[(iv)] The convex support of the measure $\lambda|F$ equals $F$, where $\lambda$ 
  is the generating measure for the exponential family and $\lambda|F$ is the 
  restriction of $\lambda$ to the exposed face $F$.
\end{itemize}

We let $\theta_n$ be a likelihood maximizing sequence of canonical parameter vectors, that is, 
\begin{equation} \label{max-like-seq}
   l_x(\theta_n) \to \sup_{\theta \in \Theta} l_x(\theta),
   \qquad \text{as $n \to \infty$},
\end{equation}
where the log likelihood $l$ is given by \eqref{log-like}, $\Theta$ is the canonical parameter space of the family, and $\sup_{\theta \in \Theta}l_x(\theta) < \infty$. Define $h_n(x) = l_x(\theta_n)$ as in Definition~\ref{closure}.  The limiting density $e^h$ corresponds to the MLE distribution in the completion.  The mathematical properties of generalized affine functions and this completion construction are studied in Section~\ref{sec:math}.

\subsection{The form of the MLE in the completion}
\label{sec:first-characterization}

Suppose we know the \emph{affine support} of the MLE distribution in the completion.  This is the smallest affine set (translate of a vector subspace) that contains the canonical statistic with probability one. Denote the affine support by $A$.  Since the observed value of the canonical statistic is contained in $A$ with probability one, and the canonical statistic for a GLM is $M^T Y$, where $M$ is the model matrix, $Y$ is the response vector, and $y$ its observed value, we have $A = M^T y + V$ for some vector space $V$.

Then the limiting conditional model (LCM) in which the MLE in the completion is found is the original model (OM) conditioned on the event
$$
   M^T (Y - y) \in V, \qquad \text{almost surely}
$$
\citep[Theorem~4.3]{geyer}. Suppose we characterize $V$ as the subspace where a finite set of linear equalities are satisfied
$$
   V = \set{ w \in \R^p : \inner{w, \eta_i} = 0, \ i = 1, \ldots, j }.
$$
Then the LCM is the OM conditioned on the event
$$
   \inner{M^T (Y - y), \eta_i} = \inner{Y - y, M \eta_i} = 0,
   \qquad i = 1, \ldots, j.
$$
From this we see that the vectors $\eta_1$, $\ldots,$ $\eta_j$ span the null space of the Fisher information matrix for the LCM. We collect this in the definition below.

\begin{defn}
	Let $Y$ be the $n$-dimensional vector with iid entries from a discrete regular full exponential family. Let $M \in \R^{n\times p}$ be a known model matrix and let $j \leq p$ be the dimension of the null space of Fisher information. Then the limiting conditional model (LCM) is the original model conditioned on the event
\begin{equation} \label{eq:condition:one}
   \inner{M^T (Y - y), \eta_i} = \inner{Y - y, M \eta_i} = 0,
   \qquad i = 1, \ldots, j,
\end{equation}
where $y$ is the observed value of the response vector $Y$ and $\eta_1$, $\ldots,$ $\eta_j$ spans the null space of the Fisher information matrix.
\end{defn}

The event \eqref{eq:condition:one} fixes some components of the response vector at their observed values and leaves the rest entirely unconstrained. Those components, that are entirely unconstrained are those for which the corresponding components of $M \eta_i$ is zero (or, taking account of the inexactness of computer arithmetic, nearly zero) for all $i = 1$, $\ldots,$ $j$.

Our theory states that the null space of the Fisher information matrix for the LCM is well approximated by the Fisher information matrix for the OM at parameter values that are close to maximizing the likelihood, see Section~\ref{sec:fisher}. The vector subspace spanned by the vectors $\eta_1$, $\ldots,$ $\eta_j$ is called the \emph{constancy space} of the LCM \citep{geyer-gdor}.

\subsection{Calculating one-sided confidence intervals for mean value parameters}
\label{sec:one-sided}

We provide a new method for calculating these one-sided confidence intervals that has not been previously published, but whose concept is found in \citet{geyer-gdor} in the penultimate paragraph of Section~3.16.2. Let $I$ denote the index set of the components of the response vector on which we condition the OM to get the LCM, and let $Y_I$ and $y_I$ denote these components considered as a random vector and as an observed value, respectively.  Let $\theta = M \beta$ denote the saturated model canonical parameter (usually called ``linear predictor'' in GLM theory) with $\beta$ being the submodel canonical parameter vector.  Then endpoints for a $100 (1 - \alpha)\%$ confidence interval for a scalar parameter $g(\beta)$ are
\begin{equation} \label{eq:g-optim}
   \min_{\substack{\gamma \in \Gamma_\text{lim} \\
   \pr_{\hat{\beta} + \gamma}(Y_I = y_I) \ge \alpha}}
   g(\hat{\beta} + \gamma)
   \qquad \text{and} \qquad
   \max_{\substack{\gamma \in \Gamma_\text{lim} \\
   \pr_{\hat{\beta} + \gamma}(Y_I = y_I) \ge \alpha}}
   g(\hat{\beta} + \gamma)
\end{equation}
where $\hat\beta$ is an MLE of the submodel canonical parameter vector in the LCM and $\Gamma_\text{lim}$ is the null space of the Fisher information matrix. At least one of \eqref{eq:g-optim} is at the end of the range of this parameter (otherwise we can use conventional two-sided intervals). Steps for obtaining inferences are outlined in Algorithm~\ref{alg:CIs}. 

For logistic and binomial regression, let $p = \logit^{-1}(\theta)$ denote the mean value parameter vector (here $\logit^{-1}$ operates componentwise). Then, 
$
   \pr_\beta(Y_I = y_I)
   =
   \prod_{i \in I} p_i^{y_i} (1 - p_i)^{n_i - y_i}
$
where the $n_i$ are the binomial sample sizes.  In logistic regression 
we have $n_i = 1$ for all $i$, but in binomial regression 
we have $n_i \geq 1$ for all $i$.  
We could take the confidence interval problem to be
\begin{equation} \label{eq:logistic}
   \text{maximize} \quad p_k,
    \qquad
   \text{subject to} 
     \quad \prod_{i \in I} p_i^{y_i} (1 - p_i)^{n_i - y_i} \ge \alpha, 
\end{equation}
where $p$ is taken to be the function of $\gamma$ described above, and this can be done for any $k \in I$.  The optimization problem in \eqref{eq:logistic} will be more computational stable written as 
\begin{equation} \label{eq:logistic-2}
\begin{split}
   \text{maximize} & \quad \theta_k
   \\
   \text{subject to} & \quad
   \sum_{i \in I} \bigl[ y_i \log(p_i) + (n_i - y_i) \log(1 - p_i) \bigr]
   \ge \log(\alpha),
\end{split}
\end{equation}
since $\log$ can be used to avoid overflow and underflow.  More details are included in Section~\ref{sec:theory-logistic} of the Appendix.

For Poisson sampling, let $\mu = \exp(\theta)$ denote the mean value 
parameter (here $\exp$ operates componentwise like the R function of the 
same name does), then
$
   \pr_\beta(Y_I = y_I) = \exp\left( - \sum_{i \in I} \mu_i \right).
$
We take the confidence interval problem to be
\begin{equation} \label{eq:poisson-ci-problem}
   \text{maximize} 
     \quad \mu_k,
   \qquad
   \text{subject to} 
     \quad - \sum_{i \in I} \mu_i \ge \log(\alpha)
\end{equation}
where $\mu$ is taken to be the function of $\gamma$ described in 
\eqref{eq:g-optim}.  The optimization in \eqref{eq:poisson-ci-problem} can 
be done for any $k \in I$.  The \texttt{inference} function in the R package \texttt{glmdr} determines one-sided confidence intervals for mean value parameters corresponding to response values $y_I$ for logistic and binomial regression as in \eqref{eq:logistic-2} and Poisson regression as in \eqref{eq:poisson-ci-problem}. 


\begin{algorithm}[H]
  \begin{algorithmic}[0] 
\State 1. Declare tolerance $\epsilon$.
\State 2. Fit GLM model and obtain estimated Fisher information matrix.
\State 3. Perform eigenvalue decomposition of estimated Fisher information matrix and assign null eigenvectors as those whose eigenvalues are less than $\epsilon$.
\State 4. Obtain the LCM using estimates of the null eigenvectors obtained in the previous set in \eqref{eq:condition:one} and determine $I$, the index set of the components of the response vector on which we condition the OM to get the LCM.
\State 5. Obtain inference for mean value parameters in the LCM corresponding to the components of $M \eta_i$ which are 0 for all $i = 1,\ldots,j$.
\State 6. Obtain estimate of $\hat\beta$ from the LCM.
\State 7. Obtain one-sided estimates of the mean value parameters as in \eqref{eq:logistic-2}.
\end{algorithmic}
\caption{Inference when canonical statistical is on the boundary of its support}
\label{alg:CIs}
\end{algorithm}

\section{Examples}

\subsection{Complete separation example}
\label{example-complete-separation}

We return to the motivating example of Section 2. Here we see that the Fisher information matrix has only null eigenvectors.  Thus the LCM is completely degenerate at the one point set containing only the observed value of the canonical statistic of this exponential family. One-sided confidence intervals for mean value parameters (success probability considered as a function of the predictor $x$) are computed as in Section~\ref{sec:one-sided}. The right panel of Figure~\ref{fig:boundary} in Section 2 displays these one-sided intervals.  

This example is reproduced in Section F of the Appendix.  The functionality in \texttt{glmdr} was used to calculate the one-sided confidence intervals for mean value parameters (\texttt{inference} function) and determine that the LCM is completely degenerate (\texttt{glmdr} function).

\subsection{Example in Section 2.3 of \cite{geyer-gdor}}
\label{example-poisson}

This example consists of a $2 \times 2 \times \cdots \times 2$ contingency table with seven dimensions hence $2^7 = 128$ cells.  These data now have a permanent location\citep{datasets}. There is one response variable $y$ that gives the cell counts and seven categorical predictors $v_1$, $\ldots$, $v_7$ that specify the cells of the contingency table.  We fit a generalized linear regression model where $y$ is taken to be Poisson distributed.  We consider a model with all three-way interactions included but no higher-order terms. 
The software in the \texttt{glmdr} package reproduces the original analysis, as seen throughout the Appendix. The \texttt{inference} function computed the one-sided confidence intervals for mean value parameters that are on the boundary of their support, in this case equal to zero.  The results are depicted in Table~\ref{ex2-tab1}, this table is the same as Table 2 in \citet{geyer-gdor} and it is reproduced in Section J of the Appendix.

\begin{table}[h!]
\caption{One-sided confidence intervals for cells with MLE equal to zero.}
\begin{center}
\begin{tabular}{ccccccccc}
 $v_1$ & $v_2$ & $v_3$ & $v_4$ & $v_5$ & $v_6$ & $v_7$ & lower & upper \\
 \hline
 0 & 0 & 0 & 0 & 0 & 0 & 0 & 0 & 0.28631 \\
 0 & 0 & 0 & 1 & 0 & 0 & 0 & 0 & 0.14083 \\
 1 & 1 & 0 & 0 & 1 & 0 & 0 & 0 & 0.21997 \\
 1 & 1 & 0 & 1 & 1 & 0 & 0 & 0 & 0.42096 \\
 0 & 0 & 0 & 0 & 0 & 1 & 0 & 0 & 0.08946 \\
 0 & 0 & 0 & 1 & 0 & 1 & 0 & 0 & 0.09377 \\
 1 & 1 & 0 & 0 & 1 & 1 & 0 & 0 & 0.19302 \\
 1 & 1 & 0 & 1 & 1 & 1 & 0 & 0 & 0.28870 \\
 0 & 0 & 0 & 0 & 0 & 0 & 1 & 0 & 0.10631 \\
 0 & 0 & 0 & 1 & 0 & 0 & 1 & 0 & 0.11415 \\
 1 & 1 & 0 & 0 & 1 & 0 & 1 & 0 & 0.09129 \\
 1 & 1 & 0 & 1 & 1 & 0 & 1 & 0 & 0.26461 \\
 0 & 0 & 0 & 0 & 0 & 1 & 1 & 0 & 0.06669 \\
 0 & 0 & 0 & 1 & 0 & 1 & 1 & 0 & 0.15478 \\
 1 & 1 & 0 & 0 & 1 & 1 & 1 & 0 & 0.14097 \\
 1 & 1 & 0 & 1 & 1 & 1 & 1 & 0 & 0.32392
\end{tabular}
\end{center}
\label{ex2-tab1}
\end{table}

The only material difference between our implementation and the linear programming in \cite{geyer-gdor} is computational time.  Our implementation provided one-sided confidence intervals for those responses that are on the boundary of their support in 1.253 seconds, while the functions in the \texttt{rcdd} package take 4.84 seconds of computer time.  This is a big difference for a relatively small amount of data.  Inference for the MLE in the LCM are included in Section K of the Appendix.

\subsection{Big data example}
\label{example-bigcategorical}

This example uses the other dataset at \citep{datasets}. It shows our methods are much faster than the linear programming method of \cite{geyer-gdor}. The functionality in the \texttt{glmdr} determined the LCM and computed one-sided confidence intervals for mean value parameters that are on the boundary of their support in about a minute. The same task using the \texttt{rcdd} package took 
3 days,
4 hours,
0 minutes, and
40.937 seconds.
(This was on \texttt{oak.stat.umn.edu}, which is an Intel(R) Core(TM) i7-6700 CPU @ 3.40GHz.) Both methods yielded the same conclusions.  

This dataset consists of five categorical variables 
with four levels each and a response variable $y$ that is Poisson 
distributed.  A model with all four-way interaction terms is fit to this 
data.  It may seem that the four way interaction model is too large (1024 
data points vs 781 parameters) but $\chi^2$ tests select this model over 
simpler models, see Table~\ref{bdtest}.  

\begin{table}[h]
\caption{Model comparisons for Example 2.  The model m1 is the main-effects 
only model, m2 is the model with all two way interactions, m3 is the model 
with all three way interactions, and m4 is the model with all four way 
interactions.}
\begin{center}
\begin{tabular}{ccccc}
null model & alternative model & df & Deviance & Pr($>\chi^2$) \\
\hline
  m1 & m4 & 765 & 904.8 & 0.00034 \\
  m2 & m4 & 675 & 799.2 & 0.00066 \\
  m3 & m4 & 405 & 534.4 & 0.00002 \\  
\end{tabular}
\end{center}
\label{bdtest}
\end{table}

One-sided 95\% confidence intervals for mean valued parameters whose MLE is equal to zero are displayed in Table~\ref{one-sided-bd}.  The full table is included in Section K.5 of the Appendix.  Some of the intervals in Table~\ref{one-sided-bd} are relatively wide, this represents non-trivial uncertainty about the observed MLE being zero. This example is completely reproduced in Section K of the Appendix.

\begin{table}
\caption{One-sided 95\% confidence intervals for 6 out of 82 mean valued 
  parameters whose MLE is equal to zero.}
\begin{center}
\begin{tabular}{lllllcc}
  \hline
  X1 & X2 & X3 & X4 & X5 & lower bound & upper bound \\ 
  \hline
  a & a & b & a & a & 0 & 0.1695 \\
  a & b & b & a & a & 0 & 0.1354 \\
  a & c & b & a & a & 0 & 0.2292 \\
  a & d & b & a & a & 0 & 2.4616 \\
  d & d & c & a & a & 0 & 0.0002 \\
  a & c & d & a & a & 0 & 0.0133 \\
  \hline
\end{tabular}
\end{center}
\label{one-sided-bd}
\end{table}

\section{Mathematical details}
\label{sec:math}

In this Section we provide the mathematical justification for our inferential 
procedure.  We develop the theory of generalized affine functions \citep{geyer} 
and then show that this theory, combined with conditions for the exponential 
family closure of \citet{brown}, facilitates the convergence of moments of all 
orders along a sequence of maximum likelihood iterates.  We close this Section 
by establishing that our mathematical technique can estimate the correct null 
space of the Fisher information matrix, and this allows for valid statistical 
inference when the MLE does not exist in the conventional sense.

\subsection{Generalized affine functions}
\label{sec:generalizedaffine}

\subsubsection{Characterization on affine spaces}

Exponential families defined on affine spaces instead of vector spaces
are in many ways more elegant \citep[Sections~1.4 and~1.5 and Chapter~4]{geyer}.
To start, a family of densities with respect to a positive Borel
measure on an affine space is a \emph{standard exponential family} if the
log densities are affine functions.  
We complete the exponential family by taking pointwise limits of densities,
allowing $+ \infty$ and $- \infty$ as limits \citep[Chapter~4]{geyer}.

We call these limits \emph{generalized affine functions}.
Real-valued affine functions on an affine space are
functions that are are both convex and concave.
\emph{Generalized affine functions} on an affine space are extended-real-valued
functions that are are both concave and convex \citep[Chapter~4]{geyer}.
(For a definition of extended-real-valued convex functions
see \citet[Chapter~4]{rock-convex}.)

We thus have two characterizations of generalized affine functions: functions
that are both convex and concave and functions that are limits of sequences
of affine functions.  Further characterizations will be given below.

Let $h_n$ denote a sequence of affine functions that are log densities
in a standard exponential family with respect to $\lambda$,
that is, $\int e^{h_n} \, d \lambda = 1$ for all $n$.
Since $e^{h_n} \to e^h$ pointwise if and only if $h_n \to h$ pointwise,
the idea of completing an exponential family naturally leads 
to the study of generalized affine functions.

If $h : E \to \exreal$ is a generalized affine function, we use the notation
\begin{align*}
   h^{- 1}(\R) & = \set{ x \in E : h(x) \in \R }
   \\
   h^{- 1}(\infty) & = \set{ x \in E : h(x) = \infty }
   \\
   h^{- 1}(-\infty) & = \set{ x \in E : h(x) = -\infty }
\end{align*}

\begin{thm} \label{recurse}
An extended-real-valued function $h$ on a finite-dimensional affine space $E$
is generalized affine if and only if one of the following cases holds
\begin{itemize}
\item[\normalfont (a)] $h^{- 1}(\infty) = E$,
\item[\normalfont (b)] $h^{- 1}(-\infty) = E$,
\item[\normalfont (c)] $h^{- 1}(\R) = E$ and $h$ is an affine function, or
\item[\normalfont (d)] there is a hyperplane $H$ such that $h(x) = \infty$
    for all points on one side of $H$, $h(x) = - \infty$ for all points on
    the other side of $H$, and $h$ restricted to $H$ is
    a generalized affine function.
\end{itemize}
\end{thm}
All theorems for which a proof does not follow the theorem statement are proved in Sections A-C in the Appendix. The intention is that this theorem is applied recursively.   If we are in case (d), then the restriction of $h$ to $H$ is another generalized affine function to which the theorem applies. Since a nested sequence of hyperplanes can have length at most the dimension of $E$, the recursion always terminates.

\subsubsection{Topology}

Let $G(E)$ denote the space of generalized affine functions on a 
finite-dimensional affine space $E$ with the topology of pointwise 
convergence.

\begin{thm}
\label{compact-Hausdorff}
$G(E)$ is a compact Hausdorff space.
\end{thm}

\begin{thm} \label{first-countable}
$G(E)$ is a first countable topological space.
\end{thm}

\begin{cor} \label{sequentially-compact}
$G(E)$ is sequentially compact.
\end{cor}
Sequentially compact means every sequence has a (pointwise) convergent subsequence.  That this follows from the two preceding theorems is well known \citep[p.~22, gives a proof]{counterexamples}. The space $G(E)$ is not metrizable, unless $E$ is zero-dimensional \citep[penultimate paragraph of Section~3.3]{geyer}.   So we cannot use $\delta$-$\varepsilon$ arguments, but we can use arguments involving sequences, using sequential compactness.

Let $\lambda$ be a positive Borel measure on $E$, and let $\mathcal{H}$ be a nonempty subset of $G(E)$ such that
\begin{equation} \label{eq:proper}
   \int e^h \, d \lambda = 1, \qquad h \in \mathcal{H}.
\end{equation}
We call $\mathcal{H}$ a \emph{standard generalized exponential family} of log densities with respect to $\lambda$.  Let $\overline{\mathcal{H}}$ denote the closure of $\mathcal{H}$ in $G(E)$.

\begin{thm} \label{closure}
Maximum likelihood estimates always exist in the closure $\overline{\mathcal{H}}$.
\end{thm}
\begin{proof}
Suppose $x$ is the observed value of the canonical statistic.  Then there
exists a sequence $h_n$ in $\mathcal{H}$ such that
$$
   h_n(x) \to \sup_{h \in \mathcal{H}} h(x).
$$
This sequence has a convergent subsequence $h_{n_k} \to h$ in $G(E)$.
This limit $h$ is in $\overline{\mathcal{H}}$ and maximizes the likelihood.
\end{proof}

For full exponential families or even closed convex exponential
families the closure only contains \emph{proper} log probability densities
($h$ that satisfy the equation in \eqref{eq:proper}).
This is shown by \citet[Chapter~2]{geyer} and also by \citet{csiszar}.
We claim that the closure $\overline{\mathcal{H}}$ is the right way to think 
about completion of the exponential families, as it is explicitly constructed to 
facilitate useful statistical inference for practitioners.  
For curved exponential families and for general non-full exponential families,
applying Fatou's lemma to pointwise convergence in $G(E)$ gives only
\begin{equation} \label{eq:improper}
   0 \le \int e^h \, d \lambda \le 1, \qquad h \in \overline{\mathcal{H}}.
\end{equation}
When the integral in \eqref{eq:improper} is strictly less than one
we say $h$ is an \emph{improper} log probability density.
Examples in \citet[Chapter~4]{geyer} show that improper probability
densities cannot be avoided in curved exponential families.

\citet[Theorem~4.3]{geyer} shows that this closure of an exponential family can be thought of as a union of exponential families, so this generalizes the notion in \citet{brown} of the closure as an \emph{aggregate exponential family}.  Thus our method generalizes all previous methods of completing exponential families. Admittedly, this characterization of the completion of an exponential family is very different from any other in its ignoring of parameters.  Only log densities appear.  Unless one wants to call them parameters --- and that conflicts with the usual definition of parameters as real-valued --- parameters just do not appear. So in the next section, we bring parameters back.

\subsubsection{Characterization on vector spaces}

In this section we take sample space $E$ to be vector space (which, of course, 
is also an affine space, so the results of the preceding section continue
to hold).  Recall from Section~\ref{sec:laptr} above, that $\Estar$ denotes
the dual space of $E$, which contains the
canonical parameter space of the exponential family.

\begin{thm} \label{vec-char}
An extended-real-valued function $h$ on a finite-dimensional vector space $E$ 
is generalized affine if and only if there exist finite sequences 
(perhaps of length zero) of vectors $\eta_1$, $\ldots,$ $\eta_j$ in
in $\Estar$ and scalars 
$\delta_1$, $\ldots,$ $\delta_j$ such that
$\eta_1$, $\ldots,$ $\eta_j$ are linearly independent and
$h$ has the following form.  
Define $H_0 = E$ and, inductively, for integers $i$ such that $0 < i \le j$
\begin{align*}
   H_i & = \set{ x \in H_{i - 1} : \inner{x, \eta_i} = \delta_i }
   \\
   C_i^+ & = \set{ x \in H_{i - 1} : \inner{x, \eta_i} > \delta_i }
   \\
   C_i^- & = \set{ x \in H_{i - 1} : \inner{x, \eta_i} < \delta_i }
\end{align*}
all of these sets (if any) being nonempty.  Then $h(x) = + \infty$ whenever 
$x \in C_i^+$ for any $i$, $h(x) = - \infty$ whenever $x \in C_i^-$ for any 
$i$, and $h$ is either affine or constant on $H_j$, where $+ \infty$ and 
$-\infty$ are allowed for constant values.
\end{thm}

The ``if any'' refers to the case where the sequences have length zero, in which case the theorem asserts that $h$ is affine on $E$ or constant on $E$. As we saw in the preceding section, we are interested in likelihood maximizing sequences. Here we represent the likelihood maximizing sequence in the coordinates of the linearly independent $\eta$ vectors that characterize the generalized affine function $h$ according to its Theorem~\ref{vec-char} representation. Let $\theta_n$ be a likelihood maximizing sequence of canonical parameter vectors as in \eqref{max-like-seq}. To make connection with the preceding section, define
$
   h_\theta(x) = l_x(\theta) = \inner{x, \theta} - c(\theta).
$
Then $h_{\theta_n}$ is a sequence of affine functions,
which has a subsequence that converges (in $G(E)$) to
some generalized affine function $h \in \overline{\mathcal{H}}$,
which maximizes the likelihood:
\begin{equation} \label{h-theta}
   h(x) = \sup_{\theta \in \Theta} l_x(\theta).
\end{equation}
The following lemma gives us a better understanding of the convergence $h_{\theta_n} \to h$.

\begin{lem} \label{properties-1}
Suppose that a generalized affine function $h$ on a finite dimensional vector 
space $E$ is finite at at least one point.  Represent $h$ as in 
Theorem~\ref{vec-char}, and extend $\eta_1$, $\ldots,$ $\eta_j$ to be a basis 
$\eta_1$, $\ldots,$ $\eta_p$ for $E^{\textstyle *}$.  Suppose $h_n$ is
a sequence of affine functions converging to $h$ in $G(E)$.
Then there are sequences 
of scalars $a_n$ and $b_{i, n}$ such that
\begin{equation} \label{affine-seq}
  h_n(y) = a_n + \sum_{i = 1}^j b_{i, n} \left(\inner{y, \eta_i} - 
     \delta_i\right) 
  + \sum_{i = j + 1}^p b_{i, n} \inner{y, \eta_i}, \qquad y \in E,
\end{equation}
and, as $n \to \infty$, we have
\begin{enumerate}
\item[\upshape (a)] $b_{i, n} \to \infty$, for $1 \le i \le j$,
\item[\upshape (b)] $b_{i, n} / b_{i - 1, n} \to 0$, for $2 \le i \le j$,
\item[\upshape (c)] $b_{i, n}$ converges, for $i > j$, and
\item[\upshape (d)] $a_n$ converges.
\end{enumerate}
\end{lem}
In \eqref{affine-seq} the first sum is empty when $j = 0$ and the second sum is empty when $j = p$.  Such empty sums are zero by convention. The results given in Lemma~\ref{properties-1} are applicable to generalized affine functions in full generality.  The case of interest to us, however, is when $h_n = h_{\theta_n}$ is the likelihood maximizing sequence constructed above.

\begin{cor} \label{properties-2}  
For data $x$ from a regular full exponential family defined on a vector space 
$E$, suppose $\theta_{n}$ is a likelihood maximizing sequence satisfying 
\eqref{max-like-seq} with log densities $h_n = h_{\theta_n}$ defined by
\eqref{h-theta} converging pointwise
to a generalized affine function $h$.  Characterize $h$ and $h_n$ as in 
Theorem~\ref{vec-char} and Lemma~\ref{properties-1}.
Define 
$\psi_n = \sum_{i=j+1}^p b_{i,n}\inner{x,\eta_i}$.  Then conclusions
\emph{(a)} and \emph{(b)} of Lemma~\ref{properties-1} hold in this setting and 
$$
 \psi_n \to \tstar, \quad \text{as} \; n \to \infty,
$$ 
where $\tstar$ is the MLE of the exponential family conditioned on the
event $H_j$.
\end{cor}
In case $j = p$ the conclusion $\psi_n \to \tstar$ is the trivial zero
converges to zero.  The original exponential family conditioned on the event
$H_j$ is what \citet{geyer-gdor} calls the LCM.

\begin{proof}
The conditions of Lemma~\ref{properties-1} are satisfied by our assumptions so 
all conclusions of Lemma~\ref{properties-1} are satisfied.  As a consequence, 
$\psi_n \to \tstar$ as $n \to \infty$.  The fact that $\tstar$ is the MLE of 
the LCM restricted to $H_j$ follows from our assumption that $\theta_n$ is 
a likelihood maximizing sequence.
\end{proof}

Taken together, Theorem~\ref{vec-char}, Lemma~\ref{properties-1}, and 
Corollary~\ref{properties-2} provide a theory of maximum likelihood estimation
in the completions of exponential families that is the theory of the preceding
section with canonical parameters brought back.

\subsection{Convergence theorems}

\subsubsection{Cumulant generating function convergence}

The CGF of the distribution of the canonical statistic for parameter value $\theta$ is the function $k_\theta$ defined by
\begin{equation} \label{eq:cum-gen-fun}
   k_\theta(t) =
   \log \int e^{\inner{x, t}} f_\theta(x) \, \lambda(d x) =
   c(\theta + t) - c(\theta)
\end{equation}
provided this distribution has a CGF, which it does if and only if
$k_\theta$ is finite on a neighborhood of zero,
that is, if and only if $\theta \in \interior(\dom c)$.
Thus every distribution in a full exponential family has a CGF 
if and only if the family is regular.  Derivatives of 
$k_\theta$ evaluated at zero are the cumulants of the distribution
for $\theta$.  These are the same as derivatives of $c$ evaluated at 
$\theta$.
 
We now show CGF convergence along likelihood maximizing sequences \eqref{max-like-seq}.  This implies convergence in distribution and convergence of moments of all orders. Theorems~\ref{main} and~\ref{convex-poly-thm} in this section say when CGF convergence occurs. Their conditions are somewhat unnatural (especially those of Theorem~\ref{main}). However, the counterexample in Section D of the Appendix shows not only that some conditions are necessary to obtain CGF convergence (it does not occur for all full discrete exponential families) but also that the conditions of Theorem~\ref{main} are sharp, being just what is needed to rule out that example.

The CGF of the distribution having log density that is the
generalized affine function $h$ is defined by 
$$  
  \kappa(t) = \log\int e^{\langle y,t \rangle}e^{h(y)} \, \lambda(d y),
$$
and similarly
$$
   \kappa_{n}(t) = 
   \log\int e^{\langle y,t \rangle}e^{h_n(y)} \, \lambda(d y)  
$$
where we assume $h_n$ are the log densities for a likelihood maximizing
sequence such that $h_n \to h$ pointwise.
The next theorem characterizes when $\kappa_n \to \kappa$ pointwise.

Let $c_A$ denote the log Laplace transform of the restriction of $\lambda$ to
the set $A$, that is,
$$
   c_A(\theta) = \log\int_A e^{\inner{y,\theta}} \, \lambda(d y),
$$
where, as usual, the value of the integral is taken to be $+ \infty$ when
the integral does not exist (a convention that will hold for the rest
of this section).

\begin{thm} \label{main}
Let $E$ be a finite-dimensional vector space of dimension $p$.  For data 
$x \in E$ from a regular full exponential family with natural parameter space 
$\Theta \subseteq \Estar$ and generating measure $\lambda$,  assume that 
every distribution in the family has a cumulant generating function.
Suppose that 
$\theta_{n}$ is a likelihood maximizing sequence satisfying 
\eqref{max-like-seq} with log densities $h_n$ converging pointwise to a 
generalized affine function $h$.  Characterize $h$ as in 
Theorem~\ref{vec-char}.  When $j \geq 2$, and for $i = 1,...,j-1$, define 
\begin{equation} \label{DiF}
  \begin{split}
     D_i &= \{y \in C_i^-: \inner{y,\eta_k} > \delta_k, \; 
       \text{some} \; k > i \}, \\
     F &= E \setminus \cup_{i=1}^{j-1}D_i 
       = \{y: \inner{y,\eta_i} \leq \delta_i, \; 1 \leq i \leq j\}, 
  \end{split}
\end{equation}
and assume that 
\begin{equation} \label{main-bound}
  \sup_{\theta \in \Theta}\sup_{y \in \cup_{i=1}^{j-1}D_i} 
    e^{\inner{y,\theta} - c_{\cup_{i=1}^{j-1}D_i}(\theta)} < \infty
  \quad \text{or} \quad
  \lambda\left(\cup_{i=1}^{j-1}D_i\right) = 0.
\end{equation}
Then $\kappa_n(t)$ converges to $\kappa(t)$ pointwise for all 
$t$ in a neighborhood of 0.
\end{thm}

{\bf Remarks}:

\begin{enumerate}

\item The quantities in \eqref{DiF} and \eqref{main-bound} are technical in nature and are an artifact of the proof technique. Without these conditions, there exists circumstances in which CGF convergence fails to hold. The next remarks elaborate these quantities. The next Theorem shows that \eqref{main-bound} is satisfied under the more intuitive conditions of \cite{brown}.
\item The sets $(H_i,C_i^-,C_i^+)$, $i = 1,\ldots,j$ arise from the characterization of a generalized affine function $h$ given in Theorem~\ref{vec-char} which is a pointwise limit of the densities of log densities $h_n$. The assumption that the exponential family is discrete and full implies that $\int e^{h(y)}\lambda(dy) = 1$ \citep[Theorem 2.7]{geyer} which in turn implies that $\lambda(C_i^+) = 0$ for all $i = 1$,$\ldots$,$j$. We now focus on sets of points $y$ such that $\lambda(\{y\}) > 0$. The first iteration of the recursive structure in the Theorem~\ref{vec-char} characterization of a generalized affine function gives $E = H_1 \cup C_1^- \cup C_1^+$. Now consider a point $y \in C_1^-$, it is possible in a full regular discrete exponential family for $\inner{y,\eta_k} > \delta_k$, $k > 1$ where the pair $(\eta_k,\delta_k)$ form the hyperplane $H_k$. Such points $y$ form the set $D_1$, and the sets $D_i$, $i > 1$, are similarly motivated. Our proof technique requires bounding of the CGF restricted to $\cup_{i=1}^{j-1}D_i$ evaluated along $\theta_n$ by the quantities in \eqref{main-bound}, see \eqref{add-ass-2} in the proof of Theorem~\ref{main}.

\item The conditions in \eqref{main-bound} rule out pathological examples for which CGF convergence does not hold. In Section~\ref{counterexample} of the Appendix we provide an example for which \eqref{main-bound} does not hold and a lack of CGF convergence follows. Moreover, this example demonstrates a lack of convergence of second moments and our approach for statistical inference fails as a result. More general closures of exponential families in \citet{csiszar,csiszar2008} and \citet[unpublished PhD thesis, Chapter~4]{geyer} do not assume  condition \eqref{main-bound} and therefore rule out CGF convergence in full generality. 

\item Discrete exponential families automatically satisfy \eqref{main-bound} when the generating measure satisfies \\
$
  \inf_{y \in \cup_{i=1}^{j-1}D_i}\lambda(\{y\}) > 0.
$  
In this setting, $e^{\inner{y,\theta} - c_{\cup_{i=1}^{j-1}D_i}(\theta)}$ corresponds to the probability mass function for the random variable conditional on the occurrence of $\cup_{i=1}^{j-1}D_i$.  Thus,
\begin{align*}
  &\sup_{\theta \in \Theta}\sup_{y\in\cup_{i=1}^{j-1}D_i}\left( 
    e^{\inner{y,\theta} - c_{\cup_{i=1}^{j-1}D_i}(\theta)}\right) \\
  &\qquad = \sup_{\theta \in \Theta}\sup_{y\in\cup_{i=1}^{j-1}D_i}\left( 
    \frac{e^{\inner{y,\theta}}\lambda(\{y\})}{\lambda(\{y\})
    \sum_{x\in \cup_{i=1}^{j-1}D_i} e^{\inner{x,\theta}}
    \lambda(\{x\})}\right) \\
  &\qquad \leq \sup_{y\in\cup_{i=1}^{j-1}D_i} \left(1/\lambda(\{y\})\right) 
    < \infty.
\end{align*}
Therefore, Theorem~\ref{main} is applicable for the non-existence of the maximum likelihood estimator that may arise in logistic and multinomial regression or any exponential family with finite support. The same is not necessarily so for Poisson regression. The next Theorem provides CGF convergence for Poisson sampling under the regularity conditions of \cite{brown}.
\end{enumerate}

We show in the next theorem that discrete families with convex polyhedral 
support $K$ also satisfy \eqref{main-bound} under additional regularity 
conditions that hold in practical applications.  When $K$ is convex 
polyhedron, we can write 
$
  K = \{y : \inner{y, \alpha_i} \leq a_i, \; \text{for} \; i = 1,...,m \},
$
as in \cite[Theorem 6.46]{rock-wets}.  When the MLE does 
not exist, the data $x \in K$ is on the boundary of $K$.  Denote the active 
set of indices corresponding to the boundary $K$ containing $x$ by
$ 
  I(x) = \{ i : \inner{x, \alpha_i} = a_i \}.
$
In preparation for Theorem~\ref{convex-poly-thm} we define 
the normal cone $N_K(x)$, 
the tangent cone $T_K(x)$, and 
faces of convex sets 
and then state conditions required on $K$.

\begin{defn}
The normal cone of a convex set $K$ in the finite dimensional vector space 
$E$ at a point $x \in K$ is
$$
  N_K(x) = \set{ \eta \in \Estar : \inner{y - x, \eta} \leq 0 
    \; \text{for all} \; y \in K }.
$$
\end{defn}

\begin{defn}
The tangent cone of a convex set $K$ in the finite dimensional vector space 
$E$ at a point $x \in K$ is
$$
  T_K(x) = \cl \set{ s(y - x) : y \in K \; \opand \; s \geq 0 }
$$
where $\cl$ denotes the set closure operation.
\end{defn}

When $K$ is a convex polyhedron, $N_K(x)$ and $T_K(x)$ are both convex 
polyhedron with formulas given in \cite[Theorem 6.46]{rock-wets}.  These 
formulas are 
\begin{align*}
  T_K(x) &= \{y : \inner{y, \alpha_i} \leq 0 \; \text{for all} \; 
    i \in I(x)\}, \\
  N_K(x) &= \{c_1 \alpha_1 + \cdots + c_m \alpha_m : c_i \geq 0 \; 
    \text{for} \; i \in I(x), \; c_i = 0 \; \text{for} \; i \notin I(x) \}.
\end{align*}

\begin{defn}
A \emph{face} of a convex set $K$ is a convex subset $F$ of $K$ such that 
every (closed) line segment in $K$ with a relative interior point in $F$ has 
both endpoints in $F$.  An \emph{exposed face} of $K$ is a face where a certain 
linear function achieves its maximum over $K$ \citep*[p. 162]{rock-convex}.
\end{defn}

The four conditions of Brown, stated in Section~\ref{sec:Assumptions} are 
required for the Theorem to hold. 
Conditions (i) and (ii) are already assumed in Theorem~\ref{main}.  It is now 
shown that discrete exponential families satisfy \eqref{main-bound} under the 
above conditions.


\begin{thm} \label{convex-poly-thm}
Assume the conditions of Theorem~\ref{main} with the omission of 
\eqref{main-bound} when $j \geq 2$.  Let $K$ denote the convex support of the 
exponential family.  Assume that the exponential 
family satisfies the conditions of Brown:
\begin{itemize}
\item[(i)] The support of the exponential family is a countable set $X$.
\item[(ii)] The exponential family is regular.
\item[(iii)] Every $x \in X$ is contained in the relative interior of an 
  exposed face $F$ of the convex support $K$.
\item[(iv)] The convex support of the measure $\lambda|F$ equals $F$, 
  where $\lambda$ is the generating measure for the exponential family.
\end{itemize}  
Then \eqref{main-bound} holds and we have that $\kappa_n(t)$ converges to 
$\kappa(t)$ pointwise for all $t$ in a neighborhood of zero.
\end{thm}


\subsection{Extensions of CGF convergence}
\label{sec:extensions}

Theorems~\ref{main} and \ref{convex-poly-thm} both verify CGF convergence 
along likelihood maximizing sequences \eqref{max-like-seq} on 
neighborhoods of zero.  
The next theorems show that CGF convergence on neighborhoods of zero is enough 
to imply convergence in distribution and of moments of all orders.
Therefore moments of distributions with log densities that are affine 
functions converge along likelihood maximizing sequences \eqref{max-like-seq} 
to those of a limiting distributions whose log density is a generalized 
affine function.

Suppose that $X$ is a random vector in a finite-dimensional vector space $E$ 
having a moment generating function (MGF) $\varphi_X$, then
$
   \varphi_X(t) = \varphi_{\inner{X, t}}(1), 
$
for $t \in \Estar$, regardless of whether the MGF exist or not.  
It follows that the MGF of $\inner{X, t}$ for all $t$ 
determine the MGF of $X$ and vice versa, when these MGF exist.  
More generally, 
\begin{equation} \label{eq:cramer-wold}
   \varphi_{\inner{X, t}}(s) = \varphi_X(s t),
   \qquad t \in E^{\textstyle *} \opand s \in \R.
\end{equation}
This observation applied to characteristic functions rather than MGF is 
called the Cram\'{e}r-Wold theorem.  In that context it is more trivial 
because characteristic functions always exist.  

If $v_1$, $\ldots,$ $v_d$ is a basis for a vector space $E$,
then \citet[Theorem~2 of Section~15]{halmos} states that 
there exists a unique dual basis $w_1$, $\ldots,$ $w_d$
for $E^{\textstyle *}$ that satisfies
\begin{equation} \label{eq:dual-bases}
   \inner{v_i, w_j} = \begin{cases} 1, & i = j \\ 0, & i \neq j \end{cases}
\end{equation}

\begin{thm} \label{MVmgf-1}
If $X$ is a random vector in $E$ having an MGF, then
the random scalar $\inner{X, t}$ has an MGF for all $t \in \Estar$.
Conversely, if $\inner{X, t}$ has an MGF for all $t \in \Estar$,
then $X$ has an MGF.
\end{thm}

\begin{thm} \label{MVmgf-2}
Suppose $X_n$, $n = 1$, $2$, $\ldots$ is a sequence of random vectors,
and suppose their moment generating functions converge pointwise
on a neighborhood $W$ of zero.   Then
\begin{equation} \label{eq:weak}
   X_n \weakto X,
\end{equation}
and $X$ has an MGF $\varphi_X$, and
$\varphi_{X_n}(t) \to \varphi_X(t)$, for $t \in E^{\textstyle *}$.
\end{thm}

\begin{thm} \label{MVmgf-3}
Under the assumptions of Theorem~\ref{MVmgf-2}, suppose
$t_1$, $t_2$, $\ldots,$ $t_k$ are vectors defined on $\Estar$, 
the dual space of $E$.  Then
$
   \prod_{i = 1}^k \inner{X_n, t_i}
$
is uniformly integrable so
$$
   \E\left\{ \prod_{i = 1}^k \inner{X_n, t_i} \right\}
   \to
   \E\left\{ \prod_{i = 1}^k \inner{X, t_i} \right\}.
$$
\end{thm}

The combination of Theorems~\ref{main}-\ref{MVmgf-3} provide a 
methodology for statistical inference along likelihood maximizing sequences 
when the MLE is in the completion of the exponential family.  In particular, 
we have convergence in distribution and convergence of moments of all orders 
along likelihood maximizing sequence.  The limiting distribution in this 
context is a generalized exponential family with density $e^h$ where $h$ is a 
generalized affine function.

\subsection{Convergence of null spaces of Fisher information}
\label{sec:fisher}

Our implementation for finding the MLE in the completion 
relies on finding the null space of Fisher information matrix. 
We first define an appropriate notion of convergence of vector subspaces, 
and then prove that the null spaces corresponding to a sequence of 
semidefinite matrices converge.

\begin{defn}
Painlev\'{e}-Kuratowski set convergence \citep*[Section 4.A]{rock-wets} can be 
defined as follows (\citet{rock-wets} give many equivalent 
characterizations).  If $C_n$ is a sequence of sets in $\R^p$ and $C$ is 
another set in $\R^p$, then we say $C_n \to C$ if
\begin{itemize}
  \item[\normalfont (i)] For every $x \in C$ there exists a subsequence $n_k$ of the 
  natural numbers and there exist $x_{n_k} \in C_{n_k}$ such that 
  $x_{n_k} \to x$.
  \item[\normalfont (ii)] For every sequence $x_n \to x$ in $\R^p$ such that there exists 
  a natural number $N$ such that $x_n \in C_n$ whenever $n \geq N$, we have 
  $x \in C$.
\end{itemize}
\end{defn}

\begin{thm} \label{implem}
Suppose that $A_n \in \R^{p \times p}$ is a sequence of positive semidefinite 
matrices and $A_n \to A$ componentwise.  Fix $\varepsilon > 0$ less than half 
of the least nonzero eigenvalue of $A$ unless $A$ is the zero matrix in which 
case $\varepsilon > 0$ may be chosen arbitrarily.  Let $V_n$ denote the 
subspace spanned by the eigenvectors of $A_n$ corresponding to eigenvalues 
that are less than $\varepsilon$.  
Let $V$ denote the null space of $A$. Then $V_n \to V$
(Painlev\'{e}-Kuratowski).
\end{thm}

In our context, the sequence of matrices $A_n$ in Theorem~\ref{implem} 
correspond to the Fisher information matrices obtained from a discrete 
exponential family whose canonical parameters are substituted for 
those in a likelihood maximizing sequence.

\section*{Supplementary Materials}

The R package \texttt{glmdr} accompanies this submission \citep{glmdr}.

\section{Discussion}

TThe theory of generalized affine functions and the geometry of exponential families allow GLM software to provide fast and scalable maximum likelihood estimation when the observed value of the canonical statistic is on the boundary of its support.  The limiting probability distribution evaluated along the iterates of a likelihood maximizing sequence has log density that is a generalized affine function with structure given by Theorem~\ref{vec-char}.  Cumulant generating functions converge along this sequence of iterates (Theorems~\ref{main} and \ref{convex-poly-thm}), as do estimates of moments of all orders (Theorem~\ref{MVmgf-3}), and so do the null spaces of Fisher information matrices (Theorem \ref{implem}).  These results allow one to obtain the MLE in the completion of the exponential family and to construct one-sided confidence intervals for mean value parameters that are on the boundary of their support.  

The \texttt{glmdr} package computes one-sided confidence intervals for mean value parameters that are on the boundary of their support. Parameter estimation in the LCM is conducted in the traditional manner. The costs of computing the support of a LCM using the \texttt{glmdr} package are minimal compared to the repeated linear programming in the \texttt{rcdd} package.  It is much faster to let optimization software, such as \texttt{glm} in R, simply go uphill on the log likelihood of the exponential family until a convergence tolerance is reached, determine null eigenvectors of the limiting Fisher information matrix, and then compute one-sided confidence intervals than it is to compute the necessary repeated linear programming to achieve the same inferences.  Our examples demonstrate that massive time savings are possible using our methodology.

The chance of observing a canonical statistic on the boundary of its support increases when the dimension of the model increases.  Researchers naturally want to include all possibly relevant covariates in an analysis, and this will often result in the MLE not existing in the conventional sense. Our methods provide a computationally inexpensive solution to this problem.

\section*{Acknowledgements}

We would like to thank Forrest W. Crawford, his comments led to an improved and more interesting version of this paper.

\appendix
\section*{Appendix}

\section{Proofs of main results}

\begin{proof}[Proof of Theorem~\ref{main}]
First consider the case when $j = 0$, the sequences of $\eta$ vectors and 
scalars $\delta$ are both of length zero.  There are no sets $C^{+}$ and 
$C^{-}$ in this setting and $h$ is affine on $E$.  
From Lemma~\ref{properties-1} we have $\psi_n = \theta_n$.  From 
Corollary~\ref{properties-2}, $\theta_n \to \tstar$ as $n\to\infty$.  We 
observe that $c(\theta_n) \to c(\tstar)$ from continuity of the cumulant 
function.  The existence of the MLE in this setting implies that there is a 
neighborhood about 0 denoted by $W$ such that 
$\tstar + W \subset \interior(\dom c)$.  Pick $t \in W$ and observe that 
$c(\theta_n + t) \to c(\tstar + t)$.  Therefore $\kappa_n(t) \to \kappa(t)$ 
when $j = 0$.

Now consider the case when $j=1$.  Define 
$c_1(\theta) = \log\int_{H_1}e^{\inner{y,\theta}}\lambda(dy)$ for all 
$\theta \in \interior(\dom\, c_1)$.  In this scenario we have
\begin{align*}
  \kappa_n(t) &= c\left(\psi_n + t + b_{1,n}\eta_1 \right) 
    - c\left(\psi_n + b_{1,n}\eta_1 \right) \\
  &=  c\left(\psi_n + t + b_{1,n}\eta_j \right) 
    - c\left(\psi_n + b_{1,n}\eta_1 \right) \pm b_{1,n}\delta_1 \\
  &=  \left[c\left(\psi_n + t + b_{1,n}\eta_1 \right) 
    - b_{1,n}\delta_1\right] 
    - \left[c\left(\psi_n + b_{1,n}\eta_1 \right) 
    - b_{1,n}\delta_1 \right].
\end{align*}
From \cite[Theorem 2.2]{geyer}, we know that
\begin{equation} \label{eq:key-too}
    c\left(\tstar + t + s\eta_1 \right) - s\delta_1 
      \to c_1\left(\tstar + t\right), \quad 
    c\left(\tstar + s\eta_1 \right) - s\delta_1 
      \to c_1\left(\tstar\right), 
\end{equation} 
as $s \to \infty$ since $\delta_1 \geq \inner{y, \eta_1}$ for all $y \in H_1$.  
The left hand side of both convergence arrows in \eqref{eq:key-too} are convex 
functions of $\theta$ and the right hand side is a proper convex function.
If $\interior(\dom\, c_1)$ is nonempty, which holds whenever
$\interior(\dom\, c)$ is nonempty, then the convergence in 
\eqref{eq:key-too} is uniform on compact subsets of $\interior(\dom\, c_1)$ 
\citep*[Theorem~7.17]{rock-wets}.  Also \citep*[Theorem~7.14]{rock-wets}, 
uniform convergence on compact sets is the same as continuous convergence.  
Using continuous convergence, we have that both 
\begin{align*}
  c\left(\psi_n + t + b_{1,n}\eta_1 \right) 
    - b_{1,n}\delta_1 &\to c_1\left(\tstar + t\right), \\
  c\left(\psi_n + b_{1,n}\eta_1 \right) 
    - b_{1,n}\delta_1 &\to c_1\left(\tstar\right), 
\end{align*}
where $b_{1,n} \to \infty$ as $n \to \infty$ by Lemma~\ref{properties-1}.  Thus 
\begin{align*}
  \kappa_n(t) &= c(\theta_n + t) - c(\theta_n) 
    \to c_1\left(\tstar + t\right) - c_1\left(\tstar\right) \\
  &= \log\int_{H_1} e^{\inner{y + t,\tstar} - c(\tstar)} \lambda(dy) 
    = \log\int_{H_1} e^{\inner{y,t} + h(y)} \lambda(dy) \\
  &= \log\int e^{\inner{y,t} + h(y)} \lambda(dy) = \kappa(t).
\end{align*}
This concludes the proof when $j=1$.  

For the rest of the proof we will assume 
that $1 < j \leq p$ where dim($E$) = $p$.  Represent the sequence $\theta_n$ in 
coordinate form as
$
   \theta_n = \sum_{i=1}^p b_{i,n}\eta_i,  
$
with scalars 
$b_{i,n}$, $i = 1,...,p$.  For $0 < j < p$, we know that $\psi_n \to \tstar$ as 
$n \to \infty$ from Corollary~\ref{properties-2}.  The existence of the MLE in 
this setting implies that there is a neighborhood about 0, denoted by $W$, 
such that $\tstar + W \subset \interior(\dom c)$.  Pick $t \in W$, fix 
$\varepsilon > 0$, and construct $\varepsilon$-boxes about $\tstar$ and 
$\tstar + t$, denoted by $\mathcal{N}_{0,\varepsilon}(\tstar)$ and 
$\mathcal{N}_{t,\varepsilon}(\tstar)$ respectively, such that both 
$\mathcal{N}_{0,\varepsilon}(\tstar), \mathcal{N}_{t,\varepsilon}(\tstar) 
\subset \interior\left(\dom\, c\right)$.  Let $V_{t,\varepsilon}$ be the set of 
vertices of $\mathcal{N}_{t,\varepsilon}(\tstar)$.  For all $y \in E$ define 
\begin{equation} \label{bounds}
  M_{t,\varepsilon}(y) = 
    \max_{v \in V_{t,\varepsilon}}\{\langle v, y \rangle\},
  \qquad
  \widetilde{M}_{t,\varepsilon}(y) = 
    \min_{v \in V_{t,\varepsilon}}\{\langle v, y \rangle\}.
\end{equation}
From the conclusions of Lemma~\ref{properties-1} and 
Corollary~\ref{properties-2}, we can pick an integer $N$ such that 
$
  \inner{y,\psi_n + t} \leq M_{t,\varepsilon}(y)
$
and $b_{(i+1),n}/b_{i,n} < 1$ for all $n > N$ and $i = 1,...,j-1$.  For all 
$y \in F$, we have 
\begin{equation} \label{add-ass-1}
    \inner{y,\theta_n + t} - \sum_{i=1}^j b_{i,n}\delta_i 
      =\inner{y,\psi_n + t} + \sum_{i=1}^j 
        b_{i,n}\left(\inner{y,\eta_i} - \delta_i\right) 
    \leq M_{t,\varepsilon}(y)
\end{equation} 
for all $n > N$.  The integrability of $e^{M_{t,\varepsilon}(y)}$ and 
$e^{\widetilde{M}_{t,\varepsilon}(y)}$ follows from
\begin{align*}
&\int e^{\widetilde{M}_{t,\varepsilon}(y)}\lambda(dy) 
  \leq \int e^{M_{t,\varepsilon}(y)}\lambda(dy) 
    =\sum_{v \in V_{t,\varepsilon}} \int\limits_{\{y:\; \langle y,v \rangle 
      = M_{t,\varepsilon}(y)\}} e^{\langle y,v \rangle} \lambda(dy) \\
  &\qquad\leq 
    \sum_{v \in V_{t,\varepsilon}} \int e^{\langle y,v \rangle} \lambda(dy) 
  < \infty.
\end{align*}
Therefore,
$$
  \inner{y,\psi_n + t} 
    + \sum_{i=1}^j b_{i,n}\left(\inner{y,\eta_i} - \delta_i\right)
  \to \left\{\begin{array}{cc}
    \inner{y,\tstar + t}, & y \in H_j, \\
    -\infty,              & y \in F \setminus H_j.
  \end{array}\right.
$$
which implies that
\begin{equation} \label{cF}
  c_F(\theta_n + t) - c_F(\theta_n) \to c_{H_j}(\tstar + t) - c_{H_j}(\tstar), 
\end{equation}
by dominated convergence.  To complete the proof, we need to verify that
\begin{equation} \label{last-condition}
  \begin{split}
    c(\theta_n + t) - c(\theta_n) &= c_F(\theta_n + t) - c_F(\theta_n) + c_{\cup_{i=1}^{j-1}D_i}(\theta_n + t) 
      - c_{\cup_{i=1}^{j-1}D_i}(\theta_n) \\
    &\to c_{H_j}(\tstar + t) - c_{H_j}(\tstar).
  \end{split}
\end{equation}
We know that \eqref{last-condition} holds when 
$
  \lambda(\cup_{i=1}^{j-1}D_i) = 0
$ 
in \eqref{main-bound} because of \eqref{cF}.  Now suppose that 
$\lambda(\cup_{i=1}^{j-1}D_i) > 0$.  We have,
\begin{equation} \label{add-lim}
  \inner{y,\psi_n + t} 
    + \sum_{i=1}^j b_{i,n}\left(\inner{y,\eta_i} - \delta_i\right)
  \to -\infty, \qquad y \in \cup_{i=1}^{j-1}D_i,
\end{equation}
and
\begin{equation} \label{add-ass-2}
  \begin{split}
    &\exp\left(c_{\cup_{i=1}^{j-1}D_i}(\theta_n + t) 
        - c_{\cup_{i=1}^{j-1}D_i}(\theta_n)\right) 
      = \int_{\cup_{i=1}^{j-1}D_i} e^{\inner{y,\theta_n + t} 
        - c_{\cup_{i=1}^{j-1}D_i}(\theta_n)} \lambda(dy) \\
      &\qquad \leq \int_{\cup_{i=1}^{j-1}D_i} e^{M_{t,\varepsilon}(y) 
        - \widetilde{M}_{0,\varepsilon}(y) + \inner{y, \theta_n}
        - c_{\cup_{i=1}^{j-1}D_i}(\theta_n)} \lambda(dy) \\
      &\qquad \leq \sup_{y\in \cup_{i=1}^{j-1}D_i} \left(e^{\inner{y, \theta_n}
        - c_{\cup_{i=1}^{j-1}D_i}(\theta_n)}\right) 
        \lambda\left(\cup_{i=1}^{j-1}D_i\right) 
        \int_{\cup_{i=1}^{j-1}D_i} e^{M_{t,\varepsilon}(y) 
          - \widetilde{M}_{0,\varepsilon}(y)}\lambda(dy) \\
      &\qquad \leq \sup_{\theta \in \Theta}\sup_{y\in \cup_{i=1}^{j-1}D_i} 
        \left(e^{\inner{y, \theta}- c_{\cup_{i=1}^{j-1}D_i}(\theta)}\right) 
        \lambda\left(\cup_{i=1}^{j-1}D_i\right) \\
      &\qquad\qquad \times  \int_{\cup_{i=1}^{j-1}D_i} e^{M_{t,\varepsilon}(y) 
        - \widetilde{M}_{0,\varepsilon}(y)}\lambda(dy) \; < \; \infty
  \end{split}
\end{equation}
for all $n > N$ by the assumption given by \eqref{main-bound}.  
The assumption that the exponential family is discrete and full implies that 
$\int e^h(y)\lambda(dy) = 1$ \citep[Theorem~2.7]{geyer}.  This in turn implies 
that $\lambda(C_i^+) = 0$ for all $i = 1,...,j$ which then implies that 
$
  c(\theta) = c_F(\theta) + c_{\cup_{i=1}^{j-1}D_i}(\theta).
$
Putting \eqref{add-ass-1}, \eqref{add-lim}, and \eqref{add-ass-2} together we 
can conclude that \eqref{last-condition} holds as $n\to\infty$ by dominated 
convergence and 
\begin{equation} \label{end}
  \begin{split}
    c_{H_j}(\tstar + t) - c_{H_j}(\tstar) 
    &= \log\int_{H_j} e^{\inner{y,\tstar + t}} \lambda(dy)
      - \log\int_{H_j} e^{\inner{y,\tstar}} \lambda(dy) \\
   &= \log\int e^{\inner{y,t} + h(y)}\lambda(dy),
  \end{split}
\end{equation}
for all $t \in W$, where the last equality is $\kappa(t)$.  This verifies CGF convergence on neighborhoods of 0. 
\end{proof}

\begin{proof}[Proof of Theorem~\ref{convex-poly-thm}]
Represent $h$ as in Theorem~\ref{vec-char}.  Denote the normal cone of the 
convex polyhedron support $K$ at the data $x$ by $N_K(x)$.  We show that a 
sequence of scalars $\deltastar_i$ and a linearly independent set of vectors 
$\etastar_i \in \Estar$ can be chosen so that $\etastar_i \in N_K(x)$, and 
\begin{equation} \label{newH}
  \begin{split}
    H_i &= \{y\in H_{i-1}: \inner{y,\etastar_i} = \deltastar_i\}, \\
    C_i^+ &= \{y\in H_{i-1}: \inner{y,\etastar_i} > \deltastar_i\}, \\
    C_i^- &= \{y\in H_{i-1}: \inner{y,\etastar_i} < \deltastar_i\},
  \end{split}
\end{equation}
for $i = 1,...,j$ where $H_0 = E$ so that \eqref{main-bound} holds.  We will 
prove this by induction with the hypothesis $\Hyp(m)$, $m = 1,...,j$, that 
\eqref{newH} holds for $i \leq m$ where the vectors $\etastar_i \in N_K(x)$ 
$i = 1,...,m$.

We first verify the basis of the induction.  The assumption that the 
exponential family is discrete and full implies that 
$\int e^h(y)\lambda(dy) = 1$ \citep[Theorem~2.7]{geyer}.  This in turn implies 
that $\lambda(C_k^+) = 0$ for all $k = 1,...,j$.  This then implies that 
$
  K \subseteq \{y\in E : \inner{y,\eta_1} \leq \delta_1\} = H_1 \cup C_1^-.
$
Thus $\eta_1 \in N_K(x)$ and the base of the induction holds with 
$\eta_1 = \etastar_1$ and $\delta_1 = \deltastar_1$.  

We now show that $\Hyp(m+1)$ follows from $\Hyp(m)$ for $m = 1,...,j-1$.  
We first establish that $K\cap H_m$ is an exposed face of $K$.  This is needed 
so that \eqref{newH} holds for $i = 1,...,m+1$.  Let $L_K$ be the 
collection of closed line segments with endpoints in $K$.  Arbitrarily choose 
$l \in L_K$ such that an interior point $y \in l$ and 
$y \in K\cap H_m$.  We can write 
$y = \gamma a + (1-\gamma)b$, $0 < \gamma < 1$, where $a$ and $b$ are the 
endpoints of $l$.  Since $a,b \in K$ by construction, we have that 
$\inner{a-x,\etastar_m} \leq 0$ and $\inner{b-x,\etastar_m} \leq 0$ because 
$\etastar_m \in N_K(x)$ by $\Hyp(m)$.  Now,
\begin{align*}
  0 &\geq \inner{a-x,\etastar_m} = \inner{a-y+y-x,\etastar_m} 
    = \inner{a-y,\etastar_m} \\
  &= \inner{a-(\gamma a + (1-\gamma)b), \etastar_m} 
    = (1-\gamma)\inner{a -b, \etastar_m}
\end{align*}
and
\begin{align*}
  0 &\geq \inner{b-x,\etastar_m} = \inner{b-y+y-x,\etastar_m} 
    = \inner{b-y,\etastar_m} \\
  &= \inner{b-(\gamma a + (1-\gamma)b), \etastar_m}
    = -\gamma\inner{a-b, \etastar_m}.  
\end{align*}
Therefore $a,b \in K\cap H_m$ and this verifies that $K\cap H_m$ is a face of 
$K$ since $l$ was chosen arbitrarily.  The function 
$
  y \mapsto \inner{y - x,\etastar_m} - \deltastar_m,
$ 
defined on $K$, is maximized over $K\cap H_m$.  Therefore $K\cap H_m$ is an 
exposed face of $K$ by definition.  The exposed face 
$
  K\cap H_m = K\cap(H_{m+1}\cup C_{m+1}^-)
$ 
since $\lambda(C_{m+1}^+) = 0$ and the convex support of the measure 
$\lambda|H_m$ is $H_m$ by assumption.  Thus, $\eta_{m+1} \in N_{K\cap H_m}(x)$.  

The sets $K$ and $H_m$ are both convex and are therefore regular at every 
point \citep*[Theorem 6.20]{rock-wets}.  We can write 
$
  N_{K\cap H_m}(x) = N_K(x) + N_{H_m}(x)
$ 
since $K$ and $H_m$ are convex sets that cannot be separated where + denotes 
Minkowski addition in this case \citep*[Theorem 6.42]{rock-wets}.  The normal 
cone $N_{H_m}(x)$ has the form
\begin{align*}
  N_{H_m}(x) &= \{\eta \in \Estar : \inner{y-x,\eta} \leq 0 \; 
    \text{for all} \; y \in H_m \} \\
    &= \{\eta \in \Estar : \inner{y-x,\eta} \leq 0 \; \text{for all} \;
      y \in E \; \\ 
    &\qquad  \text{such that} \; \inner{y-x,\eta_i} = 0, \; 
      i = 1,...,m\} \\
    &= \left\{\sum_{i=1}^m a_i\eta_i : \; a_i \in \R, \; i = 1,...,m \right\}.  
\end{align*}
Therefore, we can write 
\begin{equation} \label{etareform}
  \eta_{m+1} = \etastar_{m+1} + \sum_{i=1}^m a_{m,i}\etastar_i
\end{equation}
where $\etastar_{m+1} \in N_K(x)$ and $a_{m,i} \in \R$, $i = 1,...,m$.  For 
$y \in H_{m+1}$, we have that
\begin{align*}
  \inner{y,\etastar_{m+1}} &= \inner{y,\eta_{m+1}} - 
    \sum_{i=1}^ma_{m,i}\inner{y,\eta_{i}} 
  =\delta_{m+1} -  \sum_{i=1}^m a_{m,i}\delta_i.  
\end{align*}
Let $\deltastar_{m+1} = \delta_{m+1} - \sum_{i=1}^m a_{m,i}\delta_i$.  
We can therefore write 
$$
  H_{m+1} = \left\{y \in H_m :  
    \inner{y,\etastar_{m+1}} = \deltastar_{m+1}\right\}
$$
and
\begin{equation} \label{showCplus}
  \begin{split}
    C_{m+1}^+ &= \left\{y \in H_m : \inner{y,\eta_{m+1}} 
      > \delta_{m+1}\right\} \\
    &= \left\{y \in H_m : \inner{y,\etastar_{m+1}} 
      + \sum_{i=1}^m a_{m,i}\delta_i > \delta_{m+1}\right\} \\
    &= \left\{y \in H_m : \inner{y,\etastar_{m+1}} > \delta_{m+1}
      - \sum_{i=1}^m a_{m,i}\delta_i\right\} \\  
    &= \left\{y \in H_m : \inner{y,\etastar_{m+1}} > \deltastar_{m+1}\right\}.
  \end{split}
\end{equation}
A similar argument to that of \eqref{showCplus} verifies that 
$$
  C_i^- = \left\{y \in H_m : \inner{y,\etastar_{m+1}} 
    < \deltastar_{m+1}\right\}.
$$
This confirms that \eqref{newH} holds for $i = 1,...m+1$ and this establishes 
that $\Hyp(m+1)$ follows from $\Hyp(m)$.

Define the sets $D_i$ in \eqref{DiF} with starred quantities replacing the 
unstarred quantities.  Since the vectors 
$\etastar_1$, $\ldots$,$\etastar_j \in N_K(x)$, the sets $K \cap D_i$ are all empty 
for all $i = 1$, $\ldots$, $j-1$.  Thus \eqref{main-bound} holds with  
$
  \lambda\left(\cup_{i=1}^{j-1}D_i\right) = 0.
$
\end{proof}

\begin{proof}[Proof of Theorem~\ref{implem}]
We first consider the case that $A$ is positive definite and $V = \{0\}$.  
We can write $A_n = A + (A_n - A)$ where $(A_n - A)$ is a perturbation 
of $A$ for large $n$.  From Weyl's inequality \citep*{weyl}, we have that all 
eigenvalues of $A_n$ are bounded above zero for large $n$ and 
$V_n = \{0\}$ as a result.  Therefore, $V_n \to V$ as $n\to\infty$ when $A$ 
is positive definite.

Now consider the case that $A$ is not strictly positive definite.  Without loss 
of generality, let $x \in V$ be a unit vector.  
For all $0 < \gamma \leq \varepsilon$, let $V_n(\gamma)$ denote the 
subspace spanned by the eigenvectors of $A_n$ corresponding to eigenvalues 
that are less than $\gamma$.  By construction, $V_n(\gamma) \subseteq V_n$.

From \cite[Example 10.28]{rock-wets},
if $A$ has $k$ zero eigenvalues, then for sufficiently large $N_1$ there are 
exactly $k$ eigenvalues of $A_n$ are less than $\varepsilon$ and $p - k$ 
eigenvalues of $A_n$ greater than $\varepsilon$ for all $n > N_1$.  The same is 
true with respect to $\gamma$ for all $n$ greater than $N_2$.
Thus $j_n(\gamma) = j_n(\varepsilon)$ which implies that $V_n(\gamma) = V_n$ 
for all $n > \max\{N_1,N_2\}$.

We now verify part (i) of Painlev\'{e}-Kuratowski set convergence with respect 
to $V_n(\gamma)$.  Let $N_3$ be such that $x^TA_nx < \gamma^2$ for all 
$n \geq N_3$.  Let $\lkn$ and $\ekn$ be the eigenvalues and eigenvectors of 
$A_n$, with the eigenvalues listed in decreasing orders.  Without loss of 
generality, we assume that the eigenvectors are orthonormal.  Then,
$x = \sum_{k=1}^p(x^T\ekn)\ekn$,
$1 = \|x\|^2 = \sum_{k=1}^p (x^T\ekn)^2$,
and $x^TA_nx = \sum_{k=1}^p \lkn(x^T\ekn)^2$.
There have to be eigenvectors $\ekn$ such that $x^T\ekn \geq 1/\sqrt{p}$ with 
corresponding eigenvalues $\lkn$ that are very small since 
$\lkn(x^T\ekn)^2 < \gamma$.  But conversely, any eigenvalues $\lkn$ such 
that $\lkn \geq \gamma$ must have 
$$
  \lkn(x^T\ekn)^2 < \gamma^2 
    \implies (x^T\ekn)^2 < \gamma^2/\lkn \leq \gamma.
$$
Define $j_n(\gamma) = |\{\lkn:\lkn \leq \gamma\}|$ and 
$
  x_n = \sum_{k=p-j_n(\gamma)+1}^p (x^T\ekn)\ekn
$
where $x_n \in V_n(\gamma)$ by construction.  Now,
\begin{align*}
  \|x - x_n\| &= \|\sum_{k=1}^p (x^T\ekn)\ekn 
    - \sum_{k=p-j_n(\gamma)+1}^p (x^T\ekn)\ekn \| \\
  &= \| \sum_{k=1}^{p-j_n(\gamma)}(x^T\ekn)\ekn \| 
  \leq \sum_{k=1}^{p-j_n(\gamma)}|x^T\ekn| 
  \leq p\sqrt{\gamma}
\end{align*}
for all $n \geq N_3$.  Therefore, for every $x \in V$, there exists a 
sequence $x_n \in V_n(\gamma) \subseteq V_n$ such that $x_n \to x$ since this 
argument holds for all $0 < \gamma \leq \varepsilon$.  This establishes 
part (i) of Painlev\'{e}-Kuratowski set convergence.  

We now show part (ii) of Painlev\'{e}-Kuratowski set convergence.  
Suppose that $x_n \to x \in \R^p$ and there exists a natural number 
$N_4$ such that $x_n \in V_n(\gamma)$ whenever $n \geq N_4$, and we will 
establish that $x \in V$.  From hypothesis, we have that 
$x_n^TA_nx_n \to x^TAx$.  Without loss of generality, we assume that $x$ is a 
unit vector and that $|x_n^TA_nx_n - x^TAx| \leq \gamma$ for all $n \geq N_5$.  
From the assumption that $x_n \in V_n(\gamma)$ we have 
\begin{equation} \label{deriv}
  \begin{split}
    x_n^TA_nx_n &= \sum_{k=1}^p \lkn(x_n^T\ekn)^2 
      = \sum_{k=p-j_n(\gamma)+1}^p \lkn(x_n^T\ekn)^2 
      \leq \gamma
  \end{split}  
\end{equation}
for all $n \geq N_4$.  The reverse triangle inequality gives
$$
  ||x_n^TA_nx_n| - |x^TAx|| \leq |x_n^TA_nx_n - x^TAx| \leq \gamma
$$
and \eqref{deriv} implies 
$
  |x^TAx| \leq 2\gamma
$
for all $n \geq \max\{N_4,N_5\}$.  Since this argument holds for all 
$0 < \gamma < \varepsilon$, we have that $x \in V$.  This establishes 
part (ii) of Painlev\'{e}-Kuratowski convergence with respect to 
$V_n(\gamma)$.  Thus $V_n \to V$. 
\end{proof}




\section{Proofs of the properties of generalized affine functions}

We first prove Theorem~\ref{compact-Hausdorff}.

\begin{proof}
Let $F(E)$ denote the space of all functions $E \to \exreal$ with the
topology of pointwise convergence.   This makes $F(E) = \exreal^E$,
an infinite product.   Then $F(E)$ is compact by Tychonoff's
theorem.   We now show that $G(E)$ is closed in $F(E)$ hence compact.

Let $g$ be any point in the closure of $G(E)$.   Then there is a net
$\{g_\alpha\}$ in $G(E)$ that converges to $g$.   For any $x$ and $y$ in
$E$ such that $g(x) < \infty$ and $g(y) < \infty$ and any $t \in (0, 1)$,
write $z = x + t (y - x)$.

Then
$$
   g_\alpha(z) \le (1 - t) g_\alpha(x) + t g_\alpha(y)
$$
whenever the right hand side makes sense (is not $\infty - \infty$), which
happens eventually, since $g_\alpha(x)$ and $g_\alpha(y)$ both converge to
limits that are not $\infty$.   Hence
$$
   g(z) \le \lambda g(x) + ( 1 - \lambda ) g(y)
$$
and $g$ is convex.   By symmetry it is also concave and hence is
generalized affine.   Thus $G(E)$ contains its closure and is closed.

$F(E)$ is Hausdorff because the product of Hausdorff spaces is Hausdorff.
$G(E)$ is Hausdorff because subspaces of Hausdorff spaces are Hausdorff.
\end{proof}

In order to prove Theorem~\ref{recurse}, an intermediate 
Theorem is first stated and its proof is provided.

\begin{thm}\label{th:non-recursive}
\sloppy
An extended-real-valued function $h$ on a finite-dimensional affine space $E$
is generalized affine if and only if
$h^{- 1}(\infty)$ and $h^{- 1}(-\infty)$ are convex sets,
$h^{- 1}(\R)$ is an affine set, and $h$ is affine on $h^{- 1}(\R)$.
\end{thm}
\begin{proof}
To simplify notation, define
\begin{subequations}
\begin{align}
   A & = h^{- 1}(\R)
   \label{eq:A}
   \\
   B & = h^{- 1}(\infty)
   \label{eq:B}
   \\
   C & = h^{- 1}(-\infty)
   \label{eq:C}
\end{align}
\end{subequations}
First assume $h$ is generalized affine.
Then $C$ is convex because $h$ is convex,
and $B$ is convex because $h$ is concave.
For any two distinct points $x, y \in A$ and any $s \in \R$,
The points $x$, $y$, and $z = x + s (y - x)$ lie on a straight line.
The convexity and concavity inequalities together imply
$$
   h\bigl(x + s (y - x)\bigr)
   =
   (1 - s) h(x) + s h(y).
$$
It follows that $A$ is an affine set and $h$ restricted to $A$
is an affine function.

Conversely, assume $B$ and $C$ are convex sets,
$A$ is an affine set, and $h$ is affine on $A$.
We must show that $h$ is convex and concave.
We just prove convexity because the other proof just the same proof
applied to $- h$.  So consider two distinct points $x, y \in A \cup C$ 
and $0 < t < 1$ (the convexity inequality is vacuous when either of $x$ or 
$y$ is in $B$).  Write $z = x + t (y - x)$.

If $x$ and $y$ are both in $A$, then $A$ being an affine set 
implies $z \in A$ and the convexity inequality involving $x$, $y$, and $z$ 
follows from $h$ being affine on $A$.  
If $x$ and $y$ are both in $C$, then $C$ being a convex set 
implies $z \in C$ and the convexity inequality involving $x$, $y$, and $z$ 
follows from $h(z) = - \infty$.  

The only case remaining is $x \in A$ and $y \in C$.
In this case, there can be no other point on the line determined by $x$ and
$y$ that is in $A$, because $A$ is an affine set.
Hence all the points in this line on one side of $x$ must be in $B$ and
all the points on the other side must be in $C$.   Thus $z \in C$,
and the convexity inequality involving $x$, $y$, and $z$
follows from $h(z) = - \infty$.
\end{proof}

We now provide the proof of Theorem~\ref{recurse}.

\begin{proof}
Again we use the notation in \eqref{eq:A}, \eqref{eq:B}, and \eqref{eq:C}.
First we show that all four cases define generalized affine functions.
The first three cases obviously satisfy the conditions
of Theorem~\ref{th:non-recursive}.

In case (d),
we just prove convexity because the other proof just the same proof
applied to $- h$.

If $x$ and $y$ are both in $H$, then $h$ being generalized affine on $H$
implies the convexity inequality for $x$ and $y$ and any point between them.
If $x$ and $y$ are both in $C$ and not both in $H$, say $x \notin H$,
then any point $z$ between $x$ and $y$ is also not in $H$, and hence is
in $C$ because it is on the same side of $H$ as $x$ is.
So $h(z) = - \infty$ implies the
convexity inequality involving $x$, $y$, and $z$.  That completes the proof 
that all four cases define generalized affine functions.

So we now show that every generalized affine function falls in one of these
four cases.  Suppose $h$ is generalized affine, and assume that we are 
not in case (a), (b), or (c).  Then at least one of $B$ and $C$ is nonempty.
This implies $A \neq E$, hence, $A$ being an affine set, $A^c$ is dense in $E$.
If $B = \emptyset$, then $C$ is dense in $E$, hence $C$ being a convex set,
$C = E$ and we are in case (c) contrary to assumption.
Hence $B \neq \emptyset$.   The same proof with $B$ and $C$ swapped
implies $C \neq \emptyset$.

Hence $B$ and $C$ are disjoint nonempty convex sets,
so by the separating hyperplane
theorem \citep[Theorem~11.3]{rock-convex}, there is an affine function $g$
on $S$ such that
\begin{subequations}
\begin{alignat}{2}
    x & \notin B, & \qquad & \text{when $g(x) < 0$}
    \label{eq:separate-C}
    \\
    x & \notin C, & \qquad & \text{when $g(x) > 0$}
    \label{eq:separate-B}
\end{alignat}
\end{subequations}
and the hyperplane in question is
$$
   H = \set{ x \in E : g(x) = 0 }.
$$

Again we know $A^c$ is dense in $E$, hence $B$ is dense
in the half space on one side of $H$, and $C$ is dense in the half space
on the other side of $H$.   Now convexity of $B$ and $C$ imply
\begin{subequations}
\begin{alignat}{2}
    x & \in C, & \qquad & \text{when $g(x) < 0$}
    \label{eq:separate-C-sharp}
    \\
    x & \in B, & \qquad & \text{when $g(x) > 0$}
    \label{eq:separate-B-sharp}
\end{alignat}
\end{subequations}
That $h$ is generalized affine on $H$ follows from $h$ being generalized
affine on $E$.   Thus we are in case (d).
\end{proof}

We now want to show that $G(E)$ is first countable.   In aid of that
we first prove a lemma.

\begin{lem}\label{lem:lemon}
Every finite-dimensional affine space $E$ is second countable and metrizable.
If $D$ is a countable dense set in $E$, then every point of $E$ is contained
in the interior of the convex hull of some finite subset of $D$.
The same is true of any open convex subset $O$ of $E$:
every point of $O$ is contained
in the interior of the convex hull of some finite subset of $D \cap O$.
\end{lem}
\begin{proof}
The first assertion is trivial.   If the dimension of $E$ is $d$, then
the topology of $E$ is defined to make any invertible affine function
$E \to \R^d$ a homeomorphism.

The second assertion is just the case $O = E$ of the third assertion.

Assume to get a contradiction that the third assertion is false.
Then there is a point $x \in O$ that is disjoint
from the convex hull of $(O \cap D) \setminus \{x\}$.
It follows that there is a strongly separating hyperplane
\citep[Corollary~11.4.2]{rock-convex},
hence an affine function $g$ such that
\begin{align*}
  g(x) & < 0
  \\
  g(y) & > 0, \qquad y \in O \cap D \opand y \neq x
\end{align*}
But this violates $x$ being in $O$.
\end{proof}

We can now prove Theorem~\ref{first-countable}.

\begin{proof}
We need to show there is a countable local base
at $h$ for any $h \in G(E)$.   A set is a neighborhood of $h$ if it has the form
\begin{equation} \label{eq:general-neighborhood}
   \set{ g \in G(E) : g(x) \in O_x, \ x \in F },
\end{equation}
where $F$ is a finite subset of $E$ and each $O_x$ is a neighborhood
of $h(x)$ in $\exreal$.

We prove first countability
by induction on the dimension of $E$ using Theorem~\ref{recurse}.
For the basis of the induction, if $E = \{0\}$, then $G(E)$ is homeomorphic
to $\exreal$, hence actually second countable.

We now show that there is a countable local base at $h$ in each
of the four cases of Theorem~\ref{recurse}.  Fix a countable dense 
set $D$ in $E$ (there is one by Lemma~\ref{lem:lemon}).

There is only one $h$ satisfying case (a),
the constant function having the value $\infty$ everywhere.
In this case, a general neighborhood \eqref{eq:general-neighborhood}
contains a neighborhood of the form
$$
   W = \set{ g \in G(E) : g(x) > m, \ x \in F },
$$
where $m$ can be an integer.
Also by Lemma~\ref{lem:lemon} there exists
a finite subset $V$ of $D$ that contains $F$ in the interior
of its convex hull.
Then, by concavity of elements of $G(E)$, the neighborhood
$$
   W_{m, V} = \set{ g \in G(E) : g(x) > m, \ x \in V }
$$
is contained in $W$.   Hence the collection
\begin{equation} \label{eq:countable-local-base}
   \set{ W_{m, V} : \text{$m \in \nats$ and $V$ a finite subset of $D$} }
\end{equation}
is a countable local base at $h$.

The proof for case (b) is similar.  In case (c) we are considering an affine 
function $h$ on $E$.  In this case, a general 
neighborhood \eqref{eq:general-neighborhood} contains a neighborhood of the 
form
$$
   W = \set{ g \in G(E) : h(x) - \tfrac{1}{m} < g(x)
   < h(x) + \tfrac{1}{m}, \ x \in F },
$$
where $F$ is a finite subset of $E$ and $m$ is a positive integer.

Again use Lemma~\ref{lem:lemon} to choose a finite set $V$ containing
$F$ in the interior of its convex hull.
Then, by convexity and concavity of elements
of $G(E)$, the neighborhood
$$
   W_{m, V} = \set{ g \in G(E) : h(x) - \tfrac{1}{m} <  g(x)
   < h(x) + \tfrac{1}{m}, \ x \in V }.
$$
is contained in $W$ because any $y \in F$ can be written
as a convex combination of the elements of $V$
$$
   y = \sum_{x \in V} p_x x,
$$
where the $p_x$ are nonnegative and sum to one, so $g \in W_{m, n}$ implies
$$
   g(y) \le \sum_{x \in V} p_x g(x)
   < \left( \sum_{x \in V} p_x h(x) \right) + \frac{1}{m}
   = h(y) + \frac{1}{m}
$$
by the convexity inequality, and the same with the inequalities reversed
and $1 / m$ replaced by $- 1 / m$ by the concavity inequality.
Hence the collection \eqref{eq:countable-local-base} with $W_{m, V}$
as defined in this part is a countable local base at $h$.

In case (d) we are considering a generalized affine function $h$ that is
neither affine nor constant.   Then, as the proof of Theorem~\ref{recurse}
shows, there is a hyperplane $H$ that is the boundary of $h^{-1}(\infty)$
and $h^{-1}(-\infty)$.   The induction hypothesis is that $G(H)$ is first
countable, that is, there is a countable family $\mathcal{U}$
of neighborhoods of $h$ in $G(E)$ such that
$$
   \set{ U \cap H : U \in \mathcal{U} }
$$
is a countable local base for $G(H)$ at the restriction of $h$ to $H$.

Again consider a general neighborhood of $h$ \eqref{eq:general-neighborhood};
call it $W$.   Let $g | H$ denote the restriction of $g \in G(E)$ to $H$.
For any subset $Q$ of $G(E)$ let $Q | H$ be defined by
$$
   Q | H = \set{ q | H : q \in Q }.
$$
Then the induction hypothesis is that there exists a $U \in \mathcal{U}$
such that $U | H$ is contained in $W | H$.

Also adopt the notation \eqref{eq:A}, \eqref{eq:B}, and \eqref{eq:C} used
in the proofs of Theorems~\ref{th:non-recursive} and~\ref{recurse}.
By Lemma~\ref{lem:lemon} choose a set $V_B$ in $D \cap (B \setminus H)$
that contains $F \cap (B \setminus H)$ in the interior of its convex hull,
and choose a set $V_C$ in $D \cap (C \setminus H)$
that contains $F \cap (C \setminus H)$ in the interior of its convex hull,

Then, by convexity and concavity of elements
of $G(E)$, the neighborhood
\begin{equation} \label{eq:special-neighborhood}
   W_{m, U, V_B, V_C}
   =
   \set{ g \in U : h(x) \ge m, \ x \in V_B \opand
   h(x) \le - m, \ x \in V_C }
\end{equation}
is contained in $W$.   To see this, first consider $x \in F \cap H$
(if there are any).   Any $g$ in \eqref{eq:special-neighborhood} has
$g(x) \in O_x$ because of $U | H \subset W | H$.
Next consider $x \in F \cap B$
(if there are any).   Any $g$ in \eqref{eq:special-neighborhood} has
$g(x) \in O_x$ because of concavity of $g$ assures $g(x) \ge m$,
and we chose $m$ so that $(m, \infty) \subset O_x$.
Last consider $x \in F \cap C$
(if there are any).   Any $g$ in \eqref{eq:special-neighborhood} has
$g(x) \in O_x$ because of convexity of $g$ assures $g(x) \le - m$,
and we chose $m$ so that $(-\infty, -m) \subset O_x$.

Hence the collection
$$
   \set{ W_{m, U, V \cap (B \setminus H), V \cap (C \setminus H)} :
   \text{$m \in \nats$ and $U \in \mathcal{U}$ and $V$ a finite subset of $D$} }
$$
is a countable local base at $h$.

We forgot the case where $E$ is empty.   Then $G(E)$ is a one-point space
whose only element is the empty function (that has no argument-value pairs).
It is trivially first countable.
\end{proof}

We now prove Theorem~\ref{vec-char}.

\begin{proof}
First, assume $h$ satisfies the conditions of Theorem~\ref{recurse} on $E$.  We then show that $h$ satisfies the conditions of Theorem~\ref{vec-char} by induction on the dimension of $E$.  The induction  hypothesis, $\Hyp(p)$, is that the conclusions of Theorem~\ref{recurse} imply that the conclusions of Theorem~\ref{vec-char} hold when dim($E$) = $p$.  We now show that $\Hyp(0)$ holds.  In this setting, $E = \{0\}$.  Therefore our result holds with 
$j=0$ and $h$ is constant on $E$.  The basis of the induction holds.   

Let dim($E$) = $p+1$.  We now show that $\Hyp(p)$ implies that $\Hyp(p+1)$ holds.  In the event that $h$ is characterized by case (a) or (b) of Theorem~\ref{recurse} then our result holds with $j = 0$.  If case (c) of Theorem~\ref{recurse} characterizes $h$ then there is an affine function $f_1$ defined by 
$f_1(x) = \inner{x,\eta_1} - \delta_1, \, x \in E$, 
such that $h(x) = +\infty$ for $x$ such that $f_1(x) > 0$, $h(x) = -\infty$ 
for $x$ such that $f_1(x) < 0$, and $h$ is generalized affine on the hyperplane $H_1 = \{x: f_1(x) = 0\}$.  The hyperplane $H_1$ is $p$-dimensional affine subspace of $E$.  Now, for some arbitrary $\zeta_1 \in H_1$, define 
\begin{align*}
V_1 &= \{ x - \zeta_1: x \in H_1 \} \\
  &= \{y \in E: \inner{y,\eta_1} = \delta_1 - \inner{\zeta_1,\eta_1} \} \\
  &= \{y \in E: \inner{y,\eta_1} = 0 \}
\end{align*}
where the last equality follows from $\zeta_1 \in H_1$.  The space $V_1$ is a $p$-dimensional vector subspace of $E$ since every affine space containing the origin is a vector subspace \citep*[Theorem 1.1]{rock-convex} and because every translate of an affine space is another affine space \citep*[pp.  4]{rock-convex}.  Let 
\begin{equation} \label{h1}
  h_1(y) = h(y+\zeta_1), \qquad y \in V_1.  
\end{equation}
The function $h_1$ is convex since the composition of a convex function with an affine function is convex.  To see this, let $0 < \lambda < 1$, pick $y_1, y_2 \in V_1$ and observe that
\begin{align*}
h_1(\lambda y_1 + (1-\lambda)y_2) &= h(\lambda y_1 + 
    (1-\lambda)y_2 + \zeta_1) \\
  &\leq  \lambda h(y_1 + \zeta_1) + (1-\lambda)h(y_2 + \zeta_1) \\
  &= \lambda h_1(y_1) + (1-\lambda)h_1(y_2).
\end{align*}
A similar argument shows that $h_1$ is concave.  Therefore $h_1$ is generalized affine.  From our induction hypothesis, the conclusions of Theorem~\ref{recurse} imply that our result holds for the generalized affine function $h_1$.  These conditions are that there exist finite sequences of vectors 
$\tilde{\eta}_2$, $\ldots$, $\tilde{\eta}_j$ being a linearly independent subset of $V_1^{\textstyle{*}}$, the dual space of $V_1$, and scalars 
$\tilde{\delta}_2$, $\ldots$, $\tilde{\delta}_j$ such that $h_1$ has the following form.  Define $\widetilde{H}_1 = V_1$ and, inductively, for integers $i$ such that $2 < i \le j$
\begin{equation} \label{planeHstar}
  \begin{split}
    \widetilde{H}_i & = \set{ x \in \widetilde{H}_{i - 1} : 
      \inner{x, \tilde{\eta}_i} = \tilde{\delta}_i } \\
    \widetilde{C}_i^+ & = \set{ x \in \widetilde{H}_{i - 1} : 
      \inner{x, \tilde{\eta}_i} > \tilde{\delta}_i } \\
    \widetilde{C}_i^- & = \set{ x \in \widetilde{H}_{i - 1} : 
      \inner{x, \tilde{\eta}_i} < \tilde{\delta}_i }
  \end{split}
\end{equation}
all of these sets (if any) being nonempty.  Then $h_1(x) = + \infty$ whenever 
$x \in \widetilde{C}_i^+$ for any $i$, $h_1(x) = - \infty$ whenever 
$x \in \widetilde{C}_i^-$ for any $i$, and $h_1$ is either affine or constant 
on $\widetilde{H}_j$, where $+\infty$ and $-\infty$ are allowed for constant 
values.

It remains to show that the conditions of Theorem~\ref{vec-char} hold with respect to $h$.  The vectors $\tilde{\eta}_i$, $i = 2,...,j$ can be extended to form a set of vectors $\eta_i$, $i = 2,...,j$ in $\Estar$ by the Hahn-Banach Theorem \citep*[Theorem 3.6]{rudin}.  The vectors $\eta_i$, $i = 2,...,j$, form a linearly independent subset of $\Estar$.  To see this, let $\sum_{k=2}^ja_k\eta_k = 0$ on $E$ for scalars $a_k$, $k = 2,...,j$.  Then $\sum_{k=2}^ja_k\eta_k = 0$ on $V_1$ which implies that $a_k = 0$ for $k = 2,...,j$ by the definition of linearly independent.  Let $H_0 = E$, and, for $i = 2,...,j$, define 
\begin{equation} \label{planeH}
  \begin{split}
     H_i & = \set{ x \in H_{i - 1} : \inner{x, \eta_i} = \delta_i } \\
     C_i^+ & = \set{ x \in H_{i - 1} : \inner{x, \eta_i} > \delta_i } \\
     C_i^- & = \set{ x \in H_{i - 1} : \inner{x, \eta_i} < \delta_i }
  \end{split}
\end{equation}
where $\delta_i = \tilde{\delta}_i - \inner{\zeta_1, \eta_i}$ for $i = 2,...,j$ and $\widetilde{H}_i = H_i + \zeta_1$ as a result.  We see that $h(x) = h_1(x - \zeta_1) = +\infty$ whenever $\inner{x + \zeta_1,\eta_i} > \tilde{\delta}_i$.  Therefore $h(x) = +\infty$ for all $x \in C_i^+$ for any $i$.  The same derivation shows that $h(x) = -\infty$ whenever $x \in C_i^-$ for any $i$.  The generalized affine function $h$ is either affine or constant on $H_j$, where $+\infty$ and $-\infty$ are allowed for constant values since the composition of an affine function with an affine function is affine.  

We now show that the vectors $\eta_1,...,\eta_j$ are linearly independent.  Assume that $\sum_{k=1}^ja_k\eta_k = 0$ on $E$ for scalars $a_k$, $k = 1,...,j$.  This assumption implies that $\sum_{k=1}^ja_k\tilde{\eta} = 0$ on $V_1^{\textstyle{*}}$ where $\tilde{\eta}_1$ is the restriction of $\eta_1$ to $V_1$.  Thus $\tilde{\eta}_1$ is an element of $V_1^{\textstyle{*}}$ and $\tilde{\eta}_1 = 0$ on $V_1$ since 
$
  \inner{y,\tilde{\eta}_1} = \inner{y,\eta_1} = 0 
$
on $V_1$.  Therefore $\sum_{k=2}^ja_k\tilde{\eta}_k = 0$ where $a_k = 0$ for $k = 2,...,j$ from what has already been shown.  In the event that $a_1 = 0$, we can conclude that $\eta_1,\ldots,\eta_j$ are linearly independent.  Now consider $a_1 \neq 0$.  In this case, $\sum_{k=1}^ja_k\eta_k = 0$ implies that $\eta_1 = \sum_{k=2}^jb_k\eta_k$ where $b_k = -a_k/a_1$.  This states that $\sum_{k=2}^jb_k\tilde{\eta}_k = 0$ on $V_1$.  Therefore, $b_k = 0$ for all $k = 2,...,j$ which implies that $\eta_1$ is the zero vector, which is a contradiction.  Thus $a_1 = 0$ and we can conclude that $\eta_1,...,\eta_j$ are linearly independent.  This completes one direction of the proof.

Now assume that $h$ satisfies the conclusions of Theorem~\ref{vec-char} and 
show that these conclusions imply that Theorem~\ref{recurse} holds by induction on $j$.  The induction hypothesis, $\Hyp(j)$, is that the conclusions of Theorem~\ref{vec-char} imply that the conclusions of Theorem~\ref{recurse}  hold for sequences of length $j$.  For the basis of the induction let $j = 0$.  We now show that $\Hyp(0)$ holds.  The generalized affine function $h$ is either affine or constant on $E$ where $+\infty$ and $-\infty$ are allowed for constant values.  This characterization of $h$ is the same as cases (a) of (b) of Theorem~\ref{recurse}.  The basis of the induction holds.  

We now show that $\Hyp(j)$ implies that $\Hyp(j+1)$ holds.  When the length of sequences is $j+1$, there exist vectors $\eta_1,...,\eta_{j+1}$ and scalars $\delta_1,...,\delta_{j+1}$ such that $h$ has the following form.  Define $H_0 = E$ and, inductively, for integers $i$, $0 < i \leq j+1$, such that the sets in \eqref{planeH} are all nonempty.  Then $h(x) = +\infty$ whenever $x \in C_i^+$ for any $i$, $h(x) = -\infty$ whenever $x \in C_i^-$ for any $i$, and $h$ is either affine or constant on $H_{j+1}$, where $+\infty$ and $-\infty$ are allowed for constant values.  From the definition of the sets $H_1$, $C_1^+$, and $C_1^-$, there is an affine function $f_{1}$ defined by $f_{1}(x) = \inner{x, \eta_{1}} - \delta_{1},\, x \in E$, such that $h(x) = +\infty$ for all $x \in E$ such that $f_{1}(x) > 0$ and $h(x) = -\infty$ for all $x \in E$ such that $f_{1}(x) < 0$.  This is equivalent to the case (c) characterization of $h$ in Theorem~\ref{recurse}, provided we show that the restriction of $h$ to $H_1$ is a generalized affine function.

Define $V_1 = H_1 - \zeta_1$ for some arbitrary $\zeta_1 \in H_1$.  Let dim($E$) = $p$.  The space $V_1$ is a ($p-1$)-dimensional vector subspace of $E$.  Define $h_1$ as in \eqref{h1}.  Let $\tilde{\eta}_i$ be the restriction of $\eta_i$ to $V_1$ so that $\tilde{\eta}_i$ is an element of $V_1^{\textstyle{*}}$ for $1 < i \leq j+1$.  Now let $\widetilde{H}_1 = V_1$ and, for $1 < i \leq j+1$, we can define the sets as in \eqref{planeHstar} where $\tilde{\delta}_i = \delta_i - \inner{\zeta_1, \tilde{\eta}_i}$.  We see that $h_1(x) = h(x + \zeta_1) = +\infty$ whenever $\inner{x + \zeta_1,\eta_i} > \tilde{\delta}_i$.  Therefore $h_1(x) = +\infty$ for all $x \in \widetilde{C}_i^+$ for any $i$.  The same derivation shows that $h_1(x) = -\infty$ whenever $x \in \widetilde{C}_i^-$ for any $i$.  The generalized affine function $h_1$ is either affine or constant on $H_{j+1}$, where $+\infty$ and $-\infty$ are allowed for constant values.  Therefore $h_1$ meets the conditions of Theorem~\ref{vec-char} with sequences of length $j$.  From $\Hyp(j)$, we know that the conclusions of Theorem~\ref{recurse} hold with respect to $h_1$.  This completes the proof.
\end{proof}

We now prove Lemma~\ref{properties-1} using the characterization of generalized affine functions on finite-dimensional vector spaces given by Theorem~\ref{vec-char}.

\begin{proof}
First suppose that $h_n$ converges to $h$.  The assumption that $h$ is finite at at least one point guarantees that $h$ is affine on $H_j$ from Theorem~\ref{vec-char}.  For all $y \in H_j$ we can write 
$
  h(y) = \inner{y, \tstar} + a
$  
where $\inner{y, \tstar} = \sum_{i=j+1}^p d_i\inner{y, \eta_i}$ and $s, d_i \in \R$.  The convergence $h_n\to h$ implies that $b_{i,n} \to d_i$, $i = j+1,...,p$ where the set of $b_{i,n}$s is empty when $j = p$ and that $a_n \to a$ as $n\to\infty$.  Thus conclusions (c) and (d) hold.  To show that conclusions (a) and (b) hold we will suppose that $j > 0$, because these conclusions are vacuous when $j = 0$.  Both cases (a) and (b) will be shown by induction with the hypothesis $\Hyp(m)$ that $b_{(j-m),n}\to +\infty$ and $b_{(j-m+1),n}/b_{(j-m),n} \to 0$ as $n\to\infty$ for $0 \leq m \leq j-1$.  We now show that the basis of this induction holds.  Pick $y \in C_j^+$ and observe that 
$$
  h_n(y) = a_n + b_{j, n}\left(\inner{y,\eta_j} - \delta_j\right) 
    + \sum_{k = j + 1}^p b_{k, n} \inner{y, \eta_k} \to +\infty.
$$
since $h(y) = +\infty$ and $h_n\to h$ pointwise.  From this, we see that $b_{j,n} \to +\infty$ as $n\to\infty$ and $b_{j+1,n}/b_{j,n} \to 0$ as $n\to\infty$ from part (c).  Therefore $\Hyp(0)$ holds.  It is now shown that $\Hyp(m)$ implies that $\Hyp(m+1)$ holds.  There exists a basis $y_1,...,y_p$ in $E^{\textstyle{*}\textstyle{*}}$, the dual space of $\Estar$, such that $\inner{y_i, \eta_k} = 0$ when $i \neq k$ and $\inner{y_i, \eta_k} = 1$ when $i = k$.  The set of vectors $y_1,...,y_p$ is a basis of $E$ since  $E = E^{\textstyle{*}\textstyle{*}}$.  Arbitrarily choose a $y \in H_{j-m-1}$ such that 
$
  y = \sum_{i=1}^{j-m-1} \delta_i y_i + c_1y_{j-m}
$
where $c_1 > \delta_{j-m}$.  At this choice of $y$ we see that $h(y) = +\infty$ 
and 
\begin{align*}
  h_n(y) &= a_n + \sum_{i=1}^{j-m+1}b_{i,n}\left(\inner{y, \eta_i} 
    - \delta_i\right) \\
  &= a_n + b_{(j-m),n}\left(\inner{y, \eta_{j-m}} - \delta_{j-m}\right) \\
  &\to +\infty
\end{align*}
as $n\to\infty$.  Therefore $b_{(j-m),n}\to +\infty$ as $n\to\infty$.  Now arbitrarily choose 
$
  y = \sum_{i=1}^{j-m-1} \delta_i y_i + c_1y_{j-m} + c_2y_{j-m+1}
$ 
where $c_1$ is defined as before and $c_2 < \delta_{j-m+1}$.  At this choice of $y$ we see that $h(y) = +\infty$ and 
\begin{equation} \label{ind}
  \begin{split}
    h_n(y) &= a_n + \sum_{i=1}^{j-m+1}b_{i,n}\left(\inner{y, \eta_i} 
      - \delta_i\right) \\
    &= a_n + b_{(j-m),n}\left(\inner{y, \eta_{j-m}} - \delta_{j-m} \right.  \\
      &\left.\qquad+  \frac{b_{(j-m+1),n}}{b_{(j-m),n}} 
      \left(\inner{y, \eta_{j-m+1}} - \delta_{j-m+1}\right)\right) \\  
    &= a_n + b_{(j-m),n}\left(c_1 - \delta_{j-m}
      -  \frac{b_{(j-m+1),n}}{b_{(j-m),n}} \left(c_2 - \delta_{j-m+1}
        \right)\right) \\
    &\to +\infty
  \end{split}
\end{equation}
as $n\to\infty$.  It follows from \eqref{ind} that 
$$
  \left(c_1 - \delta_{j-m} - \frac{b_{(j-m+1),n}}{b_{(j-m),n}} 
    \left(c_2 - \delta_{j-m+1}\right)\right) \geq 0
$$
for sufficiently large $n$.  This implies that 
$$
  \frac{b_{(j-m+1),n}}{b_{(j-m),n}} \leq 
    \frac{c_1 - \delta_{j-m}}{\delta_{j-m-1} - c_2}
$$
for sufficiently large $n$.  From the arbitrariness of the constants $c_1$ and $c_2$ and \eqref{ind}, we can conclude that $b_{(j-m+1),n}/b_{(j-m),n}\to 0$ as $n\to\infty$.  Therefore $\Hyp(m+1)$ holds and this completes one direction of the proof.  

We now assume that conditions (a) through (d) and the $h_n$ takes the form in \eqref{affine-seq}.  Let 
$
  \lim_{n\to\infty} \sum_{i=j+1}^p b_{i,n}\eta_i = \tstar
$
and $\lim_{n\to\infty} a_n = a$.  Cases (a) through (d) then imply that 
\begin{equation} \label{limit-seq}
  h_n(y) \to \left\{ \begin{array}{cc}
    -\infty, &\qquad y \in C_i^- \\
    \inner{y,\tstar} + a, &\qquad y \in H_j \\
    +\infty, &\qquad y \in C_i^+
  \end{array}\right.
\end{equation}
for all $i = 1,...,j$ where the right hand side of \eqref{limit-seq} is a generalized affine function in its Theorem~\ref{vec-char} representation.  This completes the proof.
\end{proof}

\section{Proofs of MGF and moment convergence results}

We first prove Theorem~\ref{MVmgf-1}.

\begin{proof}
Suppose $\varphi_X$ is an MGF, hence finite on a neighborhood $W$ of zero.Fix $t \in \Estar$.   Then by \eqref{eq:cramer-wold} $\varphi_{\inner{X, t}}(s)$ is finite whenever $s t \in W$.  Continuity of scalar multiplication means there exists an $\varepsilon > 0$ such that $s t \in W$ whenever $\abs{s} < \varepsilon$.
That proves one direction.

Conversely, suppose $\varphi_{\inner{X, t}}$ is an MGF for each $t \in \Estar$.   Suppose $v_1$, $\ldots,$ $v_d$ is a basis for $E$ and $w_1$, $\ldots,$ $w_d$ is the dual basis for $E^{\textstyle *}$ that satisfies \eqref{eq:dual-bases}. Then there exists $\varepsilon > 0$ such that $\varphi_{\inner{X, w_i}}$ is finite on $[- \varepsilon, \varepsilon]$ for each $i$.

We can write each $t \in E^{\textstyle *}$ as a linear combination of basis vectors
$$
   t = \sum_{i = 1}^d a_i w_i,
$$
where the $a_i$ are scalars that are unique \citep[Theorem~1 of Section~15]{halmos}. Applying \eqref{eq:dual-bases} we get
$$
   \inner{v_j, t} = a_j,
$$
so
$$
   t = \sum_{i = 1}^d \inner{v_i, t} w_i,
$$
and
$$
   \inner{X, t} = \sum_{i = 1}^d \inner{v_i, t} \inner{X, w_i}.
$$
Suppose
$$
   \abs{\inner{v_i, t}} \le \varepsilon, \qquad i = 1, \ldots, d
$$
(the set of all such $t$ is a neighborhood of $0$ in $E^{\textstyle *}$). Let $\sign$ denote the sign function, which takes values $- 1$, $0$, and $+1$ as its argument is negative, zero, or positive, and write
$$
   s_i = \sign( \inner{v_i, t} ), \qquad i = 1, \ldots, d.
$$
Then we can write $\inner{X, t}$ as a convex combination
$$
   \inner{X, t} = \sum_{i = 1}^d \frac{\inner{v_i, t}}{s_i \varepsilon}
   \cdot s_i \varepsilon \inner{X, w_i}
   +
   \left( 1 - \sum_{i = 1}^d \frac{\inner{v_i, t}}{s_i \varepsilon} \right)
   \cdot \inner{X, 0}.
$$
So, by convexity of the exponential function,
$$
   \varphi_X(t)
   \le
   \sum_{i = 1}^d \frac{\inner{v_i, t}}{s_i \varepsilon}
   \varphi_{\inner{X, w_i}}(s_i \varepsilon)
   +
   \left( 1 - \sum_{i = 1}^d \frac{\inner{v_i, t}}{s_i \varepsilon} \right)
   <
   \infty.
$$
That proves the other direction.
\end{proof}

We now prove Theorem~\ref{MVmgf-2}.

\begin{proof}
The one-dimensional case of this theorem is proved in \cite{billingsley}.  We only need to show the general case follows by Cram\'{e}r-Wold.  It follows from the assumption that $\varphi_{\inner{X_n, t}}$ converges on a neighborhood $W$ of zero for each $t \in \Estar$.  Then \eqref{eq:weak} follows from the one-dimensional case of this theorem and the Cram\'{e}r-Wold theorem.   And this implies
$$
   \inner{X_n, t} \weakto \inner{X, t}, \qquad t \in E^{\textstyle *}.
$$
By the one-dimensional case of this theorem, this implies $\inner{X, t}$ has an MGF for each $t$, and then Theorem~\ref{MVmgf-1} implies $X$ has an MGF $\varphi_X$. By the one-dimensional case of this theorem, $\varphi_{\inner{X_n, t}}$ converges pointwise to $\varphi_{\inner{X, t}}$.  So by \eqref{eq:cramer-wold}, $\varphi_{X_n}$ converges pointwise to $\varphi_X$.
\end{proof}

We now prove Theorem~\ref{MVmgf-3}.

\begin{proof}
From Theorem~\ref{MVmgf-2}, we have that $\inner{X_n,t_i} \weakto \inner{X,,t_i}$.  Continuity of the exponential function implies that $e^{\inner{X_n,t_i}} \weakto e^{\inner{X,t_i}}$.  
Now, pick an $\varepsilon > 0$ such that both $\varepsilon\sum_{i=1}^k t_i \in W$ and $\varepsilon\sum_{i=1}^k u_i \in W$ where $u_1 = -t_1$ and $u_i = t_i$ 
for all $i > 1$. This construction gives 
\begin{equation} \label{bils-1}
  e^{\inner{X_n,\varepsilon\sum_{i=1}^kt_i}} \weakto 
  e^{\inner{X,\varepsilon\sum_{i=1}^kt_i}}
\end{equation}
and
\begin{equation} \label{bils-2}
  \E\left(e^{\inner{X_n,\varepsilon\sum_{i=1}^kt_i}}\right) \weakto 
  \E\left(e^{\inner{X,\varepsilon\sum_{i=1}^kt_i}}\right).
\end{equation}
Equations \eqref{bils-1} and \eqref{bils-2} imply that 
$
  e^{\inner{X_n,\varepsilon\sum_{i=1}^kt_i}}
$
is uniformly integrable by \cite[Theorem 3.6]{billingsley-conv-meas}.  
A similar argument shows that 
$
  e^{\inner{X_n,\varepsilon\sum_{i=1}^ku_i}}
$
is uniformly integrable.  We now bound 
$\abs{\varepsilon^k\prod_{i = 1}^k \inner{X_n, t_i}}$ to show uniform 
integrability of $\prod_{i = 1}^k \inner{X_n, t_i}$.  Define
$$
  A_n = \{X_n : \prod_{i = 1}^k \inner{X_n, t_i} \geq 0\}.
$$
and let $I_A$ be the indicator function.  We have,
\begin{align*}
  \varepsilon^k\prod_{i = 1}^k \inner{X_n, t_i} 
    &\leq \prod_{i = 1}^k \inner{X_n, \varepsilon t_i}I_{A_n} \\
    &\leq e^{\inner{X_n,\varepsilon\sum_{i=1}^kt_i}}I_{A_n} \\
    &\leq e^{\inner{X_n,\varepsilon\sum_{i=1}^kt_i}}
\end{align*}
and 
\begin{align*}
  -\varepsilon^k\prod_{i = 1}^k \inner{X_n, t_i} 
    &= \prod_{i = 1}^k \inner{X_n, \varepsilon u_i} \\
    &\leq \prod_{i = 1}^k \inner{X_n, \varepsilon u_i}I_{A_n^c} \\
    &\leq e^{\inner{X_n,\varepsilon\sum_{i=1}^ku_i}}I_{A_n^c} \\
    &\leq e^{\inner{X_n,\varepsilon\sum_{i=1}^ku_i}}.
\end{align*}
Therefore 
$$
  \abs{\varepsilon^k\prod_{i = 1}^k \inner{X_n, t_i}} 
    \leq e^{\inner{X_n,\varepsilon\sum_{i=1}^k t_i}} 
      + e^{\inner{X_n,\varepsilon\sum_{i=1}^k u_i}}
$$
The sum of uniformly integrable is uniformly integrable.  This implies that $\abs{\varepsilon^k\prod_{i = 1}^k \inner{X_n, t_i}}$ is uniformly integrable.  Scaling of uniformly integrable is also uniformly integrable, which implies $\prod_{i = 1}^k \inner{X_n, t_i}$ is uniformly integrable.  Our result follows from \cite[Theorem 3.5]{billingsley-conv-meas} and this completes the proof.
\end{proof}

\section{Counterexample}
\label{counterexample}

This section provides a counterexample to the non-theorem which is Theorem~\ref{main} with its conditions removed (that is, the assertion that cumulant generating function convergence always occurs). It shows that some conditions like those the theorem requires are needed.

\subsection{Model}

Suppose we have a two-dimensional exponential family with generating measure $\lambda$ concentrated on the set
$$
   S = \{ (0, 0), (0, 1) \} \cup \set{ (1, n) : n \in \nats },
$$
where $\nats$ is the set of natural numbers 0, 1, 2, $\ldots.$ And suppose $\lambda$ takes values
$$
   \lambda(x) = \frac{1}{x_2 !}, \qquad x \in S.
$$

The Laplace transform of $\lambda$ is the function of $\theta$ given by
$$
   1 + e^{\theta_2} +
   e^{\theta_1} \sum_{x_2 = 0}^\infty \frac{e^{x_2 \theta_2}}{x_2 !}
   =
   1 + e^{\theta_2} +
   e^{\theta_1} e^{e^{\theta_2}}
$$
and the cumulant function (log Laplace transform) is
\begin{equation} \label{eq:example-cumfun}
   c(\theta)
   =
   \log\left[
   1 + e^{\theta_2} +
   e^{\theta_1 + e^{\theta_2}}
   \right]
\end{equation}

\subsection{Maximum Likelihood}

Suppose the observed value of the canonical statistic is $x = (0, 1)$.

From Chapter~2 of \citet{geyer} we know that we can find the MLE
in the completion of the family by taking limits first in the direction
$\eta_1 = (-1, 0)$ (which is a direction of recession) and second in the direction
$\eta_2 = (0, 1)$ (which is a direction of recession for the limiting conditional
model resulting from the first limit).  Thus the MLE in the completion
is the completely degenerate distribution concentrated at the observed data.
The Theorem 5 (in the main article) characterization of the corresponding 
generalized affine function evaluated at data $x$, $h(x)$, yields set 
$C_1^- = \{(1,n):n\in\nats\}$ and thus $D_1 = \{(1,n),n\nats,\;n\geq 1\}$.  
Clearly $\lambda(D_1) > 0$, and we have 
$$
  \sup_{\theta \in \Theta}\sup_{y \in D_1} 
    e^{\inner{y,\theta} - c_{D_1}(\theta)} 
  \geq 
    \sup_{y \in \nats}e^{\inner{(1,y),(0,1)} - c_{D_1}((0,1))} = \infty.
$$
Therefore the bound condition of Theorem 6 in the main article is violated.  
We now show that CGF convergence along a likelihood maximizing sequence 
fails for $t$ in a neighborhood of 0.

\subsection{Log Likelihood}

The log likelihood is
\begin{align*}
   l(\theta)
   & =
   x_1 \theta_1 + x_2 \theta_2 - c(\theta)
   \\
   & =
   \theta_2 - c(\theta)
   \\
   & =
   - \log\left[
   e^{- \theta_2} + 1 +
   e^{\theta_1 - \theta_2 + e^{\theta_2}}
   \right]
\end{align*}

\subsection{Likelihood Maximizing Sequences}

Because the MLE in the completion is completely degenerate and because
$\lambda(x) = 1$, the log likelihood must go to $\log(1) = 0$ along any
likelihood maximizing sequence.

We know from Lemma~\ref{properties-1} in the main article
that any likelihood maximizing sequence $\theta_n$ must have
\begin{enumerate}
\item[(i)] $\theta_{1, n} \to - \infty$,
\item[(ii)] $\theta_{2, n} \to + \infty$,
\item[(iii)] $\abs{\theta_{2, n} / \theta_{1, n}} \to 0$,
\end{enumerate}
but now we see that, in this example, it must also have
\begin{enumerate}
\item[(iv)] $\theta_{1, n} - \theta_{2, n} + e^{\theta_{2, n}} \to - \infty$.
\end{enumerate}

Thus we see that Lemma~1 doesn't tell us everything
about likelihood maximizing sequences (it may do under the conditions
of Brown).

\subsection{Cumulant Generating Function Convergence}

The cumulant generating function for canonical parameter value $\theta$ is
$$
   k_\theta(t) = c(\theta + t) - c(\theta).
$$
Thus along a likelihood maximizing sequence we have
\begin{align*}
   k_{\theta_n}(t)
   & =
   \log\left[
   \frac{
   1 + e^{\theta_2 + t_2} +
   e^{\theta_1 + t_1 + e^{\theta_2 + t_2}}
   }{
   1 + e^{\theta_2} +
   e^{\theta_1 + e^{\theta_2}}
   }
   \right]
   \\
   & =
   \log\left[
   \frac{
   e^{- \theta_2} + e^{t_2} +
   e^{\theta_1 - \theta_2 + t_1 + e^{\theta_2 + t_2}}
   }{
   e^{- \theta_2} + 1 +
   e^{\theta_1 - \theta_2 + e^{\theta_2}}
   }
   \right]
\end{align*}

We know the denominator of the fraction converges to one along any
likelihood maximizing sequence.  The cumulant generating function of
the distribution concentrated at $x$ is the log of
$$
   e^{0 \cdot t_1 + 1 \cdot t_2}
$$
so
$$
   k_\text{limit}(t) = t_2
$$
Thus we see that to get the correct limit we need a different condition
\begin{enumerate}
\item[(v)] $\theta_{1, n} - \theta_{2, n} + e^{\theta_{2, n} + t_2} \to - \infty$.
\end{enumerate}
Since (i) through (iv) do not imply (v) unless $t_2 \le 0$, we cannot
guarantee cumulant generating function convergence on a neighborhood
of zero.

Suppose, for concreteness
\begin{equation} \label{eq:concrete-sequence}
   \theta_n = (- n, \log(n))
\end{equation}
so the sequence in (v) becomes
$$
   - n - \log(n) + n e^{t_2}
$$
Hence condition (v) is not satisfied unless $t_2 \le 0$,
but conditions (i) through (iv) are satisfied.

\subsection{Nonconvergence of First Moments}

First moments (of the canonical statistic) are given by differentiating
the cumulant function \eqref{eq:example-cumfun}
$$
   \nabla c(\theta)
   =
   \begin{pmatrix}
   \frac{e^{\theta_1 + e^{\theta_2}}}
   {1 + e^{\theta_2} + e^{\theta_1 + e^{\theta_2}}}
   \\
   \frac{e^{\theta_2} + e^{\theta_1 + e^{\theta_2} + \theta_2}}
   {1 + e^{\theta_2} + e^{\theta_1 + e^{\theta_2}}}
   \end{pmatrix}
$$
The first moment of the LCM, which is concentrated at $x$ is just $x$.
So the necessary and sufficient condition for convergence of first moments to
the first moments of the LCM is
\begin{align*}
   \frac{e^{\theta_{1, n} + e^{\theta_{2, n}}}}
   {1 + e^{\theta_{2, n}} + e^{\theta_{1, n} + e^{\theta_{2, n}}}}
   & \to 0
   \\
   \frac{e^{\theta_{2, n}} + e^{\theta_{1, n} + e^{\theta_{2, n}}
       + \theta_{2, n}}}
   {1 + e^{\theta_{2, n}} + e^{\theta_{1, n} + e^{\theta_{2, n}}}}
   & \to 1
\end{align*}

For the specific likelihood maximizing sequence \eqref{eq:concrete-sequence}
we have
\begin{align*}
   \frac{e^{\theta_{1, n} + e^{\theta_{2, n}}}}
   {1 + e^{\theta_{2, n}} + e^{\theta_{1, n} + e^{\theta_{2, n}}}}
   & =
   \frac{e^{- n + n}}
   {1 + n + e^{- n + n}}
   \\
   & =
   \frac{1}{2 + n}
   \\
   \frac{e^{\theta_{2, n}} + e^{\theta_{1, n} + e^{\theta_{2, n}}
       + \theta_{2, n}}}
   {1 + e^{\theta_{2, n}} + e^{\theta_{1, n} + e^{\theta_{2, n}}}}
   & =
   \frac{n + e^{- n + n + \log(n)}}
   {1 + n + e^{- n + n}}
   \\
   & =
   \frac{2 n}{2 + n}
\end{align*}
The first converges to 0 as it must for CGF convergence.
The second converges to 2, but it must converge to 1 for CGF convergence.
So we do not get convergence of first moments for this model and this
likelihood maximizing sequence, hence cannot have CGF convergence.
%
%

\subsection{Nonconvergence of Second Moments}

Non-convergence of first moments already makes CGF convergence impossible,
but since our main interest in CGF convergence is convergence of second
moments, which are components of the Fisher information matrix, we compute
them too.

For $c$ given by \eqref{eq:example-cumfun} and $\theta_n$ given by
\eqref{eq:concrete-sequence}
$$
   \nabla^2 c(\theta_n)
   =
   \frac{1}{(2 + n)^2}
   \begin{pmatrix}
   1 + n
   &
   n^2
   \\
   n^2
   &
   n (4 + n^2)
   \end{pmatrix}
   \to
   \begin{pmatrix}
   0
   &
   1
   \\
   1
   &
   \infty
   \end{pmatrix}
$$
The variance-covariance matrix for the LCM is the zero matrix
(the variance-covariance matrix of a completely degenerate distribution).
Hence we do not get convergence of Fisher information for this example.

\section{R}

\begin{itemize}
\item The version of R used to make this document is 3.6.1.
\item The version of the \texttt{knitr} package used to make this document is
    1.24.
\item The version of the \texttt{glmdr} package used to make this document is
    0.1.    
\item The version of the \texttt{rcdd} package used to make this document is
    1.2.2.
\item The version of the \texttt{numDeriv} package used to make this document is
    2016.8.1.1.
\item The version of the \texttt{alabama} package used to make this document is
    2015.3.1.
\item The version of the \texttt{Matrix} package used to make this document is
    1.2.17.
\end{itemize}

Load these packages.
\begin{knitrout}
\definecolor{shadecolor}{rgb}{0.969, 0.969, 0.969}\color{fgcolor}\begin{kframe}
\begin{alltt}
\hlkwd{library}\hlstd{(glmdr)}
\hlkwd{library}\hlstd{(rcdd)}
\end{alltt}

{\ttfamily\noindent\itshape\color{messagecolor}{\#\# If you want correct answers, use rational arithmetic.\\\#\# See the Warnings sections added to help pages for\\\#\#\ \ \ \  functions that do computational geometry.}}\begin{alltt}
\hlkwd{library}\hlstd{(numDeriv)}
\hlkwd{library}\hlstd{(alabama)}
\hlkwd{library}\hlstd{(Matrix)}
\end{alltt}
\end{kframe}
\end{knitrout}

Set random number generator seeds.  We only use randomness in tests.
This assures the tests always come out the same.
\begin{knitrout}
\definecolor{shadecolor}{rgb}{0.969, 0.969, 0.969}\color{fgcolor}\begin{kframe}
\begin{alltt}
\hlkwd{set.seed}\hlstd{(}\hlnum{42}\hlstd{)}
\end{alltt}
\end{kframe}
\end{knitrout}

Figure out some stuff about the machine (only works on Linux).
\begin{knitrout}
\definecolor{shadecolor}{rgb}{0.969, 0.969, 0.969}\color{fgcolor}\begin{kframe}
\begin{alltt}
\hlkwa{if} \hlstd{(}\hlkwd{Sys.info}\hlstd{()[}\hlstr{"sysname"}\hlstd{]} \hlopt{==} \hlstr{"Linux"}\hlstd{) \{}
    \hlstd{foo} \hlkwb{<-} \hlkwd{scan}\hlstd{(}\hlstr{"/proc/cpuinfo"}\hlstd{,} \hlkwc{what} \hlstd{=} \hlkwd{character}\hlstd{(}\hlnum{0}\hlstd{),} \hlkwc{sep} \hlstd{=} \hlstr{"\textbackslash{}n"}\hlstd{)}
    \hlstd{bar} \hlkwb{<-} \hlkwd{grep}\hlstd{(}\hlstr{"^model name"}\hlstd{, foo,} \hlkwc{value} \hlstd{=} \hlnum{TRUE}\hlstd{)}
    \hlstd{bar} \hlkwb{<-} \hlkwd{unique}\hlstd{(bar)}
    \hlstd{baz} \hlkwb{<-} \hlkwd{sub}\hlstd{(}\hlstr{"^model name\textbackslash{}\textbackslash{}t: "}\hlstd{,} \hlstr{""}\hlstd{, bar)}
    \hlkwd{cat}\hlstd{(}\hlstr{"computer name:"}\hlstd{,} \hlkwd{system}\hlstd{(}\hlstr{"hostname"}\hlstd{,} \hlkwc{intern} \hlstd{=} \hlnum{TRUE}\hlstd{),} \hlstr{"\textbackslash{}n"}\hlstd{)}
    \hlkwd{cat}\hlstd{(}\hlstr{"computer model:"}\hlstd{, baz,} \hlstr{"\textbackslash{}n"}\hlstd{)}
\hlstd{\}}
\end{alltt}
\end{kframe}
\end{knitrout}

Clean R global environment.
\begin{knitrout}
\definecolor{shadecolor}{rgb}{0.969, 0.969, 0.969}\color{fgcolor}\begin{kframe}
\begin{alltt}
\hlkwd{rm}\hlstd{(}\hlkwc{list} \hlstd{=} \hlkwd{ls}\hlstd{())}
\end{alltt}
\end{kframe}
\end{knitrout}

\section{Complete separation example of Agresti}
\label{sec:agresti-complete}

\subsection{Data}

\citet[Section~6.5.1]{agresti} introduces the notion of complete separation
with the following simple logistic regression example.
\begin{knitrout}
\definecolor{shadecolor}{rgb}{0.969, 0.969, 0.969}\color{fgcolor}\begin{kframe}
\begin{alltt}
\hlstd{x} \hlkwb{<-} \hlkwd{seq}\hlstd{(}\hlnum{10}\hlstd{,} \hlnum{90}\hlstd{,} \hlnum{10}\hlstd{)}
\hlstd{x} \hlkwb{<-} \hlstd{x[x} \hlopt{!=} \hlnum{50}\hlstd{]}
\hlstd{x}
\end{alltt}
\begin{verbatim}
## [1] 10 20 30 40 60 70 80 90
\end{verbatim}
\begin{alltt}
\hlstd{y} \hlkwb{<-} \hlkwd{as.numeric}\hlstd{(x} \hlopt{>} \hlnum{50}\hlstd{)}
\hlstd{y}
\end{alltt}
\begin{verbatim}
## [1] 0 0 0 0 1 1 1 1
\end{verbatim}
\end{kframe}
\end{knitrout}

These data are included in the \texttt{glmdr} package.

\begin{knitrout}
\definecolor{shadecolor}{rgb}{0.969, 0.969, 0.969}\color{fgcolor}\begin{kframe}
\begin{alltt}
\hlkwd{data}\hlstd{(complete)}
\hlkwd{all.equal}\hlstd{(complete,} \hlkwd{as.data.frame}\hlstd{(}\hlkwd{cbind}\hlstd{(x,y)))}
\end{alltt}
\begin{verbatim}
## [1] TRUE
\end{verbatim}
\end{kframe}
\end{knitrout}

\subsection{The MLE in the LCM}
\label{sec:MLEinLCM-i}

We fit these data using R function \texttt{glmdr} in the \texttt{glmdr} R 
package \citep{glmdr}. 

\begin{knitrout}
\definecolor{shadecolor}{rgb}{0.969, 0.969, 0.969}\color{fgcolor}\begin{kframe}
\begin{alltt}
\hlstd{gout} \hlkwb{<-} \hlkwd{glmdr}\hlstd{(y} \hlopt{~} \hlstd{x,} \hlkwc{family} \hlstd{=} \hlstr{"binomial"}\hlstd{,} \hlkwc{data} \hlstd{= complete)}
\hlkwd{summary}\hlstd{(gout)}
\end{alltt}
\begin{verbatim}
## 
## MLE exists in Barndorff-Nielsen completion 
## it is completely degenerate 
## the MLE says the response actually observed is the only 
## possible value that could ever be observed
\end{verbatim}
\end{kframe}
\end{knitrout}

In this example the LCM is completely degenerate and has no 
identifiable parameters.

\subsubsection{Linearity}
\label{sec:linearity}

The function \texttt{glmdr} determines which data points belong to the support 
of the LCM.  We already know that the support of the LCM is empty.

\begin{knitrout}
\definecolor{shadecolor}{rgb}{0.969, 0.969, 0.969}\color{fgcolor}\begin{kframe}
\begin{alltt}
\hlstd{gout}\hlopt{$}\hlstd{linearity}
\end{alltt}
\begin{verbatim}
## [1] FALSE FALSE FALSE FALSE FALSE FALSE FALSE FALSE
\end{verbatim}
\end{kframe}
\end{knitrout}

The support of the LCM (the linearity) are the data points with responses that 
are not conditioned to be their observed value.

\subsection{One-sided confidence intervals for mean value parameters}

We now provide one-sided confidence intervals for mean value parameters whose MLE is on the boundary.  We calculate these intervals using a new method not previously published, but whose concept is found in \citet{geyer-gdor} in the penultimate paragraph of Section~3.16.2 and further discussed in Sections 3.6.1--3.6.3 of \citet{infinity}.  Sections~\ref{sec:theory-logistic} contains a description of our method in the context of this example.  The R function \texttt{inference} in R package \texttt{glmdr} computes these one-sided confidence intervals for mean value parameters.

\begin{knitrout}
\definecolor{shadecolor}{rgb}{0.969, 0.969, 0.969}\color{fgcolor}\begin{kframe}
\begin{alltt}
\hlkwd{system.time}\hlstd{(mus.CI} \hlkwb{<-} \hlkwd{inference}\hlstd{(gout))}
\end{alltt}
\begin{verbatim}
##    user  system elapsed 
##   3.589   0.006   3.642
\end{verbatim}
\begin{alltt}
\hlstd{mus.CI}
\end{alltt}
\begin{verbatim}
##   intercept  x y     lower     upper
## 1         1 10 0 0.0000000 0.2852500
## 2         1 20 0 0.0000000 0.3940359
## 3         1 30 0 0.0000000 0.5708292
## 4         1 40 0 0.0000000 0.9499881
## 5         1 60 1 0.0500257 1.0000000
## 6         1 70 1 0.4291708 1.0000000
## 7         1 80 1 0.6059641 1.0000000
## 8         1 90 1 0.7147500 1.0000000
\end{verbatim}
\end{kframe}
\end{knitrout}

Note that for some components of the mean value parameter vector the lower or upper bound of our confidence interval is close to the quick and dirty limit (Section~\ref{sec:quick-and-dirty} below).  In particular, for $x = 40$ the upper bound is close to $0.95$ and for $x = 60$ the lower bound is close to $0.05$.  But for other components of the response vector there are much more restrictive bounds. We now make a plot of these intervals, the following code produces the right panel of Figure~\ref{fig:boundary}.

\begin{knitrout}
\definecolor{shadecolor}{rgb}{0.969, 0.969, 0.969}\color{fgcolor}\begin{kframe}
\begin{alltt}
\hlstd{bounds.lower.p} \hlkwb{<-} \hlstd{mus.CI}\hlopt{$}\hlstd{lower}
\hlstd{bounds.upper.p} \hlkwb{<-} \hlstd{mus.CI}\hlopt{$}\hlstd{upper}
\hlkwd{par}\hlstd{(}\hlkwc{mar} \hlstd{=} \hlkwd{c}\hlstd{(}\hlnum{4}\hlstd{,} \hlnum{4}\hlstd{,} \hlnum{0}\hlstd{,} \hlnum{0}\hlstd{)} \hlopt{+} \hlnum{0.1}\hlstd{)}
\hlkwd{plot}\hlstd{(x, y,} \hlkwc{axes} \hlstd{=} \hlnum{FALSE}\hlstd{,} \hlkwc{type} \hlstd{=} \hlstr{"n"}\hlstd{,}
    \hlkwc{xlab} \hlstd{=} \hlkwd{expression}\hlstd{(x),} \hlkwc{ylab} \hlstd{=} \hlkwd{expression}\hlstd{(}\hlkwd{mu}\hlstd{(x)))}
\hlkwd{segments}\hlstd{(x, bounds.lower.p, x, bounds.upper.p,} \hlkwc{lwd} \hlstd{=} \hlnum{2}\hlstd{)}
\hlkwd{box}\hlstd{()}
\hlkwd{axis}\hlstd{(}\hlkwc{side} \hlstd{=} \hlnum{1}\hlstd{)}
\hlkwd{axis}\hlstd{(}\hlkwc{side} \hlstd{=} \hlnum{2}\hlstd{)}
\hlkwd{points}\hlstd{(x, y,} \hlkwc{pch} \hlstd{=} \hlnum{21}\hlstd{,} \hlkwc{bg} \hlstd{=} \hlstr{"white"}\hlstd{)}
\end{alltt}
\end{kframe}
\end{knitrout}



\subsubsection{Theory for logistic regression}
\label{sec:theory-logistic}

The math of logistic regression is very tricky for the computer. Unless arranged very carefully, the computer may overflow or underflow causing loss of all significant figures. First there is the map from canonical to mean value parameters
$
   p = \logit^{- 1}(\theta)
$
where this inverse logit function operates componentwise
\begin{alignat*}{2}
   p_i & = \frac{e^{\theta_i}}{1 + e^{\theta_i}}
   & = \frac{1}{1 + e^{- \theta_i}}
   \\
   1 - p_i & = \frac{1}{1 + e^{\theta_i}}
   & = \frac{e^{- \theta_i}}{1 + e^{- \theta_i}}
\end{alignat*}
for all $i$. We should always choose one of these formulas for which we know we can have neither overflow, nor catastrophic cancellation.  We always calculate $1 - p_i$ using the second line, we never calculate $p_i$ and subtract from one because this results in catastrophic cancellation when $p_i$ is near one.  If $\theta_i$ is large positive, we choose a formula that has $e^{- \theta_i}$ in it, as that cannot overflow. If $\theta_i$ is large negative, we choose a formula that has $e^{\theta_i}$ in it, as that cannot overflow. If $\theta_i$ is not large, it doesn't matter which we choose.

We also never use the $\log$ function to take logarithms as this can cause horrible inaccuracy when the argument is near one. R has a function \texttt{log1p} that calculates $\log(1 + x)$ accurately for small values of $x$. Note that the map from canonical to mean value parameters gives 
\begin{alignat*}{2}
   \log(p_i) & = \theta_i - \log(1 + e^{\theta_i})
   & & = - \log(1 + e^{- \theta_i})
   \\
   \log(1 - p_i) & = - \log(1 + e^{\theta_i})
   & & = - \theta_i - \log(1 + e^{- \theta_i})
\end{alignat*}
so we calculate
\begin{alignat*}{2}
   \log(p_i) & = \theta_i - \logonep(e^{\theta_i})
   & & = - \logonep(e^{- \theta_i})
   \\
   \log(1 - p_i) & = - \logonep(e^{\theta_i})
   & & = - \theta_i - \logonep(e^{- \theta_i})
\end{alignat*}
With this care, we have a hope of getting approximately correct answers out of the computer. Thus the optimization problem in \eqref{eq:logistic} will be more computational stable written as \eqref{eq:logistic-2}.


Because optimizers expect to optimize over $\R^q$ for some $q$, let $N$ be a matrix whose columns are a basis for $\Gamma_\text{lim}$. In this example $\Gamma_\text{lim}$ is the whole parameter space so $N$ can be the identity matrix.  In other problems we take it to be a matrix whose columns are null eigenvectors of the Fisher information matrix. Then every $\gamma \in \Gamma_\text{lim}$ can be written as $\gamma = N \xi$ for some $\xi \in \R^q$, where $q$ is the column dimension of $N$ and the dimension of $\Gamma_\text{lim}$.

To an optimizer (the \texttt{inference} function in the \texttt{glmdr} package will use the R function \texttt{auglag} in CRAN package \texttt{alabama}) problem \eqref{eq:logistic-2} has the abstract form 
\begin{equation} \label{eq:abstract}
\begin{split}
   \text{minimize} & \quad f(\xi)
   \\
   \text{subject to} & \quad g(\xi) \ge 0
\end{split}
\end{equation}
and the optimization works better if derivatives of $f$ and $g$ are provided. Because R function \texttt{auglag} only does minimization, the objective function must be the negation of what we have in \eqref{eq:logistic-2}.  That is
\begin{align*}
   f(\xi) & = - \theta_k
   \\
   \frac{\partial f(\xi)}{\partial \xi_j}
   & =
   - o_{k j}
   \\
   g(\xi) & =
   \sum_{i \in I} \bigl[ y_i \log(p_i) + (n_i - y_i) \log(1 - p_i) \bigr]
   - \log(\alpha)
   \\
   \frac{\partial g(\xi)}{\partial \xi_j}
   & =
   \sum_{i \in I} (y_i - n_i  p_i) o_{i j}
\end{align*}
where $o_{i j}$ are the components of $O = M N$.

\subsubsection{Quick and dirty intervals}
\label{sec:quick-and-dirty}

As a sanity check and as a quick and dirty conservative (perhaps very
conservative) confidence interval, we note that since all the $p_i$ are
between zero and one we must have
\begin{alignat*}{2}
   p_k^{n_k} & \ge \alpha, & \qquad & y_k = n_k
   \\
   (1 - p_k)^{n_k} & \ge \alpha, & & y_k = 0
\end{alignat*}
or
\begin{alignat*}{2}
   \alpha^{1 / n_k} \le p_k \le 1, & & \qquad & y_k = n_k
   \\
   0 \le p_k \le 1 - \alpha^{1 / n_k}, & & & y_k = 0
\end{alignat*}
For $\alpha = 0.05$ and $n_k = 1$ we have
\begin{align*}
   \alpha^{1 / n_k} & = 0.05
   \\
   1 - \alpha^{1 / n_k} & = 0.95
\end{align*}
In this example,
no upper bound for a one-sided 95\% confidence interval for the
mean value parameter for a cell for which the MLE in the LCM is zero
can be larger than than 0.95 and no lower bound for the analogous
confidence interval for which the MLE in the LCM is one can be
smaller than $0.05$.

\subsection{Support of the submodel canonical statistic}

In this section we duplicate Figure~2 of \citep{infinity}, which is Figure~\ref{fig:boundary} in the main text. The methods of this section take computer time proportional to the size of the sample space.  Hence they can only be used on toy problems and are useless for practical applications.  They do help in understanding the Barndorff-Nielsen completion.

For GLM the (submodel) canonical statistic is $M^T Y$, where $M$ is the model matrix and $y$ is the response vector \citep[Section~3.9]{geyer-gdor}. There are $2^n$ possible values where $n$ is the dimension of the response vector because each component of $y$ can be either zero or one.  The following code makes all of those vectors.

\begin{knitrout}
\definecolor{shadecolor}{rgb}{0.969, 0.969, 0.969}\color{fgcolor}\begin{kframe}
\begin{alltt}
\hlstd{yy} \hlkwb{<-} \hlkwa{NULL}
\hlstd{n} \hlkwb{<-} \hlkwd{length}\hlstd{(y)}
\hlkwa{for} \hlstd{(i} \hlkwa{in} \hlnum{1}\hlopt{:}\hlstd{n) \{}
    \hlstd{j} \hlkwb{<-} \hlnum{2}\hlopt{^}\hlstd{(i} \hlopt{-} \hlnum{1}\hlstd{)}
    \hlstd{k} \hlkwb{<-} \hlnum{2}\hlopt{^}\hlstd{n} \hlopt{/} \hlstd{j} \hlopt{/} \hlnum{2}
    \hlstd{yy} \hlkwb{<-} \hlkwd{cbind}\hlstd{(}\hlkwd{rep}\hlstd{(}\hlkwd{rep}\hlstd{(}\hlnum{0}\hlopt{:}\hlnum{1}\hlstd{,} \hlkwc{each} \hlstd{= j),} \hlkwc{times} \hlstd{= k), yy)}
\hlstd{\}}
\end{alltt}
\end{kframe}
\end{knitrout}

But there are not so many distinct values of the submodel canonical statistic.
\begin{knitrout}
\definecolor{shadecolor}{rgb}{0.969, 0.969, 0.969}\color{fgcolor}\begin{kframe}
\begin{alltt}
\hlstd{m} \hlkwb{<-} \hlkwd{cbind}\hlstd{(}\hlnum{1}\hlstd{, x)}
\hlstd{mtyy} \hlkwb{<-} \hlkwd{t}\hlstd{(m)} \hlopt{%*%} \hlkwd{t}\hlstd{(yy)}
\hlstd{t1} \hlkwb{<-} \hlstd{mtyy[}\hlnum{1}\hlstd{, ]}
\hlstd{t2} \hlkwb{<-} \hlstd{mtyy[}\hlnum{2}\hlstd{, ]}
\hlstd{t1.obs} \hlkwb{<-} \hlkwd{sum}\hlstd{(y)}
\hlstd{t2.obs} \hlkwb{<-} \hlkwd{sum}\hlstd{(x} \hlopt{*} \hlstd{y)}
\end{alltt}
\end{kframe}
\end{knitrout}



The left panel of Figure~\ref{fig:boundary} shows these possible values of the submodel canonical statistic.

\subsection{Linearity by computational geometry}
\label{sec:linearity-example-i-rcdd}

For comparison of computer times and to see that our new methods give
correct results, we redo some of our analysis above using the methods
of \citep{geyer-gdor}.  In this section we find the linearity of the
tangent cone \citep[Sections~3.6 through~3.12]{geyer-gdor}.

The computer code in this section can be found in a technical report
\citep[Section~3.12]{geyer-tech}
cited in \citep{geyer-gdor} and also in the lecture notes \citep{infinity}.

\begin{knitrout}
\definecolor{shadecolor}{rgb}{0.969, 0.969, 0.969}\color{fgcolor}\begin{kframe}
\begin{alltt}
\hlcom{## calling glm to: }
\hlcom{## 1) get model matrix and }
\hlcom{## 2) illustrate that it outputs a warning message when fit to this data}
\hlstd{out} \hlkwb{<-} \hlkwd{glm}\hlstd{(y} \hlopt{~} \hlstd{x,} \hlkwc{family} \hlstd{=} \hlstr{"binomial"}\hlstd{,} \hlkwc{data} \hlstd{= complete,} \hlkwc{x} \hlstd{=} \hlnum{TRUE}\hlstd{)}
\end{alltt}

{\ttfamily\noindent\color{warningcolor}{\#\# Warning: glm.fit: fitted probabilities numerically 0 or 1 occurred}}\begin{alltt}
\hlstd{tanv} \hlkwb{<-} \hlstd{modmat} \hlkwb{<-} \hlstd{out}\hlopt{$}\hlstd{x}
\hlstd{tanv[y} \hlopt{==} \hlnum{1}\hlstd{, ]} \hlkwb{<-} \hlstd{(}\hlopt{-}\hlstd{tanv[y} \hlopt{==} \hlnum{1}\hlstd{, ])}
\hlstd{vrep} \hlkwb{<-} \hlkwd{makeV}\hlstd{(}\hlkwc{rays} \hlstd{= tanv)}
\hlkwd{system.time}\hlstd{(lout} \hlkwb{<-} \hlkwd{linearity}\hlstd{(}\hlkwd{d2q}\hlstd{(vrep),} \hlkwc{rep} \hlstd{=} \hlstr{"V"}\hlstd{))}
\end{alltt}
\begin{verbatim}
##    user  system elapsed 
##   0.012   0.000   0.012
\end{verbatim}
\begin{alltt}
\hlstd{lout}
\end{alltt}
\begin{verbatim}
## integer(0)
\end{verbatim}
\end{kframe}
\end{knitrout}

R object \texttt{lout} is the set of indices of components of the
response vector that do not have a degenerate distribution in the LCM.
In this example it has length zero indicating that the LCM is completely
degenerate.  This agrees with our analysis in
Section~\ref{sec:MLEinLCM-i} above.

Unlike the analysis using our new methods in
Section~\ref{sec:MLEinLCM-i} above,
the analysis in this section using R package \texttt{rcdd} is guaranteed
to be correct --- as valid as any mathematical proof --- because
the functions in that package can use infinite precision rational arithmetic
(R function \texttt{linearity} is doing so in the code chunk above).

Although the analysis in this section takes a trivial amount of computer
time on this toy problem, it does not scale.  It takes days of computer
time on the example in Section~\ref{sec:big-data} below.  Our new methods
do scale.

\subsection{Generic direction of recession}
\label{sec:gdor-example-i-rcdd}

The main theoretical tool of \citet{geyer-gdor} is the notion of
a \emph{generic direction of recession} (GDOR)
\citep[Sections~3.3 through~3.13]{geyer-gdor}.
But our new methods of calculation do not need to refer to it.
(We only need to get the correct linearity using eigenvalues and
eigenvectors of the Fisher information matrix.)

The code chunk below comes from the technical report
\citep[Section~4.1]{geyer-tech} and also from the lecture notes
\citep{infinity}.
\begin{knitrout}
\definecolor{shadecolor}{rgb}{0.969, 0.969, 0.969}\color{fgcolor}\begin{kframe}
\begin{alltt}
\hlstd{p} \hlkwb{<-} \hlkwd{ncol}\hlstd{(tanv)}
\hlstd{hrep} \hlkwb{<-} \hlkwd{cbind}\hlstd{(}\hlnum{0}\hlstd{,} \hlnum{0}\hlstd{,} \hlopt{-}\hlstd{tanv,} \hlopt{-}\hlnum{1}\hlstd{)}
\hlstd{hrep} \hlkwb{<-} \hlkwd{rbind}\hlstd{(hrep,} \hlkwd{c}\hlstd{(}\hlnum{0}\hlstd{,} \hlnum{1}\hlstd{,} \hlkwd{rep}\hlstd{(}\hlnum{0}\hlstd{, p),} \hlopt{-}\hlnum{1}\hlstd{))}
\hlstd{objv} \hlkwb{<-} \hlkwd{c}\hlstd{(}\hlkwd{rep}\hlstd{(}\hlnum{0}\hlstd{, p),} \hlnum{1}\hlstd{)}
\hlstd{pout} \hlkwb{<-} \hlkwd{lpcdd}\hlstd{(}\hlkwd{d2q}\hlstd{(hrep),} \hlkwd{d2q}\hlstd{(objv),} \hlkwc{minimize} \hlstd{=} \hlnum{FALSE}\hlstd{)}
\hlkwd{names}\hlstd{(pout)}
\end{alltt}
\begin{verbatim}
## [1] "solution.type"   "primal.solution" "dual.solution"  
## [4] "optimal.value"
\end{verbatim}
\begin{alltt}
\hlstd{pout}\hlopt{$}\hlstd{solution.type}
\end{alltt}
\begin{verbatim}
## [1] "Optimal"
\end{verbatim}
\begin{alltt}
\hlstd{gdor} \hlkwb{<-} \hlstd{pout}\hlopt{$}\hlstd{primal.solution[}\hlnum{1}\hlopt{:}\hlstd{p]}
\hlstd{gdor}
\end{alltt}
\begin{verbatim}
## [1] "-5"   "1/10"
\end{verbatim}
\begin{alltt}
\hlstd{pout}\hlopt{$}\hlstd{optimal.value}
\end{alltt}
\begin{verbatim}
## [1] "1"
\end{verbatim}
\end{kframe}
\end{knitrout}
The code chunk above is not general.  It assumes the linearity is trivial,
as in the particular example we are working on.
Other examples below will have more general code.
This agrees with the calculation in \citep[Section~3.3]{infinity}.

The fact that a GDOR exists shows that our calculation of the linearity
was correct (no matter how it was done).  That a GDOR exists is shown
by \verb@pout$solution.type@ being \texttt{"Optimal"}
by \verb@pout$optimal.value@ being strictly positive.

Clean R global environment.
\begin{knitrout}
\definecolor{shadecolor}{rgb}{0.969, 0.969, 0.969}\color{fgcolor}\begin{kframe}
\begin{alltt}
\hlkwd{rm}\hlstd{(}\hlkwc{list} \hlstd{=} \hlkwd{ls}\hlstd{())}
\end{alltt}
\end{kframe}
\end{knitrout}

\section{Complete separation example of Geyer}
\label{sec:geyer-complete}

This is the example in Section~2.2 of \citep{geyer-gdor}.  Its behavior
is very similar to that of the preceding example.  The only difference
is that this does quadratic logistic regression instead of linear logistic
regression.

\subsection{Data}

Data
\begin{knitrout}
\definecolor{shadecolor}{rgb}{0.969, 0.969, 0.969}\color{fgcolor}\begin{kframe}
\begin{alltt}
\hlstd{x} \hlkwb{<-} \hlnum{1}\hlopt{:}\hlnum{30}
\hlstd{y} \hlkwb{<-} \hlkwd{c}\hlstd{(}\hlkwd{rep}\hlstd{(}\hlnum{0}\hlstd{,} \hlnum{12}\hlstd{),} \hlkwd{rep}\hlstd{(}\hlnum{1}\hlstd{,} \hlnum{11}\hlstd{),} \hlkwd{rep}\hlstd{(}\hlnum{0}\hlstd{,} \hlnum{7}\hlstd{))}
\end{alltt}
\end{kframe}
\end{knitrout}

These data are included in the \texttt{glmdr} package.

\begin{knitrout}
\definecolor{shadecolor}{rgb}{0.969, 0.969, 0.969}\color{fgcolor}\begin{kframe}
\begin{alltt}
\hlkwd{data}\hlstd{(quadratic)}
\hlkwd{all.equal}\hlstd{(quadratic,} \hlkwd{as.data.frame}\hlstd{(}\hlkwd{cbind}\hlstd{(x,y)))}
\end{alltt}
\begin{verbatim}
## [1] TRUE
\end{verbatim}
\end{kframe}
\end{knitrout}

\subsection{The MLE in the LCM}
\label{sec:MLEinLCM-iii}

The LCM is completely degenerate and has no identifiable parameters.
We fit these data using R function \texttt{glmdr}.

\begin{knitrout}
\definecolor{shadecolor}{rgb}{0.969, 0.969, 0.969}\color{fgcolor}\begin{kframe}
\begin{alltt}
\hlstd{gout} \hlkwb{<-} \hlkwd{glmdr}\hlstd{(y} \hlopt{~} \hlstd{x} \hlopt{+} \hlkwd{I}\hlstd{(x}\hlopt{^}\hlnum{2}\hlstd{),} \hlkwc{family} \hlstd{=} \hlstr{"binomial"}\hlstd{,} \hlkwc{data} \hlstd{= quadratic)}
\hlkwd{summary}\hlstd{(gout)}
\end{alltt}
\begin{verbatim}
## 
## MLE exists in Barndorff-Nielsen completion 
## it is completely degenerate 
## the MLE says the response actually observed is the only 
## possible value that could ever be observed
\end{verbatim}
\end{kframe}
\end{knitrout}

\subsection{One-sided confidence intervals for mean value parameters}

We now provide one-sided confidence intervals for mean value parameters 
whose MLE is on the boundary.  

\begin{knitrout}
\definecolor{shadecolor}{rgb}{0.969, 0.969, 0.969}\color{fgcolor}\begin{kframe}
\begin{alltt}
\hlstd{mus.CI} \hlkwb{<-} \hlkwd{inference}\hlstd{(gout)}
\hlstd{mus.CI}
\end{alltt}
\begin{verbatim}
##    intercept  x I.x.2. y      lower        upper
## 1          1  1      1 0 0.00000000 6.563500e-12
## 2          1  2      4 0 0.00000000 1.915859e-10
## 3          1  3      9 0 0.00000000 4.570606e-09
## 4          1  4     16 0 0.00000000 8.919127e-08
## 5          1  5     25 0 0.00000000 1.425473e-06
## 6          1  6     36 0 0.00000000 1.869697e-05
## 7          1  7     49 0 0.00000000 2.019500e-04
## 8          1  8     64 0 0.00000000 1.806200e-03
## 9          1  9     81 0 0.00000000 1.345421e-02
## 10         1 10    100 0 0.00000000 8.242264e-02
## 11         1 11    121 0 0.00000000 3.741233e-01
## 12         1 12    144 0 0.00000000 9.481161e-01
## 13         1 13    169 1 0.05330860 1.000000e+00
## 14         1 14    196 1 0.65501216 1.000000e+00
## 15         1 15    225 1 0.86723176 1.000000e+00
## 16         1 16    256 1 0.92920573 1.000000e+00
## 17         1 17    289 1 0.95066373 1.000000e+00
## 18         1 18    324 1 0.95616871 1.000000e+00
## 19         1 19    361 1 0.95066368 1.000000e+00
## 20         1 20    400 1 0.92920575 1.000000e+00
## 21         1 21    441 1 0.86723190 1.000000e+00
## 22         1 22    484 1 0.65501207 1.000000e+00
## 23         1 23    529 1 0.05350799 1.000000e+00
## 24         1 24    576 0 0.00000000 9.479931e-01
## 25         1 25    625 0 0.00000000 3.741229e-01
## 26         1 26    676 0 0.00000000 8.242262e-02
## 27         1 27    729 0 0.00000000 1.345423e-02
## 28         1 28    784 0 0.00000000 1.806203e-03
## 29         1 29    841 0 0.00000000 2.019507e-04
## 30         1 30    900 0 0.00000000 1.869705e-05
\end{verbatim}
\end{kframe}
\end{knitrout}

Note that for some cells of the mean value parameter vector the lower or upper bound of
our confidence interval is close to the quick and dirty limit. (Section~\ref{sec:quick-and-dirty} above). In particular, for $x = 12$ and $x = 24$ the upper bound is close to 0.95 and for $x = 13$ and $x = 23$ the lower bound is close to $0.05$. But for other components of the response vector there are much more restrictive bounds. We now make a plot of these intervals.
\begin{knitrout}
\definecolor{shadecolor}{rgb}{0.969, 0.969, 0.969}\color{fgcolor}\begin{kframe}
\begin{alltt}
\hlstd{bounds.lower.p} \hlkwb{<-} \hlstd{mus.CI}\hlopt{$}\hlstd{lower}
\hlstd{bounds.upper.p} \hlkwb{<-} \hlstd{mus.CI}\hlopt{$}\hlstd{upper}
\hlkwd{par}\hlstd{(}\hlkwc{mar} \hlstd{=} \hlkwd{c}\hlstd{(}\hlnum{4}\hlstd{,} \hlnum{4}\hlstd{,} \hlnum{0}\hlstd{,} \hlnum{0}\hlstd{)} \hlopt{+} \hlnum{0.1}\hlstd{)}
\hlkwd{plot}\hlstd{(x, y,} \hlkwc{axes} \hlstd{=} \hlnum{FALSE}\hlstd{,} \hlkwc{type} \hlstd{=} \hlstr{"n"}\hlstd{,}
    \hlkwc{xlab} \hlstd{=} \hlkwd{expression}\hlstd{(x),} \hlkwc{ylab} \hlstd{=} \hlkwd{expression}\hlstd{(}\hlkwd{mu}\hlstd{(x)))}
\hlkwd{segments}\hlstd{(x, bounds.lower.p, x, bounds.upper.p,} \hlkwc{lwd} \hlstd{=} \hlnum{2}\hlstd{)}
\hlkwd{box}\hlstd{()}
\hlkwd{axis}\hlstd{(}\hlkwc{side} \hlstd{=} \hlnum{1}\hlstd{)}
\hlkwd{axis}\hlstd{(}\hlkwc{side} \hlstd{=} \hlnum{2}\hlstd{)}
\hlkwd{points}\hlstd{(x, y,} \hlkwc{pch} \hlstd{=} \hlnum{21}\hlstd{,} \hlkwc{bg} \hlstd{=} \hlstr{"white"}\hlstd{)}
\end{alltt}
\end{kframe}
\end{knitrout}

Our Figure~\ref{fig:three} agrees with Figure~2 in \citep{geyer-gdor}, which was done by methods that are much more messy and made obsolete by the methods presented here.
\begin{figure}
\begin{center}
\begin{knitrout}
\definecolor{shadecolor}{rgb}{0.969, 0.969, 0.969}\color{fgcolor}
\includegraphics[width = 0.75\textwidth]{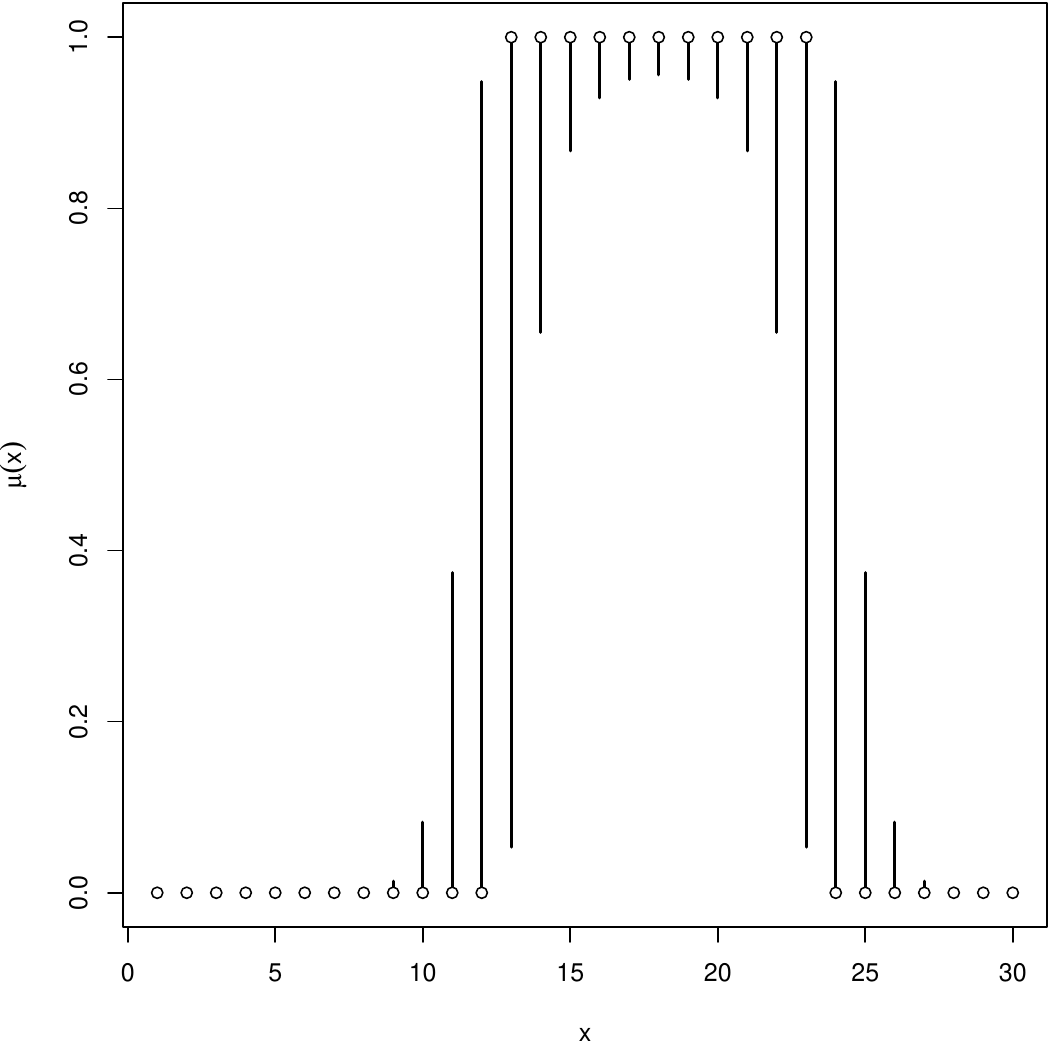} 

\end{knitrout}
\end{center}
\caption{One-sided 95\% confidence intervals for mean
  value parameters.  Bars are the intervals.  Vertical axis is the probability
  of observing response value one when the predictor value is $x$.
  Solid dots are the observed data.}
\label{fig:three}
\end{figure}

\subsection{Linearity by computational geometry}

Calculate linearity using R package \texttt{rcdd} like
in Section~\ref{sec:linearity-example-i-rcdd} above.
\begin{knitrout}
\definecolor{shadecolor}{rgb}{0.969, 0.969, 0.969}\color{fgcolor}\begin{kframe}
\begin{alltt}
\hlcom{## calling glm to: }
\hlcom{## 1) get model matrix and }
\hlcom{## 2) illustrate that it outputs a warning message when fit to this data}
\hlstd{out} \hlkwb{<-} \hlkwd{glm}\hlstd{(y} \hlopt{~} \hlstd{x} \hlopt{+} \hlkwd{I}\hlstd{(x}\hlopt{^}\hlnum{2}\hlstd{),} \hlkwc{family} \hlstd{=} \hlstr{"binomial"}\hlstd{,}
  \hlkwc{data} \hlstd{= quadratic,} \hlkwc{x} \hlstd{=} \hlnum{TRUE}\hlstd{)}
\end{alltt}

{\ttfamily\noindent\color{warningcolor}{\#\# Warning: glm.fit: algorithm did not converge}}

{\ttfamily\noindent\color{warningcolor}{\#\# Warning: glm.fit: fitted probabilities numerically 0 or 1 occurred}}\begin{alltt}
\hlstd{tanv} \hlkwb{<-} \hlstd{modmat} \hlkwb{<-} \hlstd{out}\hlopt{$}\hlstd{x}
\hlstd{tanv[y} \hlopt{==} \hlnum{1}\hlstd{, ]} \hlkwb{<-} \hlstd{(}\hlopt{-}\hlstd{tanv[y} \hlopt{==} \hlnum{1}\hlstd{, ])}
\hlstd{vrep} \hlkwb{<-} \hlkwd{makeV}\hlstd{(}\hlkwc{rays} \hlstd{= tanv)}
\hlstd{lout} \hlkwb{<-} \hlkwd{linearity}\hlstd{(}\hlkwd{d2q}\hlstd{(vrep),} \hlkwc{rep} \hlstd{=} \hlstr{"V"}\hlstd{)}
\hlstd{lout}
\end{alltt}
\begin{verbatim}
## integer(0)
\end{verbatim}
\end{kframe}
\end{knitrout}

So this agrees with our analysis in Section~\ref{sec:MLEinLCM-iii} above.

\subsection{Generic direction of recession}

Calculate a GDOR using R package \texttt{rcdd} like
in Section~\ref{sec:gdor-example-i-rcdd} above.
\begin{knitrout}
\definecolor{shadecolor}{rgb}{0.969, 0.969, 0.969}\color{fgcolor}\begin{kframe}
\begin{alltt}
\hlstd{p} \hlkwb{<-} \hlkwd{ncol}\hlstd{(tanv)}
\hlstd{hrep} \hlkwb{<-} \hlkwd{cbind}\hlstd{(}\hlnum{0}\hlstd{,} \hlnum{0}\hlstd{,} \hlopt{-}\hlstd{tanv,} \hlopt{-}\hlnum{1}\hlstd{)}
\hlstd{hrep} \hlkwb{<-} \hlkwd{rbind}\hlstd{(hrep,} \hlkwd{c}\hlstd{(}\hlnum{0}\hlstd{,} \hlnum{1}\hlstd{,} \hlkwd{rep}\hlstd{(}\hlnum{0}\hlstd{, p),} \hlopt{-}\hlnum{1}\hlstd{))}
\hlstd{objv} \hlkwb{<-} \hlkwd{c}\hlstd{(}\hlkwd{rep}\hlstd{(}\hlnum{0}\hlstd{, p),} \hlnum{1}\hlstd{)}
\hlstd{pout} \hlkwb{<-} \hlkwd{lpcdd}\hlstd{(}\hlkwd{d2q}\hlstd{(hrep),} \hlkwd{d2q}\hlstd{(objv),} \hlkwc{minimize} \hlstd{=} \hlnum{FALSE}\hlstd{)}
\hlkwd{names}\hlstd{(pout)}
\end{alltt}
\begin{verbatim}
## [1] "solution.type"   "primal.solution" "dual.solution"  
## [4] "optimal.value"
\end{verbatim}
\begin{alltt}
\hlstd{pout}\hlopt{$}\hlstd{solution.type}
\end{alltt}
\begin{verbatim}
## [1] "Optimal"
\end{verbatim}
\begin{alltt}
\hlstd{gdor} \hlkwb{<-} \hlstd{pout}\hlopt{$}\hlstd{primal.solution[}\hlnum{1}\hlopt{:}\hlstd{p]}
\hlstd{gdor}
\end{alltt}
\begin{verbatim}
## [1] "-587/11" "72/11"   "-2/11"
\end{verbatim}
\begin{alltt}
\hlstd{pout}\hlopt{$}\hlstd{optimal.value}
\end{alltt}
\begin{verbatim}
## [1] "1"
\end{verbatim}
\end{kframe}
\end{knitrout}
This agrees with the GDOR found in the technical report \citep[Section~4.1]{geyer-tech} that is supplementary material for \citep{geyer-gdor}. We clean R global environment. 
\begin{knitrout}
\definecolor{shadecolor}{rgb}{0.969, 0.969, 0.969}\color{fgcolor}\begin{kframe}
\begin{alltt}
\hlkwd{rm}\hlstd{(}\hlkwc{list} \hlstd{=} \hlkwd{ls}\hlstd{())}
\end{alltt}
\end{kframe}
\end{knitrout}

\section{Sports standings example of Geyer}
\label{sec:geyer-sports}

This is the example in Section~2.4 of \citet{geyer-gdor}.  Its behavior
is different from any of the preceding examples, because the
LCM is not completely degenerate and also because the binomial sample
size is two for all components of the response vector.

\pagebreak[3]
\subsection{Data}

Data
\begin{knitrout}
\definecolor{shadecolor}{rgb}{0.969, 0.969, 0.969}\color{fgcolor}\begin{kframe}
\begin{alltt}
\hlstd{team.names} \hlkwb{<-} \hlkwd{c}\hlstd{(}\hlstr{"ants"}\hlstd{,} \hlstr{"beetles"}\hlstd{,} \hlstr{"cows"}\hlstd{,} \hlstr{"dogs"}\hlstd{,}
    \hlstr{"egrets"}\hlstd{,} \hlstr{"foxes"}\hlstd{,} \hlstr{"gerbils"}\hlstd{,} \hlstr{"hogs"}\hlstd{)}
\hlstd{data} \hlkwb{<-} \hlkwd{matrix}\hlstd{(}\hlkwd{c}\hlstd{(}\hlnum{NA}\hlstd{,} \hlnum{2}\hlstd{,} \hlnum{2}\hlstd{,} \hlnum{2}\hlstd{,} \hlnum{2}\hlstd{,} \hlnum{2}\hlstd{,} \hlnum{2}\hlstd{,} \hlnum{2}\hlstd{,} \hlnum{0}\hlstd{,} \hlnum{NA}\hlstd{,}
    \hlnum{1}\hlstd{,} \hlnum{2}\hlstd{,} \hlnum{2}\hlstd{,} \hlnum{2}\hlstd{,} \hlnum{2}\hlstd{,} \hlnum{2}\hlstd{,} \hlnum{0}\hlstd{,} \hlnum{1}\hlstd{,} \hlnum{NA}\hlstd{,} \hlnum{2}\hlstd{,} \hlnum{1}\hlstd{,} \hlnum{2}\hlstd{,} \hlnum{2}\hlstd{,} \hlnum{2}\hlstd{,} \hlnum{0}\hlstd{,}
    \hlnum{0}\hlstd{,} \hlnum{0}\hlstd{,} \hlnum{NA}\hlstd{,} \hlnum{1}\hlstd{,} \hlnum{1}\hlstd{,} \hlnum{2}\hlstd{,} \hlnum{2}\hlstd{,} \hlnum{0}\hlstd{,} \hlnum{0}\hlstd{,} \hlnum{1}\hlstd{,} \hlnum{1}\hlstd{,} \hlnum{NA}\hlstd{,} \hlnum{1}\hlstd{,} \hlnum{2}\hlstd{,} \hlnum{2}\hlstd{,}
    \hlnum{0}\hlstd{,} \hlnum{0}\hlstd{,} \hlnum{0}\hlstd{,} \hlnum{1}\hlstd{,} \hlnum{1}\hlstd{,} \hlnum{NA}\hlstd{,} \hlnum{2}\hlstd{,} \hlnum{2}\hlstd{,} \hlnum{0}\hlstd{,} \hlnum{0}\hlstd{,} \hlnum{0}\hlstd{,} \hlnum{0}\hlstd{,} \hlnum{0}\hlstd{,} \hlnum{0}\hlstd{,} \hlnum{NA}\hlstd{,}
    \hlnum{1}\hlstd{,} \hlnum{0}\hlstd{,} \hlnum{0}\hlstd{,} \hlnum{0}\hlstd{,} \hlnum{0}\hlstd{,} \hlnum{0}\hlstd{,} \hlnum{0}\hlstd{,} \hlnum{1}\hlstd{,} \hlnum{NA}\hlstd{),} \hlkwc{byrow} \hlstd{=} \hlnum{TRUE}\hlstd{,} \hlkwc{nrow} \hlstd{=} \hlnum{8}\hlstd{)}
\hlkwd{dimnames}\hlstd{(data)} \hlkwb{<-} \hlkwd{list}\hlstd{(team.names, team.names)}
\hlkwd{print}\hlstd{(data)}
\end{alltt}
\begin{verbatim}
##         ants beetles cows dogs egrets foxes gerbils hogs
## ants      NA       2    2    2      2     2       2    2
## beetles    0      NA    1    2      2     2       2    2
## cows       0       1   NA    2      1     2       2    2
## dogs       0       0    0   NA      1     1       2    2
## egrets     0       0    1    1     NA     1       2    2
## foxes      0       0    0    1      1    NA       2    2
## gerbils    0       0    0    0      0     0      NA    1
## hogs       0       0    0    0      0     0       1   NA
\end{verbatim}
\end{kframe}
\end{knitrout}

We model these data with Bradley-Terry model.  We code this differently
from the technical report \citep{geyer-tech} accompanying \cite{geyer-gdor}.

First we format the data the way R function \texttt{glm} likes
(in a data.frame).
\begin{knitrout}
\definecolor{shadecolor}{rgb}{0.969, 0.969, 0.969}\color{fgcolor}\begin{kframe}
\begin{alltt}
\hlstd{wins} \hlkwb{<-} \hlstd{data[}\hlkwd{upper.tri}\hlstd{(data)]}
\hlstd{team.plus} \hlkwb{<-} \hlkwd{row}\hlstd{(data)[}\hlkwd{upper.tri}\hlstd{(data)]}
\hlstd{team.minus} \hlkwb{<-} \hlkwd{col}\hlstd{(data)[}\hlkwd{upper.tri}\hlstd{(data)]}
\hlstd{modmat} \hlkwb{<-} \hlkwd{matrix}\hlstd{(}\hlnum{0}\hlstd{,} \hlkwd{length}\hlstd{(wins),} \hlkwd{nrow}\hlstd{(data))}
\hlkwa{for} \hlstd{(i} \hlkwa{in} \hlnum{1}\hlopt{:}\hlkwd{ncol}\hlstd{(modmat)) \{}
    \hlstd{modmat[team.plus} \hlopt{==} \hlstd{i, i]} \hlkwb{<-} \hlnum{1}
    \hlstd{modmat[team.minus} \hlopt{==} \hlstd{i, i]} \hlkwb{<-} \hlstd{(}\hlopt{-}\hlnum{1}\hlstd{)}
\hlstd{\}}
\hlstd{losses} \hlkwb{<-} \hlnum{2} \hlopt{-} \hlstd{wins}
\hlstd{resp} \hlkwb{<-} \hlkwd{cbind}\hlstd{(wins, losses)}

\hlkwd{colnames}\hlstd{(modmat)} \hlkwb{<-} \hlstd{team.names}
\hlstd{sportsdata} \hlkwb{<-} \hlkwd{cbind}\hlstd{(modmat, wins, losses)}
\hlstd{sportsdata} \hlkwb{<-} \hlkwd{as.data.frame}\hlstd{(sportsdata)}
\end{alltt}
\end{kframe}
\end{knitrout}

These data are included in the \texttt{glmdr} package.

\begin{knitrout}
\definecolor{shadecolor}{rgb}{0.969, 0.969, 0.969}\color{fgcolor}\begin{kframe}
\begin{alltt}
\hlkwd{data}\hlstd{(sports)}
\hlkwd{all.equal}\hlstd{(sports, sportsdata)}
\end{alltt}
\begin{verbatim}
## [1] TRUE
\end{verbatim}
\end{kframe}
\end{knitrout}

\subsection{Fitting the Model}
\label{sec:sports-glmdr}

We first fit the model using the R function \texttt{glmdr}.

\begin{knitrout}
\definecolor{shadecolor}{rgb}{0.969, 0.969, 0.969}\color{fgcolor}\begin{kframe}
\begin{alltt}
\hlstd{gout} \hlkwb{<-} \hlkwd{glmdr}\hlstd{(}\hlkwd{cbind}\hlstd{(wins, losses)} \hlopt{~} \hlnum{0} \hlopt{+} \hlstd{.,}
  \hlkwc{family} \hlstd{=} \hlstr{"binomial"}\hlstd{,} \hlkwc{data} \hlstd{= sports)}
\hlkwd{summary}\hlstd{(gout)}
\end{alltt}
\begin{verbatim}
## 
## MLE exists in Barndorff-Nielsen completion 
## it is conditional on components of the response 
## corresponding to object$linearity == FALSE being 
## conditioned on their observed values 
## 
## GLM summary for limiting conditional model
## 
## 
## Call:
## stats::glm(formula = cbind(wins, losses) ~ 0 + ., family = "binomial", 
##     data = sports, subset = c("3", "5", "6", "8", "9", "10", 
##     "12", "13", "14", "15", "28"), x = TRUE, y = TRUE)
## 
## Deviance Residuals: 
##     Min       1Q   Median       3Q      Max  
## -1.1692  -0.1970   0.3941   0.5038   0.6153  
## 
## Coefficients: (3 not defined because of singularities)
##           Estimate Std. Error z value Pr(>|z|)  
## ants            NA         NA      NA       NA  
## beetles  3.024e+00  1.487e+00   2.034   0.0419 *
## cows     2.310e+00  1.328e+00   1.740   0.0819 .
## dogs    -5.189e-17  1.080e+00   0.000   1.0000  
## egrets   5.609e-01  1.078e+00   0.520   0.6029  
## foxes           NA         NA      NA       NA  
## gerbils  0.000e+00  1.414e+00   0.000   1.0000  
## hogs            NA         NA      NA       NA  
## ---
## Signif. codes:  
## 0 '***' 0.001 '**' 0.01 '*' 0.05 '.' 0.1 ' ' 1
## 
## (Dispersion parameter for binomial family taken to be 1)
## 
##     Null deviance: 13.863  on 11  degrees of freedom
## Residual deviance:  3.391  on  6  degrees of freedom
## AIC: 21.709
## 
## Number of Fisher Scoring iterations: 5
\end{verbatim}
\end{kframe}
\end{knitrout}

\subsection{Linearity}
\label{sec:sports-linearity}

As explained in Section 6.3 of the main text, the components of the response vector that are random in the LCM are those for which the null space projected to canonical parameter space of the saturated model have corresponding zeros. These components are those for which the \texttt{linearity} of the object returned by R function \texttt{glmdr} is \texttt{true}
\begin{knitrout}
\definecolor{shadecolor}{rgb}{0.969, 0.969, 0.969}\color{fgcolor}\begin{kframe}
\begin{alltt}
\hlstd{gout}\hlopt{$}\hlstd{linearity}
\end{alltt}
\begin{verbatim}
##     1     2     3     4     5     6     7     8     9    10 
## FALSE FALSE  TRUE FALSE  TRUE  TRUE FALSE  TRUE  TRUE  TRUE 
##    11    12    13    14    15    16    17    18    19    20 
## FALSE  TRUE  TRUE  TRUE  TRUE FALSE FALSE FALSE FALSE FALSE 
##    21    22    23    24    25    26    27    28 
## FALSE FALSE FALSE FALSE FALSE FALSE FALSE  TRUE
\end{verbatim}
\end{kframe}
\end{knitrout}

\subsection{Labels}
\label{sec:sports-labels}

We now want to make some confidence intervals, but first we make some
short labels for components of the response vector.
\begin{knitrout}
\definecolor{shadecolor}{rgb}{0.969, 0.969, 0.969}\color{fgcolor}\begin{kframe}
\begin{alltt}
\hlstd{foo} \hlkwb{<-} \hlstd{sports[ ,} \hlopt{!} \hlstd{(}\hlkwd{colnames}\hlstd{(sports)} \hlopt{%in%} \hlkwd{c}\hlstd{(}\hlstr{"wins"}\hlstd{,} \hlstr{"losses"}\hlstd{))]}
\hlstd{teams} \hlkwb{<-} \hlkwd{colnames}\hlstd{(foo)}
\hlstd{winner} \hlkwb{<-} \hlkwd{apply}\hlstd{(foo} \hlopt{==} \hlnum{1}\hlstd{,} \hlnum{1}\hlstd{,} \hlkwa{function}\hlstd{(}\hlkwc{x}\hlstd{) teams[x])}
\hlstd{loser} \hlkwb{<-} \hlkwd{apply}\hlstd{(foo} \hlopt{== -}\hlnum{1}\hlstd{,} \hlnum{1}\hlstd{,} \hlkwa{function}\hlstd{(}\hlkwc{x}\hlstd{) teams[x])}
\hlstd{label} \hlkwb{<-} \hlkwd{paste}\hlstd{(winner,} \hlstr{"beat"}\hlstd{, loser)}
\hlkwd{head}\hlstd{(label)}
\end{alltt}
\begin{verbatim}
## [1] "ants beat beetles" "ants beat cows"   
## [3] "beetles beat cows" "ants beat dogs"   
## [5] "beetles beat dogs" "cows beat dogs"
\end{verbatim}
\end{kframe}
\end{knitrout}

\subsection{Confidence Intervals}
\label{sec:sports-confint}

We now want to fit confidence intervals.  These come in two kinds.
First, there are confidence intervals for means of components of the response
vector that are in the linearity.  These are the usual sort of confidence
intervals for GLM, based on asymptotics, and produced by the \texttt{glm}
method of the R generic function \texttt{predict}.
Second, there are confidence intervals for means of components of the response
vector that are not in the linearity.  These are non-asymptotic
intervals, described in Section~\ref{sec:one-sided},
and produced by R function \texttt{inference} in R package \texttt{glmdr}.
These latter intervals are necessarily one-sided because the MLE mean
value parameter estimates for these components of the response vector
are on the boundary of the range of possible values.

\subsubsection{Two-Sided Intervals}

We get estimated means and standard errors as follows.
\begin{knitrout}
\definecolor{shadecolor}{rgb}{0.969, 0.969, 0.969}\color{fgcolor}\begin{kframe}
\begin{alltt}
\hlstd{preds} \hlkwb{<-} \hlkwd{predict}\hlstd{(gout}\hlopt{$}\hlstd{lcm,} \hlkwc{type} \hlstd{=} \hlstr{"response"}\hlstd{,} \hlkwc{se.fit} \hlstd{=} \hlnum{TRUE}\hlstd{)}
\hlstd{preds.tab} \hlkwb{<-} \hlkwd{cbind}\hlstd{(preds}\hlopt{$}\hlstd{fit, preds}\hlopt{$}\hlstd{se.fit)}
\hlkwd{colnames}\hlstd{(preds.tab)} \hlkwb{<-} \hlkwd{c}\hlstd{(}\hlstr{"fit"}\hlstd{,} \hlstr{"se"}\hlstd{)}
\hlkwd{rownames}\hlstd{(preds.tab)} \hlkwb{<-} \hlstd{label[gout}\hlopt{$}\hlstd{linearity]}
\hlkwd{round}\hlstd{(preds.tab,} \hlnum{3}\hlstd{)}
\end{alltt}
\begin{verbatim}
##                       fit    se
## beetles beat cows   0.671 0.274
## beetles beat dogs   0.954 0.066
## cows beat dogs      0.910 0.109
## beetles beat egrets 0.921 0.103
## cows beat egrets    0.852 0.159
## dogs beat egrets    0.363 0.249
## beetles beat foxes  0.954 0.066
## cows beat foxes     0.910 0.109
## dogs beat foxes     0.500 0.270
## egrets beat foxes   0.637 0.249
## gerbils beat hogs   0.500 0.354
\end{verbatim}
\end{kframe}
\end{knitrout}

And turn this into 95\% confidence intervals as follows.
\begin{knitrout}
\definecolor{shadecolor}{rgb}{0.969, 0.969, 0.969}\color{fgcolor}\begin{kframe}
\begin{alltt}
\hlstd{ci.tab} \hlkwb{<-} \hlkwd{apply}\hlstd{(preds.tab,} \hlnum{1}\hlstd{,} \hlkwa{function}\hlstd{(}\hlkwc{x}\hlstd{) x[}\hlnum{1}\hlstd{]} \hlopt{+} \hlkwd{c}\hlstd{(}\hlopt{-}\hlnum{1}\hlstd{,}\hlnum{1}\hlstd{)} \hlopt{*} \hlkwd{qnorm}\hlstd{(}\hlnum{0.975}\hlstd{)} \hlopt{*} \hlstd{x[}\hlnum{2}\hlstd{])}
\hlstd{ci.tab} \hlkwb{<-} \hlkwd{t}\hlstd{(ci.tab)}
\hlkwd{colnames}\hlstd{(ci.tab)} \hlkwb{<-} \hlkwd{c}\hlstd{(}\hlstr{"lwr"}\hlstd{,} \hlstr{"upr"}\hlstd{)}
\hlkwd{round}\hlstd{(ci.tab,} \hlnum{3}\hlstd{)}
\end{alltt}
\begin{verbatim}
##                        lwr   upr
## beetles beat cows    0.134 1.208
## beetles beat dogs    0.825 1.082
## cows beat dogs       0.696 1.123
## beetles beat egrets  0.720 1.123
## cows beat egrets     0.541 1.163
## dogs beat egrets    -0.126 0.852
## beetles beat foxes   0.825 1.082
## cows beat foxes      0.696 1.123
## dogs beat foxes     -0.029 1.029
## egrets beat foxes    0.148 1.126
## gerbils beat hogs   -0.193 1.193
\end{verbatim}
\end{kframe}
\end{knitrout}

As always, there is no reason why Wald confidence intervals cannot go
outside the boundaries of the parameter space, as some of these intervals do.
As noted in the discussion of \citep{geyer-gdor}, the sample sizes here are
by no means ``large''.  The last confidence interval (gerbils versus hogs)
is based on exactly two games (these teams played two games and each won one,
no other games are relevant to this inference).  So for these data,
the confidence intervals produced in this section are of questionable validity.

\subsubsection{One-Sided Intervals}

We get one-sided intervals as follows.  These numbers agree with
Table~5 in \citep{geyer-gdor},
which was done by methods that are much more messy and made obsolete
by the methods presented here.
\begin{knitrout}
\definecolor{shadecolor}{rgb}{0.969, 0.969, 0.969}\color{fgcolor}\begin{kframe}
\begin{alltt}
\hlstd{ci.tab.too} \hlkwb{<-} \hlkwd{inference}\hlstd{(gout)}
\hlkwd{rownames}\hlstd{(ci.tab.too)} \hlkwb{<-} \hlstd{label[}\hlopt{!} \hlstd{gout}\hlopt{$}\hlstd{linearity]}
\hlkwd{round}\hlstd{(ci.tab.too,} \hlnum{3}\hlstd{)}
\end{alltt}
\begin{verbatim}
##                      lower upper
## ants beat beetles    0.893     2
## ants beat cows       1.245     2
## ants beat dogs       1.886     2
## ants beat egrets     1.809     2
## ants beat foxes      1.886     2
## ants beat gerbils    1.993     2
## beetles beat gerbils 1.970     2
## cows beat gerbils    1.940     2
## dogs beat gerbils    1.526     2
## egrets beat gerbils  1.699     2
## foxes beat gerbils   1.526     2
## ants beat hogs       1.993     2
## beetles beat hogs    1.970     2
## cows beat hogs       1.940     2
## dogs beat hogs       1.526     2
## egrets beat hogs     1.699     2
## foxes beat hogs      1.526     2
\end{verbatim}
\end{kframe}
\end{knitrout}

With $n = 2$ (each team plays each other team twice), quick and dirty
confidence intervals go from zero to
$$
   1 - \alpha^{1 / 2} = 0.7763932
$$
(when $\alpha = 0.05$) or from
$$
   \alpha^{1 / 2} = 0.2236068
$$
to one (again when $\alpha = 0.05$).
None of the careful intervals calculated above are anywhere near as
wide as the quick and dirty intervals.

\subsection{Linearity by computational geometry}

Calculate linearity using R package \texttt{rcdd} like
in Section~\ref{sec:linearity-example-i-rcdd} above.
We follow Section~5 of \citet{geyer-tech-too}, except that seems to
have some errors, which we correct here.
\begin{knitrout}
\definecolor{shadecolor}{rgb}{0.969, 0.969, 0.969}\color{fgcolor}\begin{kframe}
\begin{alltt}
\hlstd{tanv} \hlkwb{<-} \hlstd{modmat}
\hlstd{tanv[losses} \hlopt{==} \hlnum{0}\hlstd{, ]} \hlkwb{<-} \hlstd{(}\hlopt{-} \hlstd{tanv[losses} \hlopt{==} \hlnum{0}\hlstd{, ])}
\hlstd{vrep} \hlkwb{<-} \hlkwd{cbind}\hlstd{(}\hlnum{0}\hlstd{,} \hlnum{0}\hlstd{, tanv)}
\hlstd{vrep[wins} \hlopt{>} \hlnum{0} \hlopt{&} \hlstd{losses} \hlopt{>} \hlnum{0}\hlstd{,} \hlnum{1}\hlstd{]} \hlkwb{<-} \hlnum{1}
\hlstd{lout} \hlkwb{<-} \hlkwd{linearity}\hlstd{(}\hlkwd{d2q}\hlstd{(vrep),} \hlkwc{rep} \hlstd{=} \hlstr{"V"}\hlstd{)}
\end{alltt}
\end{kframe}
\end{knitrout}

This result only includes the additional components found to be in the
linearity (in addition to the ones already known).  So we have to add
the others to get the correct linearity.
\begin{knitrout}
\definecolor{shadecolor}{rgb}{0.969, 0.969, 0.969}\color{fgcolor}\begin{kframe}
\begin{alltt}
\hlstd{linearity.too} \hlkwb{<-} \hlkwd{seq}\hlstd{(}\hlkwc{along} \hlstd{= wins)} \hlopt{%in%} \hlstd{lout}
\hlstd{linearity.too[wins} \hlopt{>} \hlnum{0} \hlopt{&} \hlstd{losses} \hlopt{>} \hlnum{0}\hlstd{]} \hlkwb{<-} \hlnum{TRUE}
\hlkwd{identical}\hlstd{(}\hlkwd{as.vector}\hlstd{(gout}\hlopt{$}\hlstd{linearity), linearity.too)}
\end{alltt}
\begin{verbatim}
## [1] TRUE
\end{verbatim}
\end{kframe}
\end{knitrout}

So this agrees with our analysis in
Section~\ref{sec:sports-linearity} above.

\subsection{Generic direction of recession}

Calculate a GDOR using R package \texttt{rcdd} like
in Section~\ref{sec:gdor-example-i-rcdd} above.
More specifically, we follow Section~6 of \citep{geyer-tech},
so we necessarily agree with the GDOR given in Table~4 of \citep{geyer-gdor}.
\begin{knitrout}
\definecolor{shadecolor}{rgb}{0.969, 0.969, 0.969}\color{fgcolor}\begin{kframe}
\begin{alltt}
\hlstd{p} \hlkwb{<-} \hlkwd{ncol}\hlstd{(tanv)}
\hlstd{hrep} \hlkwb{<-} \hlkwd{cbind}\hlstd{(}\hlnum{0}\hlstd{,} \hlnum{0}\hlstd{,} \hlopt{-}\hlstd{tanv,} \hlnum{0}\hlstd{)}
\hlstd{hrep[}\hlopt{!} \hlstd{gout}\hlopt{$}\hlstd{linearity,} \hlkwd{ncol}\hlstd{(hrep)]} \hlkwb{<-} \hlstd{(}\hlopt{-}\hlnum{1}\hlstd{)}
\hlstd{hrep[gout}\hlopt{$}\hlstd{linearity,} \hlnum{1}\hlstd{]} \hlkwb{<-} \hlnum{1}
\hlstd{hrep} \hlkwb{<-} \hlkwd{rbind}\hlstd{(hrep,} \hlkwd{c}\hlstd{(}\hlnum{0}\hlstd{,} \hlnum{1}\hlstd{,} \hlkwd{rep}\hlstd{(}\hlnum{0}\hlstd{, p),} \hlopt{-}\hlnum{1}\hlstd{))}
\hlstd{objv} \hlkwb{<-} \hlkwd{c}\hlstd{(}\hlkwd{rep}\hlstd{(}\hlnum{0}\hlstd{, p),} \hlnum{1}\hlstd{)}
\hlstd{pout} \hlkwb{<-} \hlkwd{lpcdd}\hlstd{(hrep, objv,} \hlkwc{minimize} \hlstd{=} \hlnum{FALSE}\hlstd{)}
\hlstd{gdor} \hlkwb{<-} \hlstd{pout}\hlopt{$}\hlstd{primal.solution[}\hlnum{1}\hlopt{:}\hlstd{p]}
\hlkwd{names}\hlstd{(gdor)} \hlkwb{<-} \hlstd{team.names}
\hlkwd{print}\hlstd{(gdor)}
\end{alltt}
\begin{verbatim}
##    ants beetles    cows    dogs  egrets   foxes gerbils 
##       2       1       1       1       1       1       0 
##    hogs 
##       0
\end{verbatim}
\end{kframe}
\end{knitrout}

Clean R global environment.

\begin{knitrout}
\definecolor{shadecolor}{rgb}{0.969, 0.969, 0.969}\color{fgcolor}\begin{kframe}
\begin{alltt}
\hlkwd{rm}\hlstd{(}\hlkwc{list} \hlstd{=} \hlkwd{ls}\hlstd{())}
\end{alltt}
\end{kframe}
\end{knitrout}

\section{Quasi-complete separation example of Agresti}
\label{sec:agresti-quasi}

\subsection{Data}

\citet[Section~6.5.1]{agresti} introduces the notion of quasi-complete
separation
with the following example, which adds two data points to the data for
his other example (Section~\ref{sec:agresti-complete} above).
\begin{knitrout}
\definecolor{shadecolor}{rgb}{0.969, 0.969, 0.969}\color{fgcolor}\begin{kframe}
\begin{alltt}
\hlstd{x} \hlkwb{<-} \hlkwd{seq}\hlstd{(}\hlnum{10}\hlstd{,} \hlnum{90}\hlstd{,} \hlnum{10}\hlstd{)}
\hlstd{x} \hlkwb{<-} \hlstd{x[x} \hlopt{!=} \hlnum{50}\hlstd{]}
\hlstd{y} \hlkwb{<-} \hlkwd{as.numeric}\hlstd{(x} \hlopt{>} \hlnum{50}\hlstd{)}
\hlstd{x} \hlkwb{<-} \hlkwd{c}\hlstd{(x,} \hlnum{50}\hlstd{,} \hlnum{50}\hlstd{)}
\hlstd{y} \hlkwb{<-} \hlkwd{c}\hlstd{(y,} \hlnum{0}\hlstd{,} \hlnum{1}\hlstd{)}
\end{alltt}
\end{kframe}
\end{knitrout}

These data are included in the \texttt{glmdr} package.

\begin{knitrout}
\definecolor{shadecolor}{rgb}{0.969, 0.969, 0.969}\color{fgcolor}\begin{kframe}
\begin{alltt}
\hlkwd{data}\hlstd{(quasi)}
\hlkwd{all.equal}\hlstd{(quasi,} \hlkwd{data.frame}\hlstd{(x, y))}
\end{alltt}
\begin{verbatim}
## [1] TRUE
\end{verbatim}
\end{kframe}
\end{knitrout}

\subsection{Maximizing the OM likelihood}

Again, we fit these data using R function \texttt{glmdr}.

\begin{knitrout}
\definecolor{shadecolor}{rgb}{0.969, 0.969, 0.969}\color{fgcolor}\begin{kframe}
\begin{alltt}
\hlstd{gout} \hlkwb{<-} \hlkwd{glmdr}\hlstd{(y} \hlopt{~} \hlstd{x,} \hlkwc{family} \hlstd{=} \hlstr{"binomial"}\hlstd{,} \hlkwc{data} \hlstd{= quasi)}
\hlkwd{summary}\hlstd{(gout)}
\end{alltt}
\begin{verbatim}
## 
## MLE exists in Barndorff-Nielsen completion 
## it is conditional on components of the response 
## corresponding to object$linearity == FALSE being 
## conditioned on their observed values 
## 
## GLM summary for limiting conditional model
## 
## 
## Call:
## stats::glm(formula = y ~ x, family = "binomial", data = quasi, 
##     subset = c("9", "10"), x = TRUE, y = TRUE)
## 
## Deviance Residuals: 
##      9      10  
## -1.177   1.177  
## 
## Coefficients: (1 not defined because of singularities)
##              Estimate Std. Error z value Pr(>|z|)
## (Intercept) 4.710e-16  1.414e+00       0        1
## x                  NA         NA      NA       NA
## 
## (Dispersion parameter for binomial family taken to be 1)
## 
##     Null deviance: 2.7726  on 1  degrees of freedom
## Residual deviance: 2.7726  on 1  degrees of freedom
## AIC: 4.7726
## 
## Number of Fisher Scoring iterations: 2
\end{verbatim}
\end{kframe}
\end{knitrout}

\subsection{Linearity}

We extract the linearity from the \texttt{glmdr} function call.

\begin{knitrout}
\definecolor{shadecolor}{rgb}{0.969, 0.969, 0.969}\color{fgcolor}\begin{kframe}
\begin{alltt}
\hlstd{gout}\hlopt{$}\hlstd{linearity}
\end{alltt}
\begin{verbatim}
##     1     2     3     4     5     6     7     8     9    10 
## FALSE FALSE FALSE FALSE FALSE FALSE FALSE FALSE  TRUE  TRUE
\end{verbatim}
\end{kframe}
\end{knitrout}

\subsection{One-sided confidence intervals for mean value parameters}

We now provide one-sided confidence intervals for mean value parameters 
whose MLE is on the boundary.  

\begin{knitrout}
\definecolor{shadecolor}{rgb}{0.969, 0.969, 0.969}\color{fgcolor}\begin{kframe}
\begin{alltt}
\hlkwd{inference}\hlstd{(gout)}
\end{alltt}
\begin{verbatim}
##       lower      upper
## 1 0.0000000 0.07082447
## 2 0.0000000 0.14043775
## 3 0.0000000 0.27199887
## 4 0.0000000 0.51720648
## 5 0.4827935 1.00000000
## 6 0.7280012 1.00000000
## 7 0.8595623 1.00000000
## 8 0.9291755 1.00000000
\end{verbatim}
\end{kframe}
\end{knitrout}

Note that for some components of the mean value parameter vector
the lower or upper bound of
our confidence interval is not close to the quick and dirty limit
(Section~\ref{sec:quick-and-dirty} above) like they were in the case of 
complete separation.

\subsection{Two-sided confidence intervals for mean value parameters}

As in the preceding example, confidence intervals for means of components
of the response vector in the linearity are given by R generic function
\texttt{predict}.
\begin{knitrout}
\definecolor{shadecolor}{rgb}{0.969, 0.969, 0.969}\color{fgcolor}\begin{kframe}
\begin{alltt}
\hlstd{preds} \hlkwb{<-} \hlkwd{predict}\hlstd{(gout}\hlopt{$}\hlstd{lcm,} \hlkwc{type} \hlstd{=} \hlstr{"response"}\hlstd{,} \hlkwc{se.fit} \hlstd{=} \hlnum{TRUE}\hlstd{)}
\hlstd{preds.tab} \hlkwb{<-} \hlkwd{cbind}\hlstd{(preds}\hlopt{$}\hlstd{fit} \hlopt{-} \hlkwd{qnorm}\hlstd{(}\hlnum{0.975}\hlstd{)} \hlopt{*} \hlstd{preds}\hlopt{$}\hlstd{se.fit,}
    \hlstd{preds}\hlopt{$}\hlstd{fit} \hlopt{+} \hlkwd{qnorm}\hlstd{(}\hlnum{0.975}\hlstd{)} \hlopt{*} \hlstd{preds}\hlopt{$}\hlstd{se.fit)}
\hlkwd{colnames}\hlstd{(preds.tab)} \hlkwb{<-} \hlkwd{c}\hlstd{(}\hlstr{"lower"}\hlstd{,} \hlstr{"upper"}\hlstd{)}
\hlkwd{round}\hlstd{(preds.tab,} \hlnum{3}\hlstd{)}
\end{alltt}
\begin{verbatim}
##     lower upper
## 9  -0.193 1.193
## 10 -0.193 1.193
\end{verbatim}
\end{kframe}
\end{knitrout}

As we saw with the sports data, these asymptotic confidence intervals
are not good for toy data.  Again we effectively have $n = 2$ for these
intervals, so they are exactly the same as the one for gerbils versus
hogs in the sports data.

Clean R global environment.
\begin{knitrout}
\definecolor{shadecolor}{rgb}{0.969, 0.969, 0.969}\color{fgcolor}\begin{kframe}
\begin{alltt}
\hlkwd{rm}\hlstd{(}\hlkwc{list} \hlstd{=} \hlkwd{ls}\hlstd{())}
\end{alltt}
\end{kframe}
\end{knitrout}

\section{Categorical data analysis example of Geyer}
\label{sec:geyer-categorical}

\subsection{Data}

This is the example in Section~2.3 of \citet{geyer-gdor}.  Its behavior
is very similar to the quasi-complete separation example
of \citeauthor{agresti} in Section~\ref{sec:agresti-quasi} above.
\begin{knitrout}
\definecolor{shadecolor}{rgb}{0.969, 0.969, 0.969}\color{fgcolor}\begin{kframe}
\begin{alltt}
\hlstd{foo} \hlkwb{<-} \hlstr{"https://conservancy.umn.edu/bitstream/handle/11299/197369/catrec.txt"}
\hlstd{bar} \hlkwb{<-} \hlkwd{sub}\hlstd{(}\hlstr{"^.*/"}\hlstd{,} \hlstr{""}\hlstd{, foo)}
\hlkwa{if} \hlstd{(}\hlopt{!} \hlkwd{file.exists}\hlstd{(bar))}
    \hlkwd{download.file}\hlstd{(foo, bar)}
\hlstd{dat} \hlkwb{<-} \hlkwd{read.table}\hlstd{(bar,} \hlkwc{header} \hlstd{=} \hlnum{TRUE}\hlstd{)}
\hlkwd{dim}\hlstd{(dat)}
\end{alltt}
\begin{verbatim}
## [1] 128   8
\end{verbatim}
\begin{alltt}
\hlkwd{names}\hlstd{(dat)}
\end{alltt}
\begin{verbatim}
## [1] "v1" "v2" "v3" "v4" "v5" "v6" "v7" "y"
\end{verbatim}
\end{kframe}
\end{knitrout}

These data are included in the \texttt{glmdr} package.

\begin{knitrout}
\definecolor{shadecolor}{rgb}{0.969, 0.969, 0.969}\color{fgcolor}\begin{kframe}
\begin{alltt}
\hlkwd{data}\hlstd{(catrec)}
\hlkwd{all.equal}\hlstd{(catrec, dat)}
\end{alltt}
\begin{verbatim}
## [1] TRUE
\end{verbatim}
\end{kframe}
\end{knitrout}

\subsection{Fitting the Model}

Following \citet{geyer-gdor} we assume Poisson rather than multinomial
sampling.  These two sampling schemes have the same MLE, even when the MLE
is in the Barndorff-Nielsen completion
[\citealp[Section~8.6.7]{agresti}; \citealp[Section~3.17]{geyer-gdor}]
but Poisson sampling is the easiest to fit.
We can use R function \texttt{glm} if the MLE exists in the conventional
sense, and R function \texttt{glmdr} otherwise.
\begin{knitrout}
\definecolor{shadecolor}{rgb}{0.969, 0.969, 0.969}\color{fgcolor}\begin{kframe}
\begin{alltt}
\hlstd{gout} \hlkwb{<-} \hlkwd{glmdr}\hlstd{(y} \hlopt{~} \hlstd{(v1} \hlopt{+} \hlstd{v2} \hlopt{+} \hlstd{v3} \hlopt{+} \hlstd{v4} \hlopt{+} \hlstd{v5} \hlopt{+} \hlstd{v6} \hlopt{+} \hlstd{v7)}\hlopt{^}\hlnum{3}\hlstd{,}
    \hlkwc{family} \hlstd{=} \hlstr{"poisson"}\hlstd{,} \hlkwc{data} \hlstd{= dat)}
\hlkwd{summary}\hlstd{(gout)}
\end{alltt}
\begin{verbatim}
## 
## MLE exists in Barndorff-Nielsen completion 
## it is conditional on components of the response 
## corresponding to object$linearity == FALSE being 
## conditioned on their observed values 
## 
## GLM summary for limiting conditional model
## 
## 
## Call:
## stats::glm(formula = y ~ (v1 + v2 + v3 + v4 + v5 + v6 + v7)^3, 
##     family = "poisson", data = dat, subset = c("2", "3", "4", 
##     "5", "6", "7", "8", "10", "11", "12", "13", "14", "15", "16", 
##     "17", "18", "19", "21", "22", "23", "24", "25", "26", "27", 
##     "29", "30", "31", "32", "34", "35", "36", "37", "38", "39", 
##     "40", "42", "43", "44", "45", "46", "47", "48", "49", "50", 
##     "51", "53", "54", "55", "56", "57", "58", "59", "61", "62", 
##     "63", "64", "66", "67", "68", "69", "70", "71", "72", "74", 
##     "75", "76", "77", "78", "79", "80", "81", "82", "83", "85", 
##     "86", "87", "88", "89", "90", "91", "93", "94", "95", "96", 
##     "98", "99", "100", "101", "102", "103", "104", "106", "107", 
##     "108", "109", "110", "111", "112", "113", "114", "115", "117", 
##     "118", "119", "120", "121", "122", "123", "125", "126", "127", 
##     "128"), x = TRUE, y = TRUE)
## 
## Deviance Residuals: 
##      Min        1Q    Median        3Q       Max  
## -1.63571  -0.30009  -0.02353   0.27258   1.42540  
## 
## Coefficients: (1 not defined because of singularities)
##              Estimate Std. Error z value Pr(>|z|)    
## (Intercept)  2.150481   0.585423   3.673 0.000239 ***
## v1           0.069795   0.587067   0.119 0.905364    
## v2          -0.524215   0.513583  -1.021 0.307396    
## v3           0.052966   0.551965   0.096 0.923552    
## v4          -0.709525   0.580147  -1.223 0.221326    
## v5           0.243002   0.548686   0.443 0.657853    
## v6          -1.163256   0.563668  -2.064 0.039044 *  
## v7          -0.990704   0.597335  -1.659 0.097208 .  
## v1:v2        0.384345   0.543024   0.708 0.479079    
## v1:v3       -0.630375   0.570151  -1.106 0.268888    
## v1:v4        0.008801   0.511458   0.017 0.986271    
## v1:v5       -1.022805   0.570440  -1.793 0.072971 .  
## v1:v6        0.540164   0.493879   1.094 0.274079    
## v1:v7        0.097178   0.536628   0.181 0.856297    
## v2:v3        0.602411   0.437371   1.377 0.168405    
## v2:v4        0.748226   0.486811   1.537 0.124295    
## v2:v5       -0.068926   0.428100  -0.161 0.872090    
## v2:v6        0.297165   0.487409   0.610 0.542071    
## v2:v7        0.274198   0.508369   0.539 0.589634    
## v3:v4       -0.124465   0.541056  -0.230 0.818060    
## v3:v5       -0.439354   0.468418  -0.938 0.348268    
## v3:v6        0.024399   0.530220   0.046 0.963296    
## v3:v7       -0.104400   0.556960  -0.187 0.851310    
## v4:v5       -0.169421   0.521323  -0.325 0.745194    
## v4:v6        0.756513   0.474213   1.595 0.110644    
## v4:v7        0.780671   0.500911   1.559 0.119114    
## v5:v6        1.245629   0.510770   2.439 0.014739 *  
## v5:v7       -0.262620   0.523125  -0.502 0.615652    
## v6:v7        0.697014   0.489957   1.423 0.154852    
## v1:v2:v3    -0.349902   0.483330  -0.724 0.469102    
## v1:v2:v4     0.101569   0.389778   0.261 0.794416    
## v1:v2:v5     0.655208   0.493737   1.327 0.184496    
## v1:v2:v6    -0.329286   0.390979  -0.842 0.399670    
## v1:v2:v7    -0.520368   0.393042  -1.324 0.185520    
## v1:v3:v4     0.353292   0.406623   0.869 0.384932    
## v1:v3:v5     0.638711   0.484979   1.317 0.187843    
## v1:v3:v6     0.352694   0.402715   0.876 0.381143    
## v1:v3:v7    -0.001586   0.413554  -0.004 0.996941    
## v1:v4:v5     0.664745   0.400212   1.661 0.096717 .  
## v1:v4:v6    -0.463885   0.368214  -1.260 0.207732    
## v1:v4:v7    -0.342583   0.372009  -0.921 0.357103    
## v1:v5:v6     0.044968   0.399958   0.112 0.910481    
## v1:v5:v7     0.447641   0.404364   1.107 0.268283    
## v1:v6:v7     0.218868   0.371499   0.589 0.555763    
## v2:v3:v4    -0.325914   0.404392  -0.806 0.420280    
## v2:v3:v5           NA         NA      NA       NA    
## v2:v3:v6    -0.247853   0.405621  -0.611 0.541168    
## v2:v3:v7     0.028322   0.414520   0.068 0.945527    
## v2:v4:v5     0.004655   0.394418   0.012 0.990583    
## v2:v4:v6    -0.111152   0.373713  -0.297 0.766141    
## v2:v4:v7    -0.148061   0.376692  -0.393 0.694279    
## v2:v5:v6    -0.766051   0.394925  -1.940 0.052412 .  
## v2:v5:v7     0.075213   0.399004   0.189 0.850482    
## v2:v6:v7     0.460826   0.381109   1.209 0.226597    
## v3:v4:v5    -0.063494   0.423318  -0.150 0.880771    
## v3:v4:v6     0.357746   0.366298   0.977 0.328741    
## v3:v4:v7    -0.106368   0.371567  -0.286 0.774672    
## v3:v5:v6    -0.234816   0.422424  -0.556 0.578295    
## v3:v5:v7     0.804923   0.423843   1.899 0.057550 .  
## v3:v6:v7    -0.659090   0.371085  -1.776 0.075714 .  
## v4:v5:v6    -0.427957   0.375755  -1.139 0.254734    
## v4:v5:v7     0.125167   0.377356   0.332 0.740119    
## v4:v6:v7     0.014192   0.370131   0.038 0.969413    
## v5:v6:v7    -0.811516   0.377098  -2.152 0.031397 *  
## ---
## Signif. codes:  
## 0 '***' 0.001 '**' 0.01 '*' 0.05 '.' 0.1 ' ' 1
## 
## (Dispersion parameter for poisson family taken to be 1)
## 
##     Null deviance: 156.215  on 111  degrees of freedom
## Residual deviance:  31.291  on  49  degrees of freedom
## AIC: 526.46
## 
## Number of Fisher Scoring iterations: 5
\end{verbatim}
\end{kframe}
\end{knitrout}

This agrees with the result in the technical
report \citep[Section~4.2.1]{geyer-tech} accompanying \cite{geyer-gdor}.

\subsection{Linearity}
\label{sec:example-iv-linearity}

We then find the linearity as in preceding sections.
\begin{knitrout}
\definecolor{shadecolor}{rgb}{0.969, 0.969, 0.969}\color{fgcolor}\begin{kframe}
\begin{alltt}
\hlstd{linearity} \hlkwb{<-} \hlstd{gout}\hlopt{$}\hlstd{linearity}
\hlstd{catrec[}\hlopt{!}\hlstd{linearity, ]}
\end{alltt}
\begin{verbatim}
##     v1 v2 v3 v4 v5 v6 v7 y
## 1    0  0  0  0  0  0  0 0
## 9    0  0  0  1  0  0  0 0
## 20   1  1  0  0  1  0  0 0
## 28   1  1  0  1  1  0  0 0
## 33   0  0  0  0  0  1  0 0
## 41   0  0  0  1  0  1  0 0
## 52   1  1  0  0  1  1  0 0
## 60   1  1  0  1  1  1  0 0
## 65   0  0  0  0  0  0  1 0
## 73   0  0  0  1  0  0  1 0
## 84   1  1  0  0  1  0  1 0
## 92   1  1  0  1  1  0  1 0
## 97   0  0  0  0  0  1  1 0
## 105  0  0  0  1  0  1  1 0
## 116  1  1  0  0  1  1  1 0
## 124  1  1  0  1  1  1  1 0
\end{verbatim}
\end{kframe}
\end{knitrout}

This agrees with (part of) Table~2 in \citep{geyer-gdor}.

\subsection{One-sided confidence intervals: Poisson sampling}
\label{sec:one-sided-poisson}

We now provide one-sided confidence intervals for mean value parameters 
whose MLE is on the boundary as done before.

\begin{knitrout}
\definecolor{shadecolor}{rgb}{0.969, 0.969, 0.969}\color{fgcolor}\begin{kframe}
\begin{alltt}
\hlkwd{system.time}\hlstd{(tab} \hlkwb{<-} \hlkwd{inference}\hlstd{(gout))}
\end{alltt}
\begin{verbatim}
##    user  system elapsed 
##   0.564   0.012   0.578
\end{verbatim}
\begin{alltt}
\hlstd{upper} \hlkwb{<-} \hlstd{tab}\hlopt{$}\hlstd{upper}
\hlkwd{cbind}\hlstd{(catrec[}\hlopt{!}\hlstd{linearity, ], upper)}
\end{alltt}
\begin{verbatim}
##     v1 v2 v3 v4 v5 v6 v7 y      upper
## 1    0  0  0  0  0  0  0 0 0.28630976
## 9    0  0  0  1  0  0  0 0 0.14082947
## 20   1  1  0  0  1  0  0 0 0.21996699
## 28   1  1  0  1  1  0  0 0 0.42095570
## 33   0  0  0  0  0  1  0 0 0.08946242
## 41   0  0  0  1  0  1  0 0 0.09376644
## 52   1  1  0  0  1  1  0 0 0.19302341
## 60   1  1  0  1  1  1  0 0 0.28869770
## 65   0  0  0  0  0  0  1 0 0.10631113
## 73   0  0  0  1  0  0  1 0 0.11415034
## 84   1  1  0  0  1  0  1 0 0.09128766
## 92   1  1  0  1  1  0  1 0 0.26461098
## 97   0  0  0  0  0  1  1 0 0.06669488
## 105  0  0  0  1  0  1  1 0 0.15477613
## 116  1  1  0  0  1  1  1 0 0.14096916
## 124  1  1  0  1  1  1  1 0 0.32392016
\end{verbatim}
\end{kframe}
\end{knitrout}

This agrees with Table~2 in \citep{geyer-gdor}.

\subsubsection{Theory}
\label{sec:theory-poisson}

Here we modify Section~\ref{sec:theory-logistic} above, changing what
needs to be changed for Poisson regression rather than logistic regression.

As in Section~\ref{sec:theory-logistic} above,
let $\beta$ denote the vector of submodel canonical parameters,
let $l(\beta)$ denote the log likelihood,
and let $\hat{\beta}$ denote an MLE in the LCM.
We will use
the vector \verb@gout$lcm$coefficients@ with \texttt{NA} values
replaced by zeros.
Let $I$ denote the index set of the components of the response vector
on which we condition the OM to get the LCM (the indices of components of
\verb@linearity@ that are \texttt{FALSE}), and let $Y_I$ and $y_I$
denote the corresponding components of the response vector considered
as a random vector and as an observed
value, respectively.  Then endpoints for a $100 (1 - \alpha)\%$ confidence
interval for a scalar parameter $g(\beta)$ are given
by \eqref{eq:g-optim}, when it does give a one-sided interval.

Since the only boundary of the mean value parameter space of the Poisson
distribution is zero,
in this section, we will be doing confidence intervals for
mean value parameters for cells of the
contingency table where the MLE in the LCM is zero.  And we know the min
is zero, so we only have to calculate the max.

In \eqref{eq:g-optim} 
$\pr$ denotes probability with respect to the OM not
the LCM.  As always in categorical data analysis, we have different possible
sampling models: Poisson, multinomial, and product multinomial.  So we get
different intervals depending on which sampling model we use.  In this
section we are assuming Poisson.

Let $M$ denote the model matrix.  Let $\theta = M \beta$ denote the saturated
model canonical parameter (usually called ``linear predictor'' in GLM theory).

Let $\mu = \exp(\theta)$ denote the mean value parameter (here $\exp$
operates componentwise like the R function of the same name does), then
$$
   \pr_\beta(Y_I = y_I) =
   \pr_\beta(Y_I = 0) =
   \exp\left( - \sum_{i \in I} \mu_i \right)
$$

We could take the confidence interval problem to be
\begin{equation} \label{eq:poisson-ci-problem-fubar}
\begin{split}
   \text{maximize} & \quad \mu_k
   \\
   \text{subject to} & \quad \exp\left(- \sum_{i \in I} \mu_i\right) \ge \alpha
\end{split}
\end{equation}
where $\mu$ is taken to be the function of $\gamma$ described above.
And this can be done for any $k \in I$.

But the problem will be more computationally stable if we state it as
\begin{equation} \label{eq:poisson-ci-problem-semi-fubar}
\begin{split}
   \text{maximize} & \quad \theta_k
   \\
   \text{subject to} & \quad - \sum_{i \in I} \mu_i \ge \log(\alpha)
\end{split}
\end{equation}
Since $\mu_k = \exp(\theta_k)$ is a monotone transformation and log is
a monotone transformation, the two
problems are equivalent (a solution for one is also a solution
for the other).

We maximize canonical rather than mean value parameters to avoid extreme
inexactness of computer arithmetic in calculating mean value parameters
near zero.  We take logs in the constraint for the same reasons
we take logs of likelihoods.

Because optimizers expect to optimize over $\R^q$ for some $q$,
let $N$ be a matrix whose columns are a basis for $\Gamma_\text{lim}$
(the R matrix \texttt{nulls} calculated above, for example).
Then every $\gamma \in \Gamma_\text{lim}$ can be written as
$\gamma = N \xi$ for some $\xi \in \R^q$,
where $q$ is the column dimension
of $N$ and the dimension of $\Gamma_\text{lim}$.

To an optimizer (we will use R function \texttt{auglag} in CRAN package
\texttt{alabama}) problem \eqref{eq:poisson-ci-problem-semi-fubar} has the 
abstract form \eqref{eq:abstract}
and the optimization works better if derivatives of $f$ and $g$ are provided.
Because R function \texttt{auglag} only does minimization, the objective
function must be the negation of what we have in 
\eqref{eq:poisson-ci-problem-semi-fubar}.  That is
\begin{align*}
   f(\xi) & = - \theta_k
   \\
   \frac{\partial f(\xi)}{\partial \xi_j}
   & =
   - o_{k j}
   \\
   g(\xi) & = - \sum_{i \in I} \mu_i - \log(\alpha)
   \\
   \frac{\partial g(\xi)}{\partial \xi_j}
   & =
   - \sum_{i \in I} \mu_i o_{i j}
\end{align*}
where $o_{i j}$ are the components of $O = M N$.

\subsubsection{Quick and dirty intervals}
\label{sec:quick-and-dirty-poisson}

As a sanity check and as a quick and dirty conservative (perhaps very
conservative) confidence interval, we note that since all the $\mu_i$
are nonnegative, the only way the constraint in
\eqref{eq:poisson-ci-problem-fubar} can be satisfied is if
$\mu_k \le - \log(\alpha)$.
For $\alpha = 0.05$ this upper bound is -$\log(0.05) = 2.996$. No upper bound for a one-sided 95\% confidence interval for the mean value parameter for a cell for which the MLE in the LCM is zero can be larger than that.

\subsection{One-sided confidence intervals: Multinomial sampling}
\label{sec:one-sided-multinomial}

\subsubsection{Theory}
\label{sec:theory-multinomial}

We use the same notation as in Section~\ref{sec:theory-poisson} above,
except where modified here.

Since the only boundary of the mean value parameter space of the multinomial
distribution is where one or more components of the state vector are zero,
we will be doing confidence intervals for
mean value parameters for cells of the
contingency table where the MLE in the LCM is zero.  And we know the min
is zero, so we only have to calculate the max.
(If the MLE in the LCM for mean value parameter vector had all but one
component equal to zero, so the other was equal to one, then we could
make one-sided intervals for all components.  But that is not a situation
we see in any of our examples, and we will leave that as an exercise for
the reader.)

For multinomial sampling, contingency table cell probabilities are defined by
\begin{equation} \label{eq:multinomial-parameterize}
   p_i = \frac{e^{\theta_i}}{\sum_{j \in J} e^{\theta_j}},
   \qquad i \in J,
\end{equation}
where $J$ is the index set for the whole table.

Now
$$
   \pr_\beta(Y_I = y_I) =
   \pr_\beta(Y_I = 0) =
   \left( \sum_{i \in J \setminus I} p_i \right)^n
$$
where
$$
   n = \sum_{j \in J} y_j
$$
is the multinomial sample size,
where $I$ is the index set of the cells that have mean value zero
for the MLE in the LCM.

So we could take the confidence interval problem to be
\begin{equation} \label{eq:multinomial-ci-problem-fubar}
\begin{split}
   \text{maximize} & \quad p_k
   \\
   \text{subject to} & \quad
   \left( \sum_{i \in J \setminus I} p_i \right)^n \ge \alpha
\end{split}
\end{equation}
where $p$ is taken to be the function of $\gamma$ described above.
And this can be done for any $k \in I$.

Unlike preceding theory for this problem, we cannot take $\theta_k$ to
be the objective function because $p_k$ is not a function of $\theta_k$ only
(much less a monotone function of it).  Consequently, to obtain computational
stability, we will take logs of both equations obtaining
\begin{equation} \label{eq:multinomial-ci-problem}
\begin{split}
   \text{maximize} & \quad \theta_k -
   \log \left( \sum_{j \in J} e^{\theta_j} \right)
   \\
   \text{subject to} & \quad
   n \log \left( \sum_{i \in J \setminus I} e^{\theta_i} \right)
   -
   n \log \left( \sum_{j \in J} e^{\theta_j} \right)
   \ge \log(\alpha)
\end{split}
\end{equation}

The parameterization \eqref{eq:multinomial-parameterize}
introduces a direction of constancy (DOC)
\citep[Theorem~1 and the following discussion]{geyer-gdor} that is
the same as the DOC we had in the Bradley-Terry model
(Section~\ref{sec:geyer-sports} above), the vector all of whose
components are the same.

So perhaps we should redo our null space of the Fisher information matrix
calculation using the multinomial distribution.  But this is not necessary.
Movement along the DOC does not change any of the $p_i$ so does not
change any of the equations in either of our optimization problems.
We do not need to add it to the null space we obtained from the Poisson
analysis.
(Section~3.17 in \citep{geyer-gdor} shows that every DOR for the Poisson
model is also a DOR for the multinomial model.)

Thus our problem has the abstract form \eqref{eq:abstract} with
\begin{align}
   f(\xi)
   & =
   - \theta_k +
   \log \left( \sum_{j \in J} e^{\theta_j} \right)
   \label{eq:f}
   \\
   \frac{\partial f(\xi)}{\partial \xi_j}
   & =
   - o_{k j} +
   \frac{\sum_{i \in J} e^{\theta_i} o_{i j}}{\sum_{i \in J} e^{\theta_i}}
   \label{eq:df}
\end{align}
where $o_{k j}$ are the components of $O = M N$, and
\begin{align}
   g(\xi) & =
   n \log \left( \sum_{i \in J \setminus I} e^{\theta_i} \right)
   -
   n \log \left( \sum_{j \in J} e^{\theta_j} \right)
   -
   \log(\alpha)
   \label{eq:g}
   \\
   \frac{\partial g(\xi)}{\partial \xi_j}
   & =
   n \frac{ \sum_{i \in J \setminus I} e^{\theta_i} o_{i j} }
   { \sum_{k \in J \setminus I} e^{\theta_k} }
   -
   n \frac{ \sum_{i \in J} e^{\theta_i} o_{i j} }
   { \sum_{k \in J} e^{\theta_k} }
   \label{eq:dg}
   \\
   & =
   n \sum_{i \in J} (p^*_i - p_i) o_{i j}
   \nonumber
\end{align}
where
$$
   p^*_i
   =
   \begin{cases}
     e^{\theta_i} \big/ \sum_{j \in J \setminus I} e^{\theta_j},
     & i \in J \setminus I
     \\
     0, & \text{otherwise}
   \end{cases}
$$
($p$ is the vector of probabilities in the OM,
$p^*$ is the vector of probabilities in the LCM).

\subsubsection{Quick and dirty intervals}

If $p_i > 0$ for some $i \in I$, then
$$
   \left( \sum_{j \in J \setminus I} p_j \right)^n
   \le
   ( 1 - p_i )^n
$$
Introducing $\mu_i = n p_i$ we get
$$
   \alpha \le
   \left( \sum_{i \in J \setminus I} p_i \right)^n
   \le
   \left( 1 - \frac{\mu_i}{n} \right)^n
   \approx
   \exp(- \mu_i)
$$
for large $n$.  Thus this agrees with our analysis in
Section~\ref{sec:quick-and-dirty-poisson} when $n$ is large.

We get the exact inequality
$$
   \alpha \le \left( 1 - \frac{\mu_i}{n} \right)^n
$$
or
$$
   \alpha^{1 / n} \le 1 - \frac{\mu_i}{n}
$$
or
$$
   \mu_i \le n (1 - \alpha^{1 / n})
   = 2.9875
$$
when $n = 544$, which is what it is in this example,
and $\alpha = 0.05$.  And this too agrees approximately with our analysis
in Section~\ref{sec:quick-and-dirty-poisson} above.

\subsubsection{Careful coding}

We can modify \eqref{eq:f} above as
$$
   f(\xi)
   =
   a - \theta_k +
   \log \left( \sum_{j \in J} e^{\theta_j - a} \right)
$$
where $a$ is any real number.  We avoid overflow and catastrophic
cancellation if we choose
$$
   a = \theta_m = \max_{j \in J} \theta_j
$$
in which case we have
$$
   f(\xi)
   =
   \theta_m - \theta_k +
   \log\left(1 + \sum_{j \in J \setminus \{m\}}
   e^{\theta_j - \theta_m} \right)
$$
in which overflow cannot occur and we avoid catastrophic cancellation
in $\log(1 + x)$ for small $x$.

Using the same definition of $\theta_m$, we modify \eqref{eq:df} above as
$$
   \frac{\partial f(\xi)}{\partial \xi_j}
   =
   - o_{k j} +
   \frac{e^{\theta_k - \theta_m} o_{k j}}
   {\sum_{i \in J} e^{\theta_i - \theta_m}}
   =
   \left[
   - 1 + 
   \frac{e^{\theta_k - \theta_m}}
   {\sum_{i \in J} e^{\theta_i - \theta_m}}
   \right]
   o_{k j}
$$
in which overflow cannot occur.

We can modify \eqref{eq:g} above as
$$
   g(\xi)
   =
   n b + n \log \left( \sum_{i \in J \setminus I} e^{\theta_i - b} \right)
   -
   n a - n \log \left( \sum_{j \in J} e^{\theta_j - a} \right)
   -
   \log(\alpha)
$$
where $a$ and $b$ are any real numbers.  We avoid overflow and catastrophic
cancellation if we choose $a$ as above and
$$
   b = \theta_{m^*} = \max_{i \in J \setminus I} \theta_i
$$
in which case we have
$$
   g(\xi)
   =
   n \left[ \theta_{m^*} - \theta_m +
   \log \left(1 + \sum_{i \in (J \setminus I) \setminus \{ m^* \}}
   e^{\theta_i - \theta_{m^*}} \right)
   -
   \log\left(1 + \sum_{j \in J \setminus \{ m \}}
   e^{\theta_j - \theta_m} \right) \right]
   -
   \log(\alpha)
$$
in which overflow cannot occur and we avoid catastrophic cancellation
in $\log(1 + x)$ for small $x$.

Then using the same definitions of $\theta_m$ and $\theta_{m^*}$ we
modify \eqref{eq:dg} above as
$$
   \frac{\partial g(\xi)}{\partial \xi_j}
   =
   n \left[
   \frac{ \sum_{i \in J \setminus I} e^{\theta_i - \theta_{m^*}} o_{i j} }
   { \sum_{k \in J \setminus I} e^{\theta_k - \theta_{m^*}} }
   -
   \frac{ \sum_{i \in J} e^{\theta_i - \theta_m} o_{i j} }
   { \sum_{k \in J} e^{\theta_k - \theta_m} }
   \right]
$$
in which overflow cannot occur.

\subsection{Linearity by computational geometry}
\label{sec:linearity-example-iv-rcdd}

Calculate linearity using R package \texttt{rcdd} like in Section~\ref{sec:linearity-example-i-rcdd} above. We follow Section~4.2 of \citep{geyer-tech}.

\begin{knitrout}
\definecolor{shadecolor}{rgb}{0.969, 0.969, 0.969}\color{fgcolor}\begin{kframe}
\begin{alltt}
\hlstd{tanv} \hlkwb{<-} \hlstd{gout}\hlopt{$}\hlstd{modmat}
\hlstd{vrep} \hlkwb{<-} \hlkwd{cbind}\hlstd{(}\hlnum{0}\hlstd{,} \hlnum{0}\hlstd{, tanv)}
\hlstd{vrep[dat}\hlopt{$}\hlstd{y} \hlopt{>} \hlnum{0}\hlstd{,} \hlnum{1}\hlstd{]} \hlkwb{<-} \hlnum{1}
\hlkwd{system.time}\hlstd{(lout} \hlkwb{<-} \hlkwd{linearity}\hlstd{(}\hlkwd{d2q}\hlstd{(vrep),} \hlkwc{rep} \hlstd{=} \hlstr{"V"}\hlstd{))}
\end{alltt}
\begin{verbatim}
##    user  system elapsed 
##   4.334   0.003   4.337
\end{verbatim}
\begin{alltt}
\hlstd{linearity.too} \hlkwb{<-} \hlstd{dat}\hlopt{$}\hlstd{y} \hlopt{>} \hlnum{0}
\hlstd{linearity.too[lout]} \hlkwb{<-} \hlnum{TRUE}
\hlkwd{identical}\hlstd{(}\hlkwd{as.vector}\hlstd{(linearity), linearity.too)}
\end{alltt}
\begin{verbatim}
## [1] TRUE
\end{verbatim}
\end{kframe}
\end{knitrout}

So this agrees with our analysis in Section~\ref{sec:example-iv-linearity} above, except that the repeated linear programming implementation is slower than the implementation developed here.

\subsection{Generic direction of recession}
\label{sec:gdor-example-iv-rcdd}

We calculate a GDOR using R package \texttt{rcdd} as in Section~\ref{sec:gdor-example-i-rcdd} above. More specifically, we follow Section~4.2 of \citep{geyer-tech}, so we necessarily agree with the GDOR given in Table~1 of \citep{geyer-gdor}.

\begin{knitrout}
\definecolor{shadecolor}{rgb}{0.969, 0.969, 0.969}\color{fgcolor}\begin{kframe}
\begin{alltt}
\hlstd{modmat} \hlkwb{<-} \hlstd{gout}\hlopt{$}\hlstd{modmat}
\hlstd{hrep} \hlkwb{<-} \hlkwd{cbind}\hlstd{(}\hlnum{0}\hlstd{,} \hlnum{0}\hlstd{,} \hlopt{-}\hlstd{tanv,} \hlnum{0}\hlstd{)}
\hlstd{hrep[}\hlopt{!} \hlstd{linearity,} \hlkwd{ncol}\hlstd{(hrep)]} \hlkwb{<-} \hlstd{(}\hlopt{-}\hlnum{1}\hlstd{)}
\hlstd{hrep[linearity,} \hlnum{1}\hlstd{]} \hlkwb{<-} \hlnum{1}
\hlstd{hrep} \hlkwb{<-} \hlkwd{rbind}\hlstd{(hrep,} \hlkwd{c}\hlstd{(}\hlnum{0}\hlstd{,} \hlnum{1}\hlstd{,} \hlkwd{rep}\hlstd{(}\hlnum{0}\hlstd{,} \hlkwd{ncol}\hlstd{(gout}\hlopt{$}\hlstd{modmat)),} \hlopt{-}\hlnum{1}\hlstd{))}
\hlstd{objv} \hlkwb{<-} \hlkwd{c}\hlstd{(}\hlkwd{rep}\hlstd{(}\hlnum{0}\hlstd{,} \hlkwd{ncol}\hlstd{(gout}\hlopt{$}\hlstd{modmat)),} \hlnum{1}\hlstd{)}
\hlstd{pout} \hlkwb{<-} \hlkwd{lpcdd}\hlstd{(}\hlkwd{d2q}\hlstd{(hrep),} \hlkwd{d2q}\hlstd{(objv),} \hlkwc{minimize} \hlstd{=} \hlnum{FALSE}\hlstd{)}
\hlstd{gdor} \hlkwb{<-} \hlstd{pout}\hlopt{$}\hlstd{primal.solution[}\hlopt{-}\hlkwd{length}\hlstd{(pout}\hlopt{$}\hlstd{primal.solution)]}

\hlstd{foo} \hlkwb{<-} \hlstd{gdor}
\hlkwd{names}\hlstd{(foo)} \hlkwb{<-} \hlkwd{colnames}\hlstd{(modmat)}
\hlkwd{cbind}\hlstd{(foo[foo} \hlopt{!=} \hlstr{"0"}\hlstd{])}
\end{alltt}
\begin{verbatim}
##             [,1]
## (Intercept) "-1"
## v1          "1" 
## v2          "1" 
## v3          "1" 
## v5          "1" 
## v1:v2       "-1"
## v1:v3       "-1"
## v1:v5       "-1"
## v2:v3       "-1"
## v2:v5       "-1"
## v3:v5       "-1"
## v1:v2:v3    "1" 
## v1:v3:v5    "1" 
## v2:v3:v5    "1"
\end{verbatim}
\end{kframe}
\end{knitrout}

This agrees with Table~1 in \citep{geyer-gdor}.
Clean R global environment.

\begin{knitrout}
\definecolor{shadecolor}{rgb}{0.969, 0.969, 0.969}\color{fgcolor}\begin{kframe}
\begin{alltt}
\hlkwd{rm}\hlstd{(}\hlkwc{list} \hlstd{=} \hlkwd{ls}\hlstd{())}
\end{alltt}
\end{kframe}
\end{knitrout}

\section{A big data example}
\label{sec:big-data}

\subsection{Data}

Load the data.
\begin{knitrout}
\definecolor{shadecolor}{rgb}{0.969, 0.969, 0.969}\color{fgcolor}\begin{kframe}
\begin{alltt}
\hlstd{foo} \hlkwb{<-} \hlstr{"https://conservancy.umn.edu/bitstream/handle/11299/197369/bigcategorical.txt"}
\hlstd{bar} \hlkwb{<-} \hlkwd{sub}\hlstd{(}\hlstr{"^.*/"}\hlstd{,} \hlstr{""}\hlstd{, foo)}
\hlkwa{if} \hlstd{(}\hlopt{!} \hlkwd{file.exists}\hlstd{(bar))}
    \hlkwd{download.file}\hlstd{(foo, bar)}
\hlstd{dat} \hlkwb{<-} \hlkwd{read.table}\hlstd{(bar,} \hlkwc{header} \hlstd{=} \hlnum{TRUE}\hlstd{,} \hlkwc{stringsAsFactors} \hlstd{=} \hlnum{TRUE}\hlstd{)}
\hlkwd{dim}\hlstd{(dat)}
\end{alltt}
\begin{verbatim}
## [1] 1024    6
\end{verbatim}
\begin{alltt}
\hlkwd{names}\hlstd{(dat)}
\end{alltt}
\begin{verbatim}
## [1] "x1" "x2" "x3" "x4" "x5" "y"
\end{verbatim}
\end{kframe}
\end{knitrout}

The response vector is \texttt{y}, the predictors \texttt{x1} through \texttt{x5} are all categorical.  The components of \texttt{y} are all counts, so this is a categorical data analysis.  This contingency table has 1024 cells and the multinomial sample size (sum of cell counts) is 1055. These data are included in the \texttt{glmdr} package.

\begin{knitrout}
\definecolor{shadecolor}{rgb}{0.969, 0.969, 0.969}\color{fgcolor}\begin{kframe}
\begin{alltt}
\hlkwd{data}\hlstd{(bigcategorical)}
\hlkwd{all.equal}\hlstd{(dat, bigcategorical)}
\end{alltt}
\begin{verbatim}
## [1] TRUE
\end{verbatim}
\end{kframe}
\end{knitrout}

\subsection{Hypothesis Tests}

As in Section~\ref{sec:geyer-categorical} above, we assume a Poisson sampling model rather than a multinomial sampling model for the reasons stated in that section. Actually, as is well known \citep[Section~8.6.7]{agresti}, the MLE for the mean value parameter vector and the asymptotic chi-square distribution of test statistics is the same for Poisson, multinomial, and product multinomial sampling.  So nothing in this section depends on the sampling model.

\begin{knitrout}
\definecolor{shadecolor}{rgb}{0.969, 0.969, 0.969}\color{fgcolor}\begin{kframe}
\begin{alltt}
\hlstd{out1} \hlkwb{<-} \hlkwd{glm}\hlstd{(y} \hlopt{~} \hlnum{0} \hlopt{+} \hlstd{.,}
    \hlkwc{family} \hlstd{= poisson,} \hlkwc{data} \hlstd{= dat,} \hlkwc{x} \hlstd{=} \hlnum{TRUE}\hlstd{,}
    \hlkwc{control} \hlstd{=} \hlkwd{list}\hlstd{(}\hlkwc{maxit} \hlstd{=} \hlnum{1e3}\hlstd{,} \hlkwc{epsilon} \hlstd{=} \hlnum{1e-12}\hlstd{))}
\hlstd{out2} \hlkwb{<-} \hlkwd{glm}\hlstd{(y} \hlopt{~} \hlnum{0} \hlopt{+} \hlstd{(.)}\hlopt{^}\hlnum{2}\hlstd{,}
    \hlkwc{family} \hlstd{= poisson,} \hlkwc{data} \hlstd{= dat,} \hlkwc{x} \hlstd{=} \hlnum{TRUE}\hlstd{,}
    \hlkwc{control} \hlstd{=} \hlkwd{list}\hlstd{(}\hlkwc{maxit} \hlstd{=} \hlnum{1e3}\hlstd{,} \hlkwc{epsilon} \hlstd{=} \hlnum{1e-12}\hlstd{))}
\hlstd{out3} \hlkwb{<-} \hlkwd{glm}\hlstd{(y} \hlopt{~} \hlnum{0} \hlopt{+} \hlstd{(.)}\hlopt{^}\hlnum{3}\hlstd{,}
    \hlkwc{family} \hlstd{= poisson,} \hlkwc{data} \hlstd{= dat,} \hlkwc{x} \hlstd{=} \hlnum{TRUE}\hlstd{,}
    \hlkwc{control} \hlstd{=} \hlkwd{list}\hlstd{(}\hlkwc{maxit} \hlstd{=} \hlnum{1e3}\hlstd{,} \hlkwc{epsilon} \hlstd{=} \hlnum{1e-12}\hlstd{))}
\hlstd{out4} \hlkwb{<-} \hlkwd{glm}\hlstd{(y} \hlopt{~} \hlnum{0} \hlopt{+} \hlstd{(.)}\hlopt{^}\hlnum{4}\hlstd{,}
    \hlkwc{family} \hlstd{= poisson,} \hlkwc{data} \hlstd{= dat,} \hlkwc{x} \hlstd{=} \hlnum{TRUE}\hlstd{,}
    \hlkwc{control} \hlstd{=} \hlkwd{list}\hlstd{(}\hlkwc{maxit} \hlstd{=} \hlnum{1e3}\hlstd{,} \hlkwc{epsilon} \hlstd{=} \hlnum{1e-12}\hlstd{))}
\end{alltt}

{\ttfamily\noindent\color{warningcolor}{\#\# Warning: glm.fit: fitted rates numerically 0 occurred}}\begin{alltt}
\hlkwd{anova}\hlstd{(out1, out2, out3, out4,} \hlkwc{test} \hlstd{=} \hlstr{"Chisq"}\hlstd{)}
\end{alltt}
\begin{verbatim}
## Analysis of Deviance Table
## 
## Model 1: y ~ 0 + (x1 + x2 + x3 + x4 + x5)
## Model 2: y ~ 0 + (x1 + x2 + x3 + x4 + x5)^2
## Model 3: y ~ 0 + (x1 + x2 + x3 + x4 + x5)^3
## Model 4: y ~ 0 + (x1 + x2 + x3 + x4 + x5)^4
##   Resid. Df Resid. Dev  Df Deviance  Pr(>Chi)    
## 1      1008    1182.17                           
## 2       918    1076.55  90   105.62    0.1247    
## 3       648     811.73 270   264.83    0.5774    
## 4       243     277.37 405   534.36 1.605e-05 ***
## ---
## Signif. codes:  
## 0 '***' 0.001 '**' 0.01 '*' 0.05 '.' 0.1 ' ' 1
\end{verbatim}
\end{kframe}
\end{knitrout}

Despite the warning from R function \texttt{glm}, all of these hypothesis tests are valid because in none of them is \texttt{gout4} the null hypothesis. Tests done as in Table 2 in the main text.

\begin{knitrout}
\definecolor{shadecolor}{rgb}{0.969, 0.969, 0.969}\color{fgcolor}\begin{kframe}
\begin{alltt}
\hlkwd{anova}\hlstd{(out1, out4,} \hlkwc{test} \hlstd{=} \hlstr{"Chisq"}\hlstd{)}
\end{alltt}
\begin{verbatim}
## Analysis of Deviance Table
## 
## Model 1: y ~ 0 + (x1 + x2 + x3 + x4 + x5)
## Model 2: y ~ 0 + (x1 + x2 + x3 + x4 + x5)^4
##   Resid. Df Resid. Dev  Df Deviance  Pr(>Chi)    
## 1      1008    1182.17                           
## 2       243     277.37 765    904.8 0.0003447 ***
## ---
## Signif. codes:  
## 0 '***' 0.001 '**' 0.01 '*' 0.05 '.' 0.1 ' ' 1
\end{verbatim}
\begin{alltt}
\hlkwd{anova}\hlstd{(out2, out4,} \hlkwc{test} \hlstd{=} \hlstr{"Chisq"}\hlstd{)}
\end{alltt}
\begin{verbatim}
## Analysis of Deviance Table
## 
## Model 1: y ~ 0 + (x1 + x2 + x3 + x4 + x5)^2
## Model 2: y ~ 0 + (x1 + x2 + x3 + x4 + x5)^4
##   Resid. Df Resid. Dev  Df Deviance  Pr(>Chi)    
## 1       918    1076.55                           
## 2       243     277.37 675   799.18 0.0006633 ***
## ---
## Signif. codes:  
## 0 '***' 0.001 '**' 0.01 '*' 0.05 '.' 0.1 ' ' 1
\end{verbatim}
\begin{alltt}
\hlkwd{anova}\hlstd{(out3, out4,} \hlkwc{test} \hlstd{=} \hlstr{"Chisq"}\hlstd{)}
\end{alltt}
\begin{verbatim}
## Analysis of Deviance Table
## 
## Model 1: y ~ 0 + (x1 + x2 + x3 + x4 + x5)^3
## Model 2: y ~ 0 + (x1 + x2 + x3 + x4 + x5)^4
##   Resid. Df Resid. Dev  Df Deviance  Pr(>Chi)    
## 1       648     811.73                           
## 2       243     277.37 405   534.36 1.605e-05 ***
## ---
## Signif. codes:  
## 0 '***' 0.001 '**' 0.01 '*' 0.05 '.' 0.1 ' ' 1
\end{verbatim}
\end{kframe}
\end{knitrout}

These agree with Table~2 in the main text.

\subsection{Maximizing the likelihood}

We fit these data using R function \texttt{glmdr}.

\begin{knitrout}
\definecolor{shadecolor}{rgb}{0.969, 0.969, 0.969}\color{fgcolor}\begin{kframe}
\begin{alltt}
\hlstd{gout} \hlkwb{<-} \hlkwd{glmdr}\hlstd{(y} \hlopt{~} \hlnum{0} \hlopt{+} \hlstd{(.)}\hlopt{^}\hlnum{4}\hlstd{,} \hlkwc{family} \hlstd{=} \hlstr{"poisson"}\hlstd{,}
  \hlkwc{data} \hlstd{= bigcategorical)}
\end{alltt}
\end{kframe}
\end{knitrout}

\subsection{Linearity}
We then find the linearity as in preceding sections.


\begin{knitrout}
\definecolor{shadecolor}{rgb}{0.969, 0.969, 0.969}\color{fgcolor}\begin{kframe}
\begin{alltt}
\hlstd{linearity} \hlkwb{<-} \hlstd{gout}\hlopt{$}\hlstd{linearity}
\hlkwd{sum}\hlstd{(linearity)}
\end{alltt}
\begin{verbatim}
## [1] 942
\end{verbatim}
\begin{alltt}
\hlkwd{sum}\hlstd{(}\hlopt{!} \hlstd{linearity)}
\end{alltt}
\begin{verbatim}
## [1] 82
\end{verbatim}
\end{kframe}
\end{knitrout}

\subsection{One-sided confidence intervals: Poisson sampling}

We now provide one-sided confidence intervals for mean value parameters whose MLE is on the boundary as done before.  This is the full table referenced in the main text.

\begin{knitrout}
\definecolor{shadecolor}{rgb}{0.969, 0.969, 0.969}\color{fgcolor}\begin{kframe}
\begin{alltt}
\hlkwd{system.time}\hlstd{(mus.CI} \hlkwb{<-} \hlkwd{inference}\hlstd{(gout))}
\end{alltt}
\begin{verbatim}
##    user  system elapsed 
##  57.416 161.807  60.102
\end{verbatim}
\begin{alltt}
\hlstd{upper} \hlkwb{<-} \hlkwd{round}\hlstd{(mus.CI[,} \hlkwd{ncol}\hlstd{(mus.CI)],} \hlnum{4}\hlstd{)}
\hlstd{tab} \hlkwb{<-} \hlkwd{cbind}\hlstd{(dat[}\hlopt{!}\hlstd{linearity, ], upper)}
\hlstd{tab}
\end{alltt}
\begin{verbatim}
##      x1 x2 x3 x4 x5 y  upper
## 17    a  a  b  a  a 0 0.1695
## 21    a  b  b  a  a 0 0.1354
## 25    a  c  b  a  a 0 0.2292
## 29    a  d  b  a  a 0 2.4616
## 48    d  d  c  a  a 0 0.0002
## 57    a  c  d  a  a 0 0.0133
## 58    b  c  d  a  a 0 0.5647
## 59    c  c  d  a  a 0 0.2790
## 60    d  c  d  a  a 0 2.1519
## 105   a  c  c  b  a 0 0.1060
## 106   b  c  c  b  a 0 0.6088
## 107   c  c  c  b  a 0 2.2809
## 108   d  c  c  b  a 0 1.4718
## 112   d  d  c  b  a 0 2.9921
## 121   a  c  d  b  a 0 0.0167
## 176   d  d  c  c  a 0 0.0008
## 183   c  b  d  c  a 0 0.0103
## 185   a  c  d  c  a 0 2.9448
## 222   b  d  b  d  a 0 0.0607
## 240   d  d  c  d  a 0 0.0027
## 249   a  c  d  d  a 0 0.0509
## 285   a  d  b  a  b 0 1.8519
## 286   b  d  b  a  b 0 0.0995
## 287   c  d  b  a  b 0 0.0027
## 288   d  d  b  a  b 0 1.1411
## 297   a  c  c  a  b 0 2.8903
## 301   a  d  c  a  b 0 2.2144
## 350   b  d  b  b  b 0 2.9408
## 361   a  c  c  b  b 0 0.0850
## 364   d  c  c  b  b 0 0.0154
## 365   a  d  c  b  b 0 0.6450
## 377   a  c  d  b  b 0 2.8350
## 397   a  d  a  c  b 0 2.9956
## 413   a  d  b  c  b 0 0.0001
## 414   b  d  b  c  b 0 0.0549
## 417   a  a  c  c  b 0 0.5509
## 421   a  b  c  c  b 0 2.4449
## 425   a  c  c  c  b 0 0.0860
## 429   a  d  c  c  b 0 0.0004
## 439   c  b  d  c  b 0 2.0680
## 445   a  d  d  c  b 0 0.0000
## 478   b  d  b  d  b 0 0.2229
## 489   a  c  c  d  b 0 0.0204
## 493   a  d  c  d  b 0 0.1364
## 505   a  c  d  d  b 0 1.2175
## 506   b  c  d  d  b 0 0.2057
## 507   c  c  d  d  b 0 1.5164
## 508   d  c  d  d  b 0 0.0560
## 517   a  b  a  a  c 0 1.1298
## 518   b  b  a  a  c 0 1.0333
## 519   c  b  a  a  c 0 0.1836
## 520   d  b  a  a  c 0 0.6489
## 525   a  d  a  a  c 0 0.0570
## 541   a  d  b  a  c 0 2.8136
## 557   a  d  c  a  c 0 0.0098
## 573   a  d  d  a  c 0 0.1153
## 588   d  c  a  b  c 0 2.4637
## 604   d  c  b  b  c 0 0.2193
## 620   d  c  c  b  c 0 0.0064
## 633   a  c  d  b  c 0 0.1588
## 636   d  c  d  b  c 0 0.3127
## 695   c  b  d  c  c 0 0.4586
## 734   b  d  b  d  c 0 0.0004
## 793   a  c  b  a  d 0 2.6538
## 834   b  a  a  b  d 0 0.0049
## 850   b  a  b  b  d 0 0.0008
## 857   a  c  b  b  d 0 0.0392
## 866   b  a  c  b  d 0 0.0019
## 876   d  c  c  b  d 0 2.9803
## 882   b  a  d  b  d 0 2.9881
## 889   a  c  d  b  d 0 0.0018
## 921   a  c  b  c  d 0 0.0484
## 951   c  b  d  c  d 0 0.4589
## 965   a  b  a  d  d 0 0.2902
## 981   a  b  b  d  d 0 1.9221
## 985   a  c  b  d  d 0 0.2544
## 990   b  d  b  d  d 0 2.9346
## 997   a  b  c  d  d 0 0.7834
## 1009  a  a  d  d  d 0 2.9673
## 1013  a  b  d  d  d 0 0.0169
## 1017  a  c  d  d  d 0 0.0211
## 1021  a  d  d  d  d 0 0.0073
\end{verbatim}
\end{kframe}
\end{knitrout}

Table~3 in the main text is provided below

\begin{knitrout}
\definecolor{shadecolor}{rgb}{0.969, 0.969, 0.969}\color{fgcolor}\begin{kframe}
\begin{alltt}
\hlkwd{head}\hlstd{(tab)}
\end{alltt}
\begin{verbatim}
##    x1 x2 x3 x4 x5 y  upper
## 17  a  a  b  a  a 0 0.1695
## 21  a  b  b  a  a 0 0.1354
## 25  a  c  b  a  a 0 0.2292
## 29  a  d  b  a  a 0 2.4616
## 48  d  d  c  a  a 0 0.0002
## 57  a  c  d  a  a 0 0.0133
\end{verbatim}
\end{kframe}
\end{knitrout}

This also agrees with the first version of this document, which took several hours to do this job (and we are taking only a few seconds). 

\subsection{Linearity by computational geometry}
\label{sec:linearity-big-data}

Calculate linearity using R package \texttt{rcdd} like
in Section~\ref{sec:linearity-example-iv-rcdd} above.
Except that we are going to cache the result and the time it takes
to compute it.  Rather than using the cache feature of R package
\texttt{knitr}, which should not be committed under version control,
we cache it ourselves.
\begin{knitrout}
\definecolor{shadecolor}{rgb}{0.969, 0.969, 0.969}\color{fgcolor}\begin{kframe}
\begin{alltt}
\hlstd{tanv} \hlkwb{<-} \hlstd{modmat} \hlkwb{<-} \hlstd{out4}\hlopt{$}\hlstd{x}
\hlstd{vrep} \hlkwb{<-} \hlkwd{cbind}\hlstd{(}\hlnum{0}\hlstd{,} \hlnum{0}\hlstd{, tanv)}
\hlstd{vrep[dat}\hlopt{$}\hlstd{y} \hlopt{>} \hlnum{0}\hlstd{,} \hlnum{1}\hlstd{]} \hlkwb{<-} \hlnum{1}
\hlkwd{suppressWarnings}\hlstd{(foo} \hlkwb{<-} \hlkwd{try}\hlstd{(}\hlkwd{load}\hlstd{(}\hlstr{"foo-linearity.rda"}\hlstd{),} \hlkwc{silent} \hlstd{=} \hlnum{TRUE}\hlstd{))}
\hlkwa{if} \hlstd{(}\hlkwd{inherits}\hlstd{(foo,} \hlstr{"try-error"}\hlstd{)) \{}
    \hlstd{time.linearity.big.data} \hlkwb{<-} \hlkwd{system.time}\hlstd{(}
        \hlstd{lout} \hlkwb{<-} \hlkwd{linearity}\hlstd{(}\hlkwd{d2q}\hlstd{(vrep),} \hlkwc{rep} \hlstd{=} \hlstr{"V"}\hlstd{)}
    \hlstd{)}
    \hlstd{hostname.linearity.big.data} \hlkwb{<-} \hlkwa{NULL}
    \hlstd{cpuinfo.linearity.big.data} \hlkwb{<-} \hlkwa{NULL}
    \hlkwa{if} \hlstd{(}\hlkwd{Sys.info}\hlstd{()[}\hlstr{"sysname"}\hlstd{]} \hlopt{==} \hlstr{"Linux"}\hlstd{) \{}
        \hlstd{foo} \hlkwb{<-} \hlkwd{scan}\hlstd{(}\hlstr{"/proc/cpuinfo"}\hlstd{,} \hlkwc{what} \hlstd{=} \hlkwd{character}\hlstd{(}\hlnum{0}\hlstd{),} \hlkwc{sep} \hlstd{=} \hlstr{"\textbackslash{}n"}\hlstd{)}
        \hlstd{bar} \hlkwb{<-} \hlkwd{grep}\hlstd{(}\hlstr{"^model name"}\hlstd{, foo,} \hlkwc{value} \hlstd{=} \hlnum{TRUE}\hlstd{)}
        \hlstd{bar} \hlkwb{<-} \hlkwd{unique}\hlstd{(bar)}
        \hlstd{baz} \hlkwb{<-} \hlkwd{sub}\hlstd{(}\hlstr{"^model name\textbackslash{}\textbackslash{}t: "}\hlstd{,} \hlstr{""}\hlstd{, bar)}
        \hlstd{qux} \hlkwb{<-} \hlkwd{system}\hlstd{(}\hlstr{"nslookup `hostname`"}\hlstd{,} \hlkwc{intern} \hlstd{=} \hlnum{TRUE}\hlstd{)}
        \hlstd{quux} \hlkwb{<-} \hlkwd{grep}\hlstd{(}\hlstr{"^Name:"}\hlstd{, qux,} \hlkwc{value} \hlstd{=} \hlnum{TRUE}\hlstd{)}
        \hlstd{quuux} \hlkwb{<-} \hlkwd{sub}\hlstd{(}\hlstr{"^Name:\textbackslash{}\textbackslash{}t"}\hlstd{,} \hlstr{""}\hlstd{, quux)}
        \hlstd{quacks} \hlkwb{<-} \hlkwd{unique}\hlstd{(quuux)}
        \hlstd{hostname.linearity.big.data} \hlkwb{<-} \hlstd{quacks[}\hlnum{1}\hlstd{]}
        \hlstd{cpuinfo.linearity.big.data} \hlkwb{<-} \hlstd{baz}
    \hlstd{\}}
    \hlkwd{save}\hlstd{(lout, time.linearity.big.data,}
        \hlstd{hostname.linearity.big.data, cpuinfo.linearity.big.data,}
        \hlkwc{file} \hlstd{=} \hlstr{"foo-linearity.rda"}\hlstd{)}
\hlstd{\}}
\hlstd{linearity.too} \hlkwb{<-} \hlstd{dat}\hlopt{$}\hlstd{y} \hlopt{>} \hlnum{0}
\hlstd{linearity.too[lout]} \hlkwb{<-} \hlnum{TRUE}
\hlkwd{identical}\hlstd{(}\hlkwd{as.vector}\hlstd{(linearity), linearity.too)}
\end{alltt}
\begin{verbatim}
## [1] TRUE
\end{verbatim}
\end{kframe}
\end{knitrout}


\bibliographystyle{plainnat}
\bibliography{conjure}

\end{document}